\pdfoutput=1

\documentclass[a4paper,12pt,twoside,onecolumn,final]{article}
\usepackage{amsmath,amsfonts,amssymb,amscd,mathrsfs}
\usepackage{extarrows}
\usepackage{epsfig}
\usepackage{enumerate}
\usepackage{color}
\usepackage{graphicx}
\usepackage{psfrag}
\usepackage[nodayofweek]{datetime}
\usepackage[paper=a4paper,hscale=0.7,vscale=0.765,centering,heightrounded]{geometry}
\usepackage[ntheorem]{empheq} % this loads amsmath as well 
\empheqset{box=\bigfbox} 
\usepackage[thmmarks,amsmath,hyperref]{ntheorem} 
\theoremstyle{plain} 
\theoremheaderfont{\normalfont\bfseries} 
\theorembodyfont{\slshape}
\theoremsymbol{\ensuremath{_\Box}}
\theoremnumbering{Alph}
\newtheorem{theorem}{Theorem}[section]

\theoremnumbering{arabic}
\newtheorem{proposition}{Proposition}[section]
\newtheorem{lemma}[proposition]{Lemma}
\newtheorem{corollary}[proposition]{Corollary}

\theoremstyle{plain}
\theoremheaderfont{\normalfont\bfseries} 
\theorembodyfont{\upshape}
\theoremsymbol{\ensuremath{_\Box}}
\newtheorem{definition}[proposition]{Definition}  % [chapter]
\newtheorem{remark}[proposition]{Remark}  % [chapter]
\newtheorem{example}[proposition]{Example} % [chapter]
\newtheorem{assumption}[proposition]{Assumption}  % [chapter]
\theoremstyle{nonumberplain}
\theoremheaderfont{\normalfont\bfseries} 
\theorembodyfont{\upshape}
\theoremsymbol{\ensuremath{_\blacksquare}} % _\blacksquare % _Box
\newtheorem{proof}{Proof}
\theoremsymbol{\ensuremath{_\blacksquare}}

\theoremlisttype{allname}
\allowdisplaybreaks[3] 
\usepackage[colorlinks=true, pdfstartview=FitV, linkcolor=redgray, citecolor=bluegray, urlcolor=black, pdfpagelabels, naturalnames=true, draft=false]{hyperref}
\usepackage[small,bf]{caption} 
\usepackage[notref,notcite,draft]{showkeys}

\definecolor{refkey}{rgb}{0,1,0} 
\definecolor{labelkey}{rgb}{0,1,1}
  \renewenvironment{thebibliography}[1]{%
    \begin{oldthebibliography}{#1}%
       \addtolength{\itemsep}{-0.5ex}%
  }%
  {%
    \end{oldthebibliography}%
  }
\graphicspath{{Figures/}{./}}

\usepackage{xfrac}

\usepackage{psfrag}
\usepackage[off]{auto-pst-pdf}

\usepackage{tikz}
\usetikzlibrary{shapes.multipart}
\usepackage{contour}

\DeclareMathAlphabet{\mathpzc}{OT1}{pzc}{m}{it}

\usepackage{bbm}
\usepackage[ulem=normalem,deletedmarkup=sout]{changes}
\newcommand\redout{\bgroup\markoverwith
{\textcolor{red}{\rule[.5ex]{2pt}{0.75pt}}}\ULon}

\setdeletedmarkup{\textcolor{red}{\redout{#1}}}

\newcommand\soutmath[2][red]{\setbox0=\hbox{$\,#2\,$}%
\rlap{\raisebox{.325\ht0}{\textcolor{#1}{\rule{\wd0}{.75pt}}}}\textcolor{red}{\,#2\,}} 

\newcommand\souteq[2][red]{\setbox0=\hbox{$#2$}%
$$\rlap{\raisebox{.325\ht0}{\textcolor{#1}{\rule{\wd0}{.75pt}}}}\textcolor{red}{\text{$#2$}}$$}

\usepackage{ifdraft}
\ifdraft{\newcommand{\removedmath}[1]{\soutmath{#1}}}{\newcommand{\removedmath}[1]{{}}}
\ifdraft{\newcommand{\removedeq}[1]{\souteq{#1}}}{\newcommand{\removedeq}[1]{{}}}
\ifdraft{\newcommand{\removed}[1]{\deleted{#1}}}{\newcommand{\removed}[1]{}}

\ifdraft{}{}
\ifdraft{}{}

\usepackage{marginnote}

\numberwithin{equation}{section}

\newcommand{\mbbE}{\mathbb{E}}

\newcommand{\mbbM}{\mathbb{M}}
\newcommand{\mbbN}{\mathbb{N}}

\newcommand{\mbbR}{\mathbb{R}}

\newcommand{\mbbU}{\mathbb{U}}
\newcommand{\mbbV}{\mathbb{V}}
\newcommand{\mbbW}{\mathbb{W}}
\newcommand{\mbbX}{\mathbb{X}}
\newcommand{\mbbY}{\mathbb{Y}}
\newcommand{\mbbZ}{\mathbb{Z}}

\newcommand{\Image}{\operatorname{Im}}
\newcommand{\Kernel}{\operatorname{Ker}}
\newcommand{\osc}{\operatorname{osc}}

\renewcommand{\ker}{\Kernel}

\newcommand{\CAO}{C_{\text{\tiny$\mathrm{AO}$}}}
\newcommand{\CBM}{C_{\text{\tiny$\mathrm{BM}$}}}

\newcommand{\dBM}{d_{\text{\tiny$\mathrm{BM}$}}}

\newcommand{\sign}{\operatorname{sign}}

\definecolor{darkgray}{gray}{0.4}  
\definecolor{ddarkgray}{gray}{0.2} 
\definecolor{redgray}{rgb}{0.5,0.25,0.25}  
\definecolor{bluegray}{rgb}{0.25,0.25,0.5}  
\renewcommand*{\thefootnote}{\fnsymbol{footnote}}
\title{\Large \textsc{%
Discretization of Linear Problems in \\
Banach Spaces: Residual Minimization, \\
Nonlinear Petrov--Galerkin, and \\
Monotone Mixed Methods%
}%
}
\author{\large \textbf{I.~Muga\footnotemark[2]\, $\boldsymbol{\,\cdot\,}$ K.G. van der Zee\footnotemark[3]}
\\
{\footnotesize {5$^\mathrm{th}$ June, 2018}}}
\date{\vspace*{-1.75\baselineskip}}
\newlength{\bigfboxsep}
\newcommand{\bigfbox}[1]{\setlength{\fboxsep}{\bigfboxsep}\fbox{#1}}

\newcommand{\negquad}{\mspace{-18.0mu}}
\newcommand{\negqquad}{\negquad\negquad}

\newcommand{\norm}[1]{{\|#1\|}}
\newcommand{\bignorm}[1]{\big\|#1\big\|}
\newcommand{\Bignorm}[1]{\Big\|#1\Big\|}

\newcommand{\innerprod}[1]{(#1)}

\newcommand{\dual}[1]{\langle#1\rangle}
\newcommand{\bigdual}[1]{\big\langle#1\big\rangle}

\newcommand{\realset}{\mathbb{R}} 

\newcommand{\nozero}{\setminus\{0\}}

\newcommand{\ds}[1]{{\displaystyle{#1}}}

\newcommand{\mbb}[1]{\mathbb{#1}}

\newcommand{\mcal}[1]{\mathcal{#1}}

\newcommand{\mcJ}{\mcal{J}}

\newcommand{\mcL}{\mcal{L}}

\newcommand{\<}{\left<}
\renewcommand{\>}{\right>}
\newcommand{\rr}{\realset}

\begin{document}
%-----------------------------------------------------------------------------%
%
%=============================================================================%
% Title
%=============================================================================%
%
%-----------------------------------------------------------------------------%
\maketitle
%-----------------------------------------------------------------------------%
%
%=============================================================================%
% Author footnotes
%=============================================================================%
%
%-----------------------------------------------------------------------------%
\footnotetext[2]{Pontificia Univ.~Cat\'olica de Valpara\'iso, Instituto de Mathem\'aticas, \mbox{ignacio.muga@pucv.cl}}
\footnotetext[3]{University of Nottingham, School of Mathematical Sciences, \mbox{kg.vanderzee@nottingham.ac.uk}}
\renewcommand{\thefootnote}{\arabic{footnote}}
%-----------------------------------------------------------------------------%
%
%=============================================================================%
% Abstract
%=============================================================================%
%
%-----------------------------------------------------------------------------%
\begin{abstract}
\noindent
% Abstract
% 
This work presents a comprehensive discretization theory for abstract linear operator equations in \emph{Banach} spaces. The fundamental starting point of the theory is the idea of residual minimization in dual norms, and its inexact version using \emph{discrete} dual norms. It is shown that this development, in the case of strictly-convex reflexive Banach spaces with strictly-convex dual, gives rise to a class of \emph{nonlinear} Petrov--Galerkin methods and, equivalently, abstract mixed methods with \emph{monotone nonlinearity}.
% These methods are formulated by means of the (nonlinear) bijective \emph{duality map}. 
Crucial in the formulation of these methods is the (nonlinear) bijective \emph{duality map}. 
% of the Banach space. 
%
\par
Under the \emph{Fortin} condition, we prove discrete stability of the abstract inexact method, and subsequently carry out a complete error analysis. As part of our analysis, we prove new bounds for best-approximation projectors, which involve constants depending on the geometry of the underlying Banach space. 
% To illustrate the significance of the discretization framework, we consider its application to the numerical solution of linear partial differential equations, highlighting its ability to approximate irregular solutions. In particular, we focus on weak non-standard Banach-space settings for the advection--reaction equation and a second-order elliptic problem. For both cases, we prove the Fortin condition for suitable choices of finite element spaces, thereby achieving \emph{quasi-optimality} of the resulting discretizations.
%
The theory generalizes and extends the classical Petrov--Galerkin method as well as existing residual-minimization approaches, such as the discontinuous Petrov--Galerkin method.

\vspace*{.25\baselineskip}
\end{abstract}
{\footnotesize
\parbox[t]{\textwidth}{%
\textbf{Keywords~} 
Operators in Banach spaces 
$\cdot$~Residual minimization 
$\cdot$~Petrov--Galerkin discretization 
$\cdot$~Error analysis
$\cdot$~Quasi-optimality 
$\cdot$~Duality mapping 
$\cdot$~Best approximation 
$\cdot$~Geometric constants 
%$\cdot$~Advection--reaction 
%$\cdot$~Laplacian
\\[.5\baselineskip]
\textbf{Mathematics Subject Classification~}
41A65 %Approximations and expansions, Abstract approximation theory (approximation in normed linear spaces and other abstract spaces)
$\cdot$~65J05 %Numerical analysis in abstract spaces, General theory
$\cdot$~46B20 %Functional analysis, Geometry and structure of normed linear spaces
% 41A25 $\cdot$ %Approximations and expansions, Rate of convergence
$\cdot$~65N12 %Partial differential equations, boundary value problems, Stability and convergence of numerical methods
$\cdot$~65N15 %Partial differential equations, boundary value problems, Error bounds
%$\cdot$~35L02 %Partial differential equations, First-order hyperbolic equations
%$\cdot$~35J25 %Partial differential equations, Boundary value problems for second-order elliptic equations
}
}
%
%-----------------------------------------------------------------------------
\newpage
\setcounter{secnumdepth}{3}
\setcounter{tocdepth}{3}
{
\small
\tableofcontents
}
%-----------------------------------------------------------------------------%
%
%=============================================================================%
% Headings
%=============================================================================%
%
%-----------------------------------------------------------------------------%
\pagestyle{myheadings}
\thispagestyle{plain}
\markboth{\small \textit{I.~Muga and K.G.~van der Zee}}{\small \textit{Discretization in Banach Spaces}}
%-----------------------------------------------------------------------------%
%
%=============================================================================%
%\input{Schedule.tex}
% \newpage
\section{Introduction}
\label{sec:intro}
%
%-----------------------------------------------------------------------------%
%\subsection{Preliminaries}
%
In the setting of \emph{Banach} spaces, we consider the abstract problem
\begin{empheq}[left=\left\{\,,right=\right.,box=]{alignat=2}
\notag & \text{Find }  u\in \mbbU : 
\\
\label{eq:Bu=f}
&  \quad B u = f \qquad \text{in } \mbbV^*\,,
\end{empheq}
where $\mbbU$ and $\mbbV$ are infinite-dimensional Banach spaces and the data $f$~is a given element in the dual space~$\mbbV^*$. The operator $B:\mbbU\rightarrow \mbbV^*$ is a continuous, bounded-below linear operator, that is, there is a continuity constant~$M_B>0$ and bounded-below constant~$\gamma_B>0$ such that
\begin{alignat}{2}
\label{eq:normEquiv}
  \gamma_B \norm{w}_\mbbU \le \norm{B w}_{\mbbV^*} \le M_B \norm{w}_\mbbU
  \,, \qquad \forall w\in \mbbU \,.\footnotemark
\end{alignat}
\footnotetext{The smallest possible constant~$M_B$ coincides with~$\ds{\norm{B}:= \sup_{w\in \mbbU\setminus \{0\}} \norm{Bw}_{\mbbV^*} /\norm{w}_\mbbU}$, while the largest possible constant~$\gamma_B$ coincides with~$1/\norm{B^{-1}}$, where~$B^{-1}:\Image (B) \rightarrow \mbbU$\,.}%
%
%Existence of a unique solution to~\eqref{eq:Bu=f} is guaranteed if $f\in \Image B$ or if $B^*$ is injective by Banach's open mapping and closed-range theorem.
%
\par
Problem~\eqref{eq:Bu=f} is equivalent to the variational statement
\begin{alignat*}{2}
  \bigdual{ B u,v }_{\mbbV^*,\mbbV} = \bigdual{ f , v }_{\mbbV^*,\mbbV} \qquad \forall v\in \mbbV\,,
\end{alignat*}
commonly encountered in
%for example 
the weak formulation of partial differential equations (PDEs), i.e.,  when 
$\bigdual{ B u,v }_{\mbbV^*,\mbbV} := b(u,v)$ and $b:\mbbU\times \mbbV\to \rr$ is a bilinear form. Note that the above Banach-space setting allows the consideration of PDEs in non-standard (non-Hilbert) settings suitable for dealing with rough data (e.g., measure-valued sources) and irregular solutions (e.g., in $W^{1,p}$, $p\neq 2$, or in~$BV$).

\par
%

%
%\subsection{A central problem in numerical analysis}
%
A central problem in numerical analysis is \emph{to devise a discretization method that}, for a given family~$\{\mbbU_n\}_{n\in\mbbN}$ of discrete (finite-dimensional) subspaces $\mbbU_n\subset \mbbU$, 
% parametrized by~$n$, 
% the objective of this paper is to present a Galerkin-based discretization technique which 
\emph{is guaranteed to provide a near-best approximation}~$u_n\in \mbbU_n$ to the solution~$u\in\mbbU$ of~\eqref{eq:Bu=f}. This means that, for some constant~$C\ge 1$ independent of~$n$, the approximation~$u_n$ satisfies the error bound
\begin{alignat}{2}
  \label{eq:qOpt}
  \norm{u-u_n}_\mbbU \le C 
  % \operatorname{dist}\big( u, \mbbU_n\big) \,,
  \inf_{w_n\in \mbbU_n}
  \norm{ u -  w_n }_{\mbbU}\,.
\end{alignat}
%
% where the infimum on the right-hand side is the best-approximation error.
%hence the error is, up to the constant~$C$, bounded by the best-approximation error.
%where the distance is given by the best-approximation error:
%
%\begin{alignat*}{2}
%\operatorname{dist}\big( u, \mbbU_n\big)  :=  \inf_{w_n\in \mbbU_n}
%  \norm{ u -  w_n }_{\mbbU}\,.
%\end{alignat*}
%
% In other words, up to the constant~$C$, $u_n$~is nearly as good as a best approximation.
% ($=\arg \inf_{w_n\in \mbbU_n}\norm{u-w_n}$).
%In this case 
A discretization method for which this is true, is said to be \emph{quasi-optimal}. 
\par
It is the purpose of this paper to propose and analyze a new quasi-optimal discretization method for the problem in~\eqref{eq:Bu=f} that generalizes and improves upon existing methods.
%
%-----------------------------------------------------------------------------%
\subsection{Petrov--Galerkin discretization and residual minimization}
A standard method for~\eqref{eq:Bu=f} is the \emph{Petrov--Galerkin} discretization:
\begin{empheq}[left=\left\{\,,right=\right.,box=]{alignat=2}
\notag 
 & \text{Find }  u_n\in \mbbU_n : 
\\ \label{eq:PGdisc}
 & \quad
  \bigdual{Bu_n,v_n}_{\mbbV^*,\mbbV} = \bigdual{f,v_n}_{\mbbV^*,\mbbV}
  \qquad \forall v_n \in \mbbV_n\,,
\end{empheq}
with $\mbbV_n\subset \mbbV$ a discrete subspace of the same dimension as~$\mbbU_n$. It is well-known however, that the fundamental difficulty of~\eqref{eq:PGdisc} is \emph{stability}: 
%, however, in the design of a discretization method for the problem in~\eqref{eq:Bu=f} is \emph{stability}. 
%it is well-known that 
One must come up with a test space~$\mbbV_n$ that is \emph{precisely compatible} with~$\mbbU_n$ in the sense that they have the same dimension and the discrete inf--sup condition is satisfied; see, e.g.,~\cite[Chapter~2]{ErnGueBOOK2004} and~\cite[Chapter~10]{StaHolBOOK2011}. Incompatible pairs~$(\mbbU_n,\mbbV_n)$ may lead to spurious numerical artifacts and non-convergent approximations. 
% Lemma~2.28
%
%-----------------------------------------------------------------------------%
% \subsection{Residual minimization in dual Banach spaces}
%
\par
An alternative method, which is not common, is \emph{residual minimization}:
\begin{empheq}[left=\left\{\,,right=\right.,box=]{alignat=2}
\notag 
 & \text{Find }  u_n\in \mbbU_n : 
\\ \label{eq:introResMin}
 & \quad
  u_n = \arg \min_{w_n\in \mbbU_n} \bignorm{f-B w_n}_{\mbbV^*} \,,
\end{empheq}
where the dual norm is given by
\begin{alignat}{2}
\label{eq:dualNorm}
  \norm{ r }_{\mbbV^*} = \sup_{v\in \mbbV \nozero} \frac{\dual{r,v}_{\mbbV^*,\mbbV}}{\norm{v}_\mbbV}\,,
  \qquad \text{for any~} r\in \mbbV^*\,.
\end{alignat}
The residual-minimization method is appealing for its stability and quasi-optimality \emph{without} requiring additional conditions. This was proven by Guermond~\cite{GueSINUM2004}, who studied residual minimization abstractly in Banach spaces, and focussed on the case where the residual is in an $L^p$~space, for $1\le p < \infty$. 
%
% for well-posed was studied abstractly in Banach spaces by Guermond~\cite{GueSINUM2004}, who proved its well-posedness and quasi-optimality, and focussed on the case where the residual is in an $L^p$~space, for $1\le p < \infty$. 
%
% was studied abstractly in Banach spaces by Guermond~\cite{GueSINUM2004}, who proved its well-posedness and quasi-optimality, and focussed on the case where the residual is in an $L^p$~space, for $1\le p < \infty$. 
% The stability of the method essentially follows from its equivalence to the problem of best approximation in the norm~$\norm{B(\cdot)}_{\mbbV^*}$; see Section~\ref{}. 
If~$\mbbV$ is a Hilbert space, residual minimization corresponds to the familiar \emph{least-squares minimization} method~\cite{BocGunBOOK2009}; otherwise it requires the minimization of a convex (non-quadratic) functional. 
%
%-----------------------------------------------------------------------------%
\subsection{Residual minimization in discrete dual norms}
Although residual minimization is a quasi-optimal method, an essential complication is that the dual norm~\eqref{eq:dualNorm} may be \emph{non-computable} in practice, because it requires a supremum over~$\mbbV$ that may be intractable. This is the case, for example, for the Sobolev space~$\mbbV= W_0^{1,p}(\Omega)$, with~$\Omega\subset \mathbb{R}^d$ a bounded $d$-dimensional domain, for which the dual is the negative Sobolev space $\mbbV^* = [W_0^{1,p}(\Omega)]^* =: W^{-1,q}(\Omega)$ (see, e.g.,~\cite{AdaFouBOOK2003}), where $p^{-1} + q^{-1} = 1$. 
% Other examples are dual graph-norms (the norm of the dual of graph spaces),  
Situations with non-computable dual norms are very common in weak formulations of~PDEs and, therefore, such complications can not be neglected. 
\par
A natural replacement that makes such dual norms computationally-tractable is obtained by restricting the supremum to \emph{discrete} subspaces~$\mbbV_m\subset \mbbV$. This idea leads to the following \emph{inexact residual minimization}:
\begin{empheq}[left=\left\{\,,right=\right.,box=]{alignat=2}
\notag 
 & \text{Find }  u_n\in \mbbU_n : 
\\ \label{eq:introInexactResMin}
 & \quad
   u_n = \arg \min_{w_n\in \mbbU_n} \bignorm{f-B w_n}_{(\mbbV_m)^*} \,, 
\end{empheq}
where the \emph{discrete} dual norm is now given by
\begin{alignat}{2}
\label{eq:dualInexactNorm}
  \norm{ r }_{(\mbbV_m)^*} = \sup_{v_m\in \mbbV_m \nozero} \frac{\dual{r,v_m}_{\mbbV^*,\mbbV}}{\norm{v_m}_\mbbV} \,, \qquad \text{for any~} r\in \mbbV^*\,.
\end{alignat}
Note that a notation with a separate parametrization~$(\cdot)_m$ is used to highlight the fact that~$\mbbV_m$ need not necessarily be related to~$\mbbU_n$. 
%
%The idea advocated by 
%
%see, e.g.,~\cite{DemGopNMPDE2011} and the recent overview in~\cite{DemGopBOOK-CH2014}.
%
%
%The optimal Petrov--Galerkin method can be interpreted as a residual-minimization method in the dual space~$\mbbV^*$ (minimizing $w_n\mapsto \norm{f-Bw_n}_{\mbbV^*}$) as clarified first in~\cite{DemGopNMPDE2011}.
%
%
%In the setting of \emph{Hilbert} spaces, the inexact residual-minimization method naturally arises when attempting to approximate the~\emph{optimal} test space~$\mbbV_n$ for the Petrov--Galerkin method~\cite{}
%
%
%has recently received 
%
%
%The motivation for inexact method originated from a theory of \emph{optimal} Petrov--Galerkin discretizations which has been erected in a pioneering sequence of papers by Demkowicz, Gopalakrishnan, and other; see, e.g., 
%where 
%
%origin of the motivation for In the setting of \emph{Hilbert} spaces, 
%
%inexact residual minimization naturally arises 
%
%
%the inexact residual-minimization method was recently analyzed by Gopalakrishnan \&~Qiu~\cite{GopQiuMOC2014} and Dahmen et al.~\cite{DahHuaSchWelSINUM2012}. 
%%
%\cite{DemGopBOOK-CH2014}
%Relevant follow-up works~\cite{}? 
%-----------------------------------------------------------------------------%
\subsection{Main results}
\label{sec:MainResult}
The main objective of this work is to present a comprehensive abstract analysis of the inexact residual-minimization method~\eqref{eq:introInexactResMin} in the setting of \emph{Banach} spaces. As part of our analysis, we also obtain new abstract results for the (exact) residual-minimization method~\eqref{eq:introResMin} when specifically applied in non-Hilbert spaces. We use the remainder of the introduction to announce the main results in this work and discuss their significance.
\par
Most of our results are valid in the case that~$\mbbV$ is a \emph{reflexive} Banach space such that $\mbbV$ and $\mbbV^*$ are \emph{strictly convex}%
\footnote{A normed space~$\mbbY$ is said to be \emph{strictly convex} if, for all $y_1,y_2\in \mbbY$ such that $y_1\neq y_2$ and $\norm{y_1}= \norm{y_2} = 1$, it holds that $\norm{\theta y_1 + (1-\theta)  y_2 }_\mbbY < 1 $ for all $\theta \in (0,1)$, see e.g., \cite{DeiBOOK1985, BreBOOK2011, CiaBOOK2013}.},
which we shall refer to as the \emph{reflexive smooth setting}. The reflexive smooth setting includes Hilbert spaces, but also important \emph{non}-Hilbert spaces, since $L^p(\Omega)$ (as well as $p$-Sobolev spaces) for $p \in (1,\infty)$ are reflexive and strictly convex, however not so for~$p=1$ and~$p=\infty$ (see~\cite[Chapter~II]{CioBOOK1990} and~\cite[Section~4.3]{BreBOOK2011}). We assume this special setting throughout the remainder of Section~\ref{sec:intro}. 
\par
Indispensable in our analysis is the \emph{duality mapping}, % ~$\mcJ_\mbbV:\mbbV\rightarrow 2^{\mbbV^*}$, 
which is a well-studied operator in nonlinear functional analysis that can be thought of as the extension to Banach spaces of the well-known \emph{Riesz map} (which is a Hilbert-space construct). In the reflexive smooth setting, the duality mapping $J_\mbbV:\mbbV \rightarrow \mbbV^*$ is a bijective monotone operator that is nonlinear in the non-Hilbert case. To give a specific example, if~$\mbbV = W^{1,p}_0(\Omega)$ then $J_\mbbV$ is a (normalized) $p$-Laplace-type operator. We refer to Section~\ref{sec:duality_mapping} for details and other relevant properties.
%
% We note that $L^p(\Omega)$ (as well as $p$-Sobolev spaces) are reflexive strictly-convex Banach spaces for~$p\in (1,\infty)$, but not for~$p=1$ or~$\infty$.
\par
%
%Indispensable in our analysis is the (multi-valued) \emph{duality mapping}~$\mcJ_\mbbV:\mbbV\rightarrow 2^{\mbbV^*}$, which is a well-studied in convex analysis that extends the Riesz map (see Section~\ref{sec:duality_mapping} for relevant properties). 
% Strict convexity refers to the following geometric condition:
%-----------------------------------------------------------------------------%
%\begin{definition}[Strictly-convex Banach space]
%\label{def:introStrictConvexBspace}
%A Banach space~$\mbbY$ is said to be \emph{strictly convex} if, for all~$y_1,y_2\in \mbbY$ such that $y_1 \neq y_2$ and $\norm{y_1}_\mbbY = \norm{y_2}_\mbbY = 1$, it holds that
%%
%\begin{alignat*}{2}
% \bignorm{\theta y_1 + (1-\theta) y_2 }_\mbbY < 1\,,
% \qquad \forall \theta \in (0,1)\,.
%\end{alignat*}
%%
%\end{definition}
%%
%
% \par
%-----------------------------------------------------------------------------%
\subsubsection*{Main result I. Residual minimization: Equivalences and a~priori bounds}
%  is a \emph{nonlinear} Petrov--Galerkin discretization, a monotone mixed formulation, and a constrained minimization}
%
\noindent
The first main result in this paper concerns~\emph{equivalent characterizations} of the solution to the exact residual-minimization method~\eqref{eq:introResMin}, see Theorem~\ref{thm:ref_smooth}, as well as novel \emph{a priori bounds}. 
%marginnote{@Kris: Maybe we should permute (8) and (9)? Since (9) is the direct equivalence to (4) in Theorem 5.A}
In particular, the most important equivalence to~\eqref{eq:introResMin} is given by:
\begin{empheq}[left=\left\{\,,right=\right.,box=]{alignat=2}
\notag 
 & \text{Find }  u_n\in \mbbU_n : 
\\ \label{eq:introNPG}
 & \quad  \bigdual{\nu_n, J_{\mbbV}^{-1} (f- B u_n) }_{\mbbV^*,\mbbV} = 0
  \qquad \forall \nu_n \in B\mbbU_n \subset \mbbV^* \,.
\end{empheq}
which we refer to as a \emph{nonlinear Petrov--Galerkin} formulation. Note that statement~\eqref{eq:introNPG} is equal to
\begin{alignat}{2}
\label{eq:introNPG2}
  \bigdual{B w_n \,,\, J_{\mbbV}^{-1} (f- B u_n) }_{\mbbV^*,\mbbV} = 0
  \qquad \forall w_n\in \mbbU_n\,.
\end{alignat}
In other words, \emph{the residual minimizer~$u_n\in\mbbU_n$ of~\eqref{eq:introResMin} can be obtained by solving the nonlinear Petrov--Galerkin problem~\eqref{eq:introNPG} for~$u_n\in \mbbU_n$, and vice versa}. Owing to the equivalence, the known stability and quasi-optimality results for residual minimization~\eqref{eq:introResMin} transfer to the nonlinear Petrov--Galerkin discretization~\eqref{eq:introNPG}. Let us point out that the non-computable supremum norm in residual minimization translates into a non-tractable duality-map \emph{inverse}~$J_\mbbV^{-1}$ in \eqref{eq:introNPG}. 
\par
By introducing the auxiliary variable~$r = J_{\mbbV}^{-1} \big(f- B u_n\big) \in \mbbV$ (a residual representer), one arrives at a semi-infinite mixed formulation with monotone nonlinearity, for simplicity referred to as a \emph{monotone mixed formulation}:
% the equivalent characterization referred to as a (semi-infinite) :
%
\begin{subequations}
\label{eq:introMMF}
\begin{empheq}[left=\left\{\;,right=\right.,box=]{alignat=3} 
\notag 
& \text{Find } (r,u_n)\in \mbbV \times \mbbU_n: 
  \negquad \negquad \negquad \negquad \negquad 
\\ \label{eq:introMMFa}
& \quad J_\mbbV(r) + Bu_n &&= f \qquad && \text{in } \mbbV^* \,,
\\ \label{eq:introMMFb}
& \quad \<B w_n, r\>_{\mbbV^{\ast},\mbbV} 
&& =0  && \forall w_n\in \mbbU_n\,.
\end{empheq}
\end{subequations}
This formulation, in turn, is equivalent to a \emph{constrained-minimization formulation} (i.e., a semi-infinite saddle-point problem) involving the Lagrangian $(v,w_n)\mapsto \tfrac{1}{2}\norm{v}_\mbbV^2 - \dual{f,v}_{\mbbV^*,\mbbV} + \dual{B w_n,v} : \mbbV \times \mbbU_n \rightarrow \mathbb{R}\,$. See Section~\ref{sec:NPGMMM} for details.
\par
In the setting of \emph{Hilbert} spaces, $J_\mbbV$ coincides with the Riesz map $R_\mbbV:\mbbV \rightarrow \mbbV^*$, and~\eqref{eq:introNPG2} reduces to (recall $R_\mbbV^{-1}$ is self-adjoint):
\begin{alignat}{2}
\label{eq:introOptPG}
  \bigdual{ f- B u_n\,,\, R_\mbbV^{-1} B w_n }_{\mbbV^*,\mbbV} = 0
  \qquad \forall w_n\in \mbbU_n\,.
%\\
%  \bigdual{ f- B u_n\,,\, v_n }_{\mbbV^*,\mbbV} = 0
%  \qquad \forall v_n\in R_\mbbV^{-1} B \mbbU_n\,.
\end{alignat}% 
This coincides with a Petrov--Galerkin method~\eqref{eq:PGdisc} with the \emph{optimal}, but intractable, test space~$\mbbV_n = R_\mbbV^{-1} B \mbbU_n$. Methods that aim to approximately compute this optimal~$\mbbV_n$ have received renewed interest since~2010, starting from a pioneering sequence of papers by Demkowicz \& Gopalakrishnan on the so-called \emph{discontinuous Petrov--Galerkin} (DPG) method; see, e.g.,~\cite{DemGopCMAME2010, DemGopNMPDE2011} and the overviews in~\cite{DemGopBOOK-CH2014, GopARXIV2014}. 
% In the Hilbert-space setting, 
In the Hilbert-space setting, the connection between~\eqref{eq:introOptPG} and residual minimization was clarified first in~\cite{DemGopNMPDE2011}, while the connection with the mixed formulation~\eqref{eq:introMMF} (with $R_\mbbV$ instead of~$J_\mbbV$) was obtained by Dahmen et al~\cite{DahHuaSchWelSINUM2012}. A connection with the \emph{variational multiscale} framework has also been made~\cite{CohDahWelM2AN2012, ChaEvaQiuCAMWA2014}. 
%Optimal Petrov--Galerkin methods have 
%and is a 
%Historically, the inexact residual-minimization method originated within the setting of \emph{Hilbert} spaces from \emph{optimal Petrov--Galerkin} methods. These methods aim to approximately compute an optimal (intractable) test-space~$\mbbV_n$ for the Petrov--Galerkin method,  and they have recently received renewed interest, starting from a pioneering sequence of papers by Demkowicz \& Gopalakrishnan on the so-called \emph{discontinuous Petrov--Galerkin} method (involving a hybrid formulation with broken Sobolev spaces); see, e.g.,~\cite{DemGopNMPDE2011} and the recent overviews in~\cite{DemGopBOOK-CH2014, GopNOTES2013, DemBOOK-CH2017}. The connection with residual minimization has been clarified first in~\cite{DemGopNMPDE2011}, and the inexact method was recently analyzed in Hilbert-space settings by Gopalakrishnan \&~Qiu~\cite{GopQiuMOC2014} and Dahmen et al.~\cite{DahHuaSchWelSINUM2012}. 
%
%Furthermore, since in Hilbert-space settings $J_\mbbV$ coincides with the Riesz map $R_\mbbV:\mbbV \rightarrow \mbbV^*$,~\eqref{eq:introNPG} becomes equivalent to (recall $R_\mbbV^{-1}$ is self-adjoint):
%%
%\begin{alignat*}{2}
%  \bigdual{ f- B u_n\,,\, v_n }_{\mbbV^*,\mbbV} = 0
%  \qquad \forall v_n\in R_\mbbV^{-1} B \mbbU_n\,.
%\end{alignat*}% 
%
%
\par
In our brief review of residual minimization, we rely on theory of best approximation in Banach spaces; see Theorem~\ref{thm:resMin}.%
\footnote{For an alternative approach to the analysis of residual minimization, see Guermond~\cite[Theorem~2.1]{GueSINUM2004}.}
Within this context, we prove two novel a~priori bounds for abstract \emph{best approximations} (hence also for residual minimizers). While the classical statement, $\norm{y_0}_\mbbY \le \widetilde{C} \norm{y}_\mbbY$ with~$\widetilde{C}=2$, is valid for best approximations~$y_0$ to~$y$ within \emph{any} Banach space~$\mbbY$, this result can be sharpened in special Banach spaces. In our first improvement (see Proposition~\ref{prop:ba_apriori_bound}), we prove that the constant~$\widetilde{C}$ can be taken as the \emph{Banach--Mazur constant}~$\CBM(\mbbY)\in [1,2]$ of the underlying Banach space~$\mbbY$. The Banach--Mazur constant is an example of a so-called \emph{geometrical constant} that quantifies how ``close'' a Banach space is to being Hilbert, and this particular geometrical constant was recently introduced by Stern~\cite{SteNM2015} to sharpen the a~priori \emph{error} estimate for the Petrov--Galerkin discretization~\eqref{eq:PGdisc}. 
%    examples of so-called \emph{geometrical constants} that quantify how ``close'' a Banach space is to being Hilbert. 
In our second improvement (see Proposition~\ref{prop:improvedApriori}), we prove that~$\widetilde{C}$ can also be taken as~\mbox{$1{+}\CAO(\mbbY)$}, where $\CAO(\mbbY) \in [0,1]$ is a newly introduced constant that we refer to as the \emph{asymmetric-orthogonality constant} %~$\CAO(\mbbY)$ 
of~$\mbbY$. 
%
% while in the second improvement it can be taken as \mbox{$1{+}\CAO(\mbbY)$}, where $\CAO(\mbbY)$ is a newly introduced constant that we refer to as the \emph{asymmetric-orthogonality constant}~$\CAO(\mbbY)$ of~$\mbbY$. The Banach--Mazur constant, recently introduced by Stern~\cite{SteNM2015} to sharpen the a~priori \emph{error} estimate for the Petrov--Galerkin discretization~\eqref{eq:PGdisc}, and the asymmetric-orthogonality constant are examples of so-called \emph{geometrical constants} that quantify how ``close'' a Banach space is to being Hilbert. In particular, $\CBM(\mbbY) $
%
% This constant explicitly on the \emph{geometry} . \marginnote{Ref \cite{JohLinBOOK2001} when first mentioning such constants}
%Many such geometrical constants exist in literature, and they quantify how ``close'' a Banach space is to being Hilbert. recently introduced by Stern~\cite{SteNM2015}.
% Such geometrical constants are important in the study of the underlying  in the 
% As part of our review of the analysis results 
%
%\par
%%
%In the Hilbert-space case, 
%The optimal Petrov--Galerkin method can be interpreted as a residual-minimization method in the dual space~$\mbbV^*$ (minimizing $w_n\mapsto \norm{f-Bw_n}_{\mbbV^*}$) as clarified first in~\cite{DemGopNMPDE2011}.
%
%Cohen, Dahmen and Welper~\cite{CohDahWelM2AN2012}
%
%-----------------------------------------------------------------------------%
\subsubsection*{Main result II. Inexact method: Stability and quasi-optimality}
\noindent
The second main result in this paper is the complete analysis of the \emph{inexact} residual minimization method~\eqref{eq:introInexactResMin}. To carry out the analysis, we first present equivalent characterizations of the solution of the inexact method,
% Henceforth 
see Theorem~\ref{teo:discrete_residual}. These characterizations are the \emph{fully discrete} versions of the above-mentioned characterizations for exact residual minimization. In particular, the \emph{inexact} nonlinear Petrov--Galerkin method corresponds to~\eqref{eq:introNPG} with~$J_\mbbV^{-1}$ replaced by~$I_m J_{\mbbV_m}^{-1} \circ I_m^*$, where $J_{\mbbV_m}^{-1}$ is the inverse of the duality map in~$\mbbV_m$, and $I_m : \mbbV_m \rightarrow \mbbV$ is the natural injection. The computationally most-insightful equivalence is the one corresponding to the mixed formulation~\eqref{eq:introMMF}, and 
% , since it involves a monotone nonlinearity
we refer to the resulting discretization as a \emph{monotone mixed method}:
\begin{subequations}
\label{eq:introMMM}
\begin{empheq}[left=\left\{\;,right=\right.,box=]{alignat=3} 
\notag
& \text{Find } (r_m,u_n)\in \mbbV_m \times \mbbU_n :
\\ \label{eq:introMMMa}
& \quad \bigdual{ J_\mbbV(r_m),v_m }_{\mbbV^{\ast},\mbbV}
+\bigdual{ Bu_n,v_m }_{\mbbV^{\ast},\mbbV} 
&&=\bigdual{f,v_m}_{\mbbV^{\ast},\mbbV} 
 & \qquad  & \forall v_m\in \mbbV_m\,,
\\ \label{eq:introMMMb}
& \quad \bigdual{B^*r_m,w_n}_{\mbbU^{\ast},\mbbU} 
&& =0  && \forall w_n\in \mbbU_n\,,
\end{empheq}
\end{subequations}
where the auxiliary variable~$r_m$ is a \emph{discrete} residual representer.
\par
In the analysis of the stability and quasi-optimality of the inexact method, some compatibility is demanded on the pair~$(\mbbU_n,\mbbV_m)$. This compatibility is stated in terms of~\emph{Fortin's condition} (involving a Fortin operator~$\Pi:\mbbV\rightarrow \mbbV_m$, see Assumption~\ref{assumpt:Fortin}), which is essentially a discrete inf--sup condition on~$(\mbbU_n,\mbbV_m)$. (Note that this compatibility requirement is less stringent than the one required for the Petrov--Galerkin discretization since the dimensions of~$\mbbU_n$ and $\mbbV_m$ can be distinct.) Under Fortin's condition, we prove the unique existence of the pair~$(r_m,u_n)\in
\mbbV_m\times\mbbU_n$ solving~\eqref{eq:introMMM} and its continuous dependence on the data, see Theorem~\ref{thm:discrete}. Then, we prove a corresponding \emph{a~posteriori} error estimate, see Theorem~\ref{thm:aposteriori}, the result of which happens to coincide with the result in the Hilbert case~\cite{CarDemGopSINUM2014, GopARXIV2014}. A straightforward consequence (see Corollary~\ref{cor:aPrioriErrorEst}) is then the quasi-optimal error estimate~\eqref{eq:qOpt} with
\begin{alignat}{2}
\label{eq:CqOpt1}
  C = \frac{(D_\Pi + C_\Pi) M_B}{\gamma_B}\,, 
\end{alignat}
where $D_\Pi$ and $C_\Pi$ are boundedness constants appearing in Fortin's condition, and $M_B$ and $\gamma_B$ are the continuity and bounded-below constants of~$B$. 
\par
The importance of Fortin's condition in the analysis of the inexact method was recognized first by Gopalakrishnan and Qiu~\cite{GopQiuMOC2014}, who studied the inexact optimal Petrov--Galerkin method in Hilbert spaces. In this setting, Fortin's condition implies that $R_{\mbbV_m}^{-1} B \mbbU_n$ is a \emph{near-optimal} test space that is \emph{sufficiently close} to the optimal one~$R_\mbbV^{-1} B \mbbU_n$ (cf.~\cite[Proposition~2.5]{BroSteCAMWA2014}). The fact that near-optimal test spaces imply quasi-optimality was established by Dahmen et al.; see \cite[Section~3]{DahHuaSchWelSINUM2012}. Let us point out however, that the concept of an optimal test space is completely absent in our Banach-space theory, and seems to apply to Hilbert spaces only.
% sufficient closeness (quantifiable through the concept of \emph{$\delta$-proximality}) implies quasi-optimality of the inexact Petrov--Galerkin method was established by Dahmen et al.; see \cite[Sec.~3]{DahHuaSchWelSINUM2012}. Let us point out however that the concept of an optimal test space is completely absent in our current Banach-space theory, and seems to be a Hilbert-space construct only.
% It is important to note however that in our current Banach-space theory, there is no concept of an optimal test space which naturally follows from 
%
% using the test space~$\mbbV_n = R_{\mbbV_m}^{-1} B \mbbU_n$; see also~\cite{BroSteCAMWA2014}. This condition implies that the 
%
%An alternative condition, referred to as~$\delta$-proximal Independently, Dahmen et al~\cite{DahHuaSchWel} gave a The connection between 
%
% from which an \emph{a~priori} error estimate naturally follows; see Corollary~\ref{cor:aPrioriErrorEst}.
%
\par
Although the result in~\eqref{eq:CqOpt1} demonstrates quasi-optimality for the inexact method in Banach-space settings, the constant in~\eqref{eq:CqOpt1} does not reduce to the known result, $C = C_\Pi M_B / \gamma_B$, when restricting to Hilbert-space settings~\cite[Theorem~2.1]{GopQiuMOC2014}. To resolve the discrepancy, we improve the constant in~\eqref{eq:CqOpt1} by including the dependence on the geometry of the involved Banach spaces; see Theorem~\ref{thm:discrete_apriori}. The proof of this sharper estimate is nontrivial, as it requires a suitable extension of a Hilbert-space technique due to Xu and Zikatanov~\cite{XuZikNM2003} involving the classical identity $\norm{I-P} = \norm{P}$ for Hilbert-space projectors~$P$, which is generally attributed to Kato~\cite{KatNM1960} (cf.~\cite{SzyNALG2006}). A key idea is the recent extension $\norm{I-P} \le C_\mathrm{S} \norm{P}$ for Banach-space projectors by Stern~\cite{SteNM2015}, where~$C_\mathrm{S}$ depends on the Banach--Mazur constant, however, since that extension applies to linear projectors, we generalize Stern's result to a suitable class of nonlinear projectors (see Lemma~\ref{lem:I-P}). 
% This is combined with an a~priori bound for the \emph{inexact} residual minimizer involving the asymmetric-orthogonality constant, (inexact) nonlinear Petrov--Galerkin projectors . 
Combined with an a~priori bound for the inexact residual minimizer involving the asymmetric-orthogonality constant, we then prove that the constant in the quasi-optimal error estimate~\eqref{eq:CqOpt1} can be improved to 
\begin{alignat*}{2}
 C = \min\left\{ 
  \frac{C_\Pi}{\gamma_B}\,\big(1+\CAO(\mbbV)\big)\,M_B\,\CBM(\mbbU) \,,\,1+ \frac{C_\Pi }{\gamma_B}\big(1+\CAO(\mbbV)\big) M_B  
  \right\}\,,
%  \marginnote{@Ignacio: Can you please make this correct?}
\end{alignat*}
which is consistent with the Hilbert-space result $C=C_\Pi M_B/\gamma_B$ since in that case $\CAO(\mbbV) = 0$ and $\CBM(\mbbU) = 1$.
%
%Because of all the 
%The quasi-optimality analysis generalizes the proof of various existing methods. 
%
\subsection{Discussion: Unifying aspects}
% Unification of methods and theories} 
% and applicability
%
Let us emphasize that the above quasi-optimality theory for the inexact method generalizes existing theories for other methods that are in some sense contained within the inexact method, and therefore it provides a unification of these theories. In particular, the theory generalizes Babu\v{s}ka's theory for Petrov--Galerkin methods~\cite{BabNM1971}, Guermond's theory for exact residual minimization~\cite{GueSINUM2004}, and the Hilbert-space theory for inexact residual minimization (including the DPG method)~\cite{GopQiuMOC2014, DahHuaSchWelSINUM2012}. For a schematic hierarchy with these connections and its detailed discussion, we refer to Section~\ref{sec:connections}.
\subsection{Outline of paper}
\newcommand{\IgnacioBullet}{--~}
The remainder of the paper is organized as follows.
\\
%\IgnacioBullet
%%\begin{enumerate}[label=\textbullet, topsep=0ex, itemsep=0ex, parsep=0ex, leftmargin=0ex, itemindent=1.25em]
%%\item
%Section~\ref{sec:linBanach} reviews the solvability of abstract operator equations in Banach spaces, putting emphasis on equations driven by variational forms (e.g., weak formulations of PDEs). 
%\\
\IgnacioBullet
Section~\ref{sec:resMin} is devoted to necessary preliminaries on the duality mapping and abstract theory of best approximation in Banach spaces.
\\
\IgnacioBullet
Section~\ref{sec:NPGMMM} is dedicated to residual minimization and its characterization via the duality mapping giving rise to several equivalences, including what is referred to as the \emph{nonlinear Petrov-Galerkin method} and \emph{monotone mixed formulation}.
\\
\IgnacioBullet
Section~\ref{sec:practical} analyzes the tractable inexact method. We establish the equivalence of the \emph{inexact residual minimization}, \emph{inexact nonlinear Petrov--Galerkin} and \emph{monotone mixed method}. We then study the stability of this method and perform a comprehensive error analysis.
\\
%\IgnacioBullet
%%Section~\ref{sec:App1AdvecReac} and Section~\ref{sec:App2Poisson} are devoted to show applications of the theory developed in the previous sections. In particular, 
%Section~\ref{sec:App1AdvecReac} applies the theory to the \emph{advection--reaction} problem in a weak setting in Banach spaces. Compatible discrete spaces are discussed, and the elimination of the \emph{Gibbs phenomena} is illustrated. 
%% we illustrate one of the strengths of the proposed method by .
%%On the other hand, 
%\\
%\IgnacioBullet
%Section~\ref{sec:App2Poisson} applies the theory to the \emph{Laplace operator} in a non-symmetric Banach-space setting. A general compatible pair is discussed, and a Galerkin instability on graded meshes is illustrated and resolved.
%%of the \emph{Gibbs phenomena} is illustrated. 
%%
%% is dedicated to show the main features of the classical \emph{Poisson problem}, but in a Banach-space setting. %, including remarkable illustrations as the elimination of nonuniform stability in graded meshes. 
%\\
\IgnacioBullet
Finally, Section~\ref{sec:connections} reviews the connections to other existing methods (standard Petrov--Galerkin, exact residual minimization and the inexact method in Hilbert-space settings),
%  that are contained within the discrete method.
and points out how the presented quasi-optimality analysis applies in each situation.
\section{Preliminaries: Duality mappings and best approximation} 
\label{sec:resMin}
%-----------------------------------------------------------------------------%
In this section we briefly review some relevant theory in the classical subject of duality mappings, and elementary results from best approximation theory in Banach spaces. These preliminaries are required for our analysis of (inexact) residual-minimization problems. 
%
% Thus, we briefly review here some of its main properties. 
%%
%As the fundamental approximation to problem~\eqref{eq:Bu=f}, we consider the minimization of the residual of the equation.
%% in~\eqref{eq:operator_form}
%In this section we analyze this discretization method, relying on best-approximation theory in Banach spaces. 
%%
%\par
%%
%As part of our analysis, we prove a novel a~priori bound on best approximations in Banach spaces, which explicitly depends on a geometrical constant of the underlying Banach space. This result translates into a novel a~priori bound on residual minimizers.
%%
%\par
%%
%The concept of the duality mapping will be needed to characterize the solution of residual minimization. Thus, we briefly review here some of its main properties. 
%Additionally, at the end of this section we prove another novel a~priori bound for best-approximation minimizers. This bound depends on a new geometrical constant that is defined using the duality mapping. 
%------------------------------------
\subsection{The duality mapping}
\label{sec:duality_mapping}
An extensive treatment on duality mappings can be found in Cioranescu~\cite{CioBOOK1990}. Other relevant treatments in the context of nonlinear functional analysis are by Brezis~\cite[Chapter~1]{BreBOOK2011}, Deimling~\cite[Section~12]{DeiBOOK1985}, Chidume~\cite[Chapter~3]{ChiBOOK2009} and Zeidler~\cite[Chapter~32.3d]{ZeiBOOK1990b}, while an early treatment on duality mappings is by Lions~\cite[Chapter~2, Section~2.2]{LioBOOK1969}.
% \cite{LioBOOK1969, DeiBOOK1985, ZeiBOOK1985, CioBOOK1990,  ChiBOOK2009, BreBOOK2011}.
%
We recall results for duality mappings that will be useful for the characterization of best approximations and residual minimizers. 
%-----------------------------------------------------------------------------%
%
%-----------------------------------------------------------------------------%
\begin{definition}[Duality mapping]
\label{def:Dmap}
Let $\mbbY$ be a normed vector space. The multivalued mapping ${\mathcal{J}_\mbbY: \mbbY\to 2^{\mbbY^\ast}}$ defined by
$$
\mathcal{J}_\mbbY(y):=\left\{y^\ast\in \mbbY^\ast \,:\, \<y^\ast,y\>_{\mbbY^\ast,\mbbY}=\|y\|^2_\mbbY=\|y^\ast\|_{\mbbY^\ast}^2\right\}\,,
$$
is the \emph{duality mapping} on $\mbbY$. 
\end{definition}
%-----------------------------------------------------------------------------%
%
By the Hahn-Banach extension Theorem (see, e.g.,~\cite[Corollary 1.3]{BreBOOK2011}), the set $\mathcal{J}_\mbbY(y)\subset \mbbY^*$ is non-empty for every $y\in \mbbY$. 
% Observe that the definition of the duality mapping depends on the norm that is used in~$\mbbY$. 
Some basic properties of $\mathcal{J}_\mbbY$ are summarized in the following. 
%
%-----------------------------------------------------------------------------%
\begin{proposition}[Duality mapping]
\label{prop:duality} Let $\mbbY$ be a normed vector space and \mbox{$y\in \mbbY$}. 
\begin{enumerate}[(i)]
\item The set 
$\mathcal{J}_\mbbY(y)\subset\mbbY^*$ is bounded, convex, and closed. 
\item The duality mapping $\mathcal{J}_\mbbY$ is homogeneous, and it is monotone in the sense that:
$$
\<y^\ast-z^\ast,y-z\>_{\mbbY^\ast,\mbbY}\geq \big( \norm{y}_\mbbY -\norm{z}_\mbbY\big)^2 \ge 0, 
$$
for all  $y,z\in \mbbY$, for all  
$y^\ast\in \mathcal J_\mbbY(y)$ and for all $z^\ast\in \mathcal J_\mbbY(z)$.
\item For any~$y^*\in \mcJ_\mbbY(y)$, its norm supremum is achieved by~$y$, i.e., 
\begin{alignat}{2}
\label{eq:normSup}
  \sup_{z\in \mbbY} \frac{\dual{y^*,z}_{\mbbY^*,\mbbY}}
  {\norm{z}_{\mbbY}} = \frac{\dual{y^*,y}_{\mbbY^*,\mbbY}}{\norm{y}_{\mbbY}} \,.
\end{alignat}
\end{enumerate}
\end{proposition}
%-----------------------------------------------------------------------------%
%
%-----------------------------------------------------------------------------%
\begin{proof} 
These results are classical; see, e.g., \cite[Chapter~1]{BreBOOK2011}. 
% PROOF:
%For the sake of completeness we provide a proof. 
%To show convexity assume that $x_1^\ast,x_2^\ast\in\mathcal{J}_\mbbX(x)$ and observe that for all $t\in(0,1)$ we have:
%$$
%\<tx_1^\ast + (1-t)x_2^\ast,x\>_{\mbbX^*,\mbbX}=\|x\|_\mbbX^2 \quad\hbox{and}\quad 
%\|tx_1^\ast + (1-t)x_2^\ast\|_{\mbbX^\ast}\leq\|x\|_{\mbbX}.
%$$
%Homogeneity, closedness and boundedness are immediate from the definition of $\mathcal J_\mbbX(x)$ (use the continuity of the norm). 
%Finally, let $x,y\in \mbbX$, $x^\ast\in\mathcal J_\mbbX(x)$ and $y^\ast\in\mathcal J_\mbbX(y)$. Using Young's inequality we get:
%$$
%\begin{array}{rl}
%& \<x^\ast,y\>_{\mbbX^\ast,\mbbX}\leq \|x^\ast\|_{\mbbX^\ast}\|y\|_{\mbbX}\leq \ds{1\over2}\|x^\ast\|_{\mbbX^\ast}^2+\ds{1\over2}\|y\|_{\mbbX}^2=\ds{1\over2}\<x^\ast,x\>_{\mbbX^\ast,\mbbX}+\ds{1\over2}\<y^\ast,y\>_{\mbbX^\ast,\mbbX}\\
%\hbox{and} & \\
%& \<y^\ast,x\>_{\mbbX^\ast,\mbbX}\leq \|y^\ast\|_{\mbbX^\ast}\|x\|_{\mbbX}\leq \ds{1\over2}\|y^\ast\|_{\mbbX^\ast}^2+\ds{1\over2}\|x\|_{\mbbX}^2=\ds{1\over2}\<y^\ast,y\>_{\mbbX^\ast,\mbbX}+\ds{1\over2}\<x^\ast,x\>_{\mbbX^\ast,\mbbX}\,\,.
%\end{array}
%$$
%Thus, monotonicity follows by adding these inequalities.
\end{proof}
%
%-----------------------------------------------------------------------------%
%\subsection{Special Banach spaces}\label{sec:specialcases}
%
We next list important properties of the duality mapping~$\mcJ_{\mbbY}:\mbbY\rightarrow 2^{\mbbY^*}$ in special Banach spaces. 
% Particularly of interest is the property of strict convexity (i.e., strict convexity of the norm). For a precise definition of a strictly convex Banach space we refer to Brezis~\cite{BreBOOK2011}, Deimling~\cite{DeiBOOK1985} or Ciarlet~\cite{CiaBOOK2013}.
%
\subsubsection{Strict convexity of $\mbbY^\ast$}\label{sec:strictconvexdual}
The space $\mbbY^*$ is strictly convex if and only if
%(see Brezis~\cite{BreBOOK2011}, Deimling~\cite{DeiBOOK1985} or Ciarlet~\cite{CiaBOOK2013}), 
$\mcJ_{\mbbY}:\mbbY\rightarrow 2^{\mbbY^\ast}$ is a \emph{single-valued map}; see~\cite[Proposition 12.3]{DeiBOOK1985}. In that case we use the notation: 
\begin{alignat*}{2}
  J_\mbbY:\mbbY\to \mbbY^\ast, \quad \text{in other words,} 
  \quad \mcJ_\mbbY(y) = \{J_\mbbY(y)\}.
\end{alignat*}
%
%To see that $\mcJ_{\mbbY}(y)$ is single-valued, let $\mathcal{J}_\mbbY(y)$ contain two different elements, say $y_1^\ast$ and $y_2^\ast$. Then, since $\mathcal{J}_\mbbY(y)$ is a convex set (see Proposition~\ref{prop:duality}(i)), $\theta y_1^\ast + (1-\theta)y_2^\ast\in \mathcal{J}_\mbbY(y)$ for any $\theta\in(0,1)$. Therefore, by strict convexity,
%$$
%\|y\|_\mbbY=\|\theta y_1^\ast + (1-\theta)y_2^\ast\|_{\mbbY^\ast}<\|y_1^\ast\|_{\mbbY^\ast}=\|y\|_\mbbY,
%$$
%which is a contradiction.
% Strict convexity of $\mbbY^*$ is also a necessary condition for single-valuedness.
%
% \par
%
Furthermore, if~$\mbbY^*$ is strictly convex, then $\mcJ_{\mbbY}:\mbbY\rightarrow 2^{\mbbY^\ast}$ is \emph{hemi-continuous}~\cite[Section~12.3]{DeiBOOK1985}:
\begin{alignat}{2}
\label{eq:hemicont}
   J_{\mbbY}(y+\lambda z) \rightharpoonup  J_{\mbbY}(y)
   \qquad \text{as } \lambda \rightarrow 0^+,\qquad \forall y,z\in \mbbY.
\end{alignat}
\par
Another important property is concerned with the duality map on subspaces.
We state this as the following Lemma, and we include a proof since we could not find this result in the existing literature.
%
%-----------------------------------------------------------------------------%
\begin{lemma}[Duality map on a subspace]
\label{lem:IJI}
Let $\mbbY$ be a Banach space, $\mbbY^*$ strictly convex, and $J_\mbbY: \mbbY \rightarrow \mbbY^*$ denote the duality map on $\mbbY$. Let $\mbbM\subset \mbbY$ denote a linear subspace of~$\mbbY$, and $J_\mbbM:\mbbM\rightarrow \mbbM^*$ denote the corresponding duality map on $\mbbM$.
Then, 
\begin{alignat*}{2}
  I_\mbbM^* J_\mbbY \circ I_\mbbM = J_{\mbbM}\,,
\end{alignat*}
where $I_\mbbM:\mbbM \rightarrow \mbbY$ is the natural injection. 
\end{lemma}
%-----------------------------------------------------------------------------%
%
%-----------------------------------------------------------------------------%
\begin{proof}
Let $z\in\mbbM$ and consider the linear and continuous functional ${J_\mbbM(z)\in\mbbM^*}$. Using the Hahn--Banach extension (see~\cite[Corollary~1.2]{BreBOOK2011}), we extend this functional to an element $\widetilde{J_\mbbM(z)}\in\mbbY^*$ such that ${\norm{\widetilde{J_\mbbM(z)}}_{\mbbY^*}=\norm{J_\mbbM(z)}_{\mbbM^*}}\,$.%
\footnote{In fact, the Hahn--Banach extension is unique on account of strict convexity of~$\mbbY^*$.}
Observe that the extension satisfies
\begin{alignat*}{2}
&\norm{\widetilde{J_\mbbM(z)}}_{\mbbY^*}=\|I_\mbbM z\|_\mbbY 
\qquad \text{and}
\\
&\dual{\widetilde{J_\mbbM(z)},I_\mbbM z}_{\mbbY^*,\mbbY}=
\dual{J_\mbbM(z),z}_{\mbbM^*,\mbbM}=\norm{I_\mbbM z}^2_\mbbY .
\end{alignat*}
So, as a matter of fact, $\widetilde{J_\mbbM(z)}=J_\mbbY(I_\mbbM z)$. Therefore, by the extension property of~$\widetilde{J_\mbbM(z)}$ we obtain
\begin{alignat*}{2}
I_\mbbM^*J_\mbbY(I_\mbbM z)=I_\mbbM^*\widetilde{J_\mbbM(z)}=J_\mbbM(z).
\end{alignat*}
\end{proof}
%-----------------------------------------------------------------------------%
%
%-----------------------------------------------------------------------------%
\subsubsection{Strict convexity of $\mbbY$}
If $\mbbY$ is strictly convex, then $\mathcal J_\mbbY$ is \emph{strictly monotone}, that is:
\begin{alignat}{2}
 \label{eq:strictMonotone}
 \<y^\ast-z^\ast,y-z\>_{\mbbY^\ast,\mbbY} > 0, 
 \quad \text{for all } y \neq z\,,
\text{ any } y^* \in \mcJ_\mbbY(y)
\text{ and } z^* \in \mcJ_\mbbY(z) \,. 
\end{alignat}
Furthermore, $\mathcal J_\mbbY:\mbbY\to 2^{\mbbY^*}$ is \emph{injective}. In fact, if $y$ and $z$ are two distinct points in $\mbbY$, then $\mathcal J_\mbbY(y)\cap\mathcal J_\mbbY(z)=\emptyset$ (otherwise~\eqref{eq:strictMonotone} would be contradicted).
%Indeed, if $y^\ast\in \mathcal J_\mbbY(y)\cap\mathcal J_\mbbY(z)$, then 
%$\|y\|_\mbbY=\|z\|_\mbbY$ and
%$$
%\|y\|^2_\mbbY=\<y^\ast,\theta y+(1-\theta)z\>_{\mbbY^\ast,\mbbY}\leq \|y\|_\mbbY\,\,\|\theta y+(1-\theta)z\|_\mbbY<\|y\|_\mbbY^2,
%$$
%which is a contradiction. 
%%
%\par
%%
%Furthermore, if $\mbbY$ is strictly convex, then $\mathcal J_\mbbY$ is \emph{strictly monotone}:
%%
%\begin{alignat}{2}
% \label{eq:strictMonotone}
% \<y^\ast-z^\ast,y-z\>_{\mbbY^\ast,\mbbY} > 0, 
% \quad \text{for all } y \neq z\,,
%\text{ any } y^* \in \mcJ_\mbbY(y)
%\text{ and } z^* \in \mcJ_\mbbY(z) \,. 
%\end{alignat}
%
It is known that the converse holds as well: Strict monotonicity of $\mathcal J_\mbbY$ implies strict convexity of $\mbbY$, a result due to Petryshyn~\cite{PetJFA1970}. 
%
%Additionally, we can show the duality mapping is strictly monotone. Indeed, assume that there are $x,y\in \mbbX$, $x^\ast\in\mathcal J_\mbbX(x)$ and $y^\ast\in\mathcal J_\mbbX(y)$ satisfying:
%$$
%\<x^\ast-y^\ast,x-y\>_{\mbbX^\ast,\mbbX}= 0.
%$$
%Then, using the same technique as in the proof of Proposition~\ref{prop:duality}, it follows that $x^\ast\in \mathcal J_\mbbX(y)$ and $y^\ast\in \mathcal J_\mbbX(x)$, which is only possible if $x=y$.
%
%-----------------------------------------------------------------------------%
\subsubsection{Reflexivity of $\mbbY$}\label{sec:reflexivecase}
The space $\mbbY$ is a reflexive Banach space if and only if $\mathcal J_\mbbY:\mbbY\to 2^{\mbbY^*}$ is \emph{surjective}; see~\cite[Theorem 12.3]{DeiBOOK1985}. This is meant in the following sense: Every $y^\ast\in \mbbY^\ast$ belongs to a set $\mathcal J_\mbbY(y)$, for some $y\in \mbbY$. 
%Indeed, let $\mathcal J_{\mbbY^*}:\mbbY^*\to\mbbY^{**}$ be the duality mapping on $\mbbY^\ast$ and choose some $y^{\ast\ast}\in \mathcal J_{\mbbY^*}(y^\ast)$. By reflexivity, there is a $y\in \mbbY$ such that:
%$$
%\|y\|_\mbbY^2=\|y^{\ast\ast}\|_{\mbbY^{\ast\ast}}^2=\|y^\ast\|^2_{\mbbY^\ast}=\<y^{\ast\ast},y^*\>_{\mbbY^{\ast\ast},\mbbY^\ast}=\<y^{\ast},y\>_{\mbbY^{\ast},\mbbY}\quad.
%$$
%Hence $y^\ast\in \mathcal J_\mbbY(y)$.

%-----------------------------------------------------------------------------%
\subsubsection{Reflexive smooth setting}
\label{sec:dualMapSmooth}
An important case in our study is when the Banach space $\mathbb Y$ has all the previously listed properties, i.e., $\mbbY$ and $\mbbY^*$ are strictly convex and reflexive Banach spaces, referred to as the \emph{reflexive smooth setting}. Two important straightforward consequences need to be remarked in this situation:
\begin{enumerate}[(i)]
\item The duality maps $J_\mbbY:\mbbY\to\mbbY^*$ and $J_{\mbbY^*}:\mbbY^*\to\mbbY^{**}$ are bijective.
\item $J_{\mbbY^*}=\mathcal I_\mbbY\circ J^{-1}_\mbbY$, where $\mathcal I_\mbbY:\mbbY\to\mbbY^{**}$ is the canonical injection. Shortly, $J_{\mbbY^*}= J^{-1}_\mbbY$, by means of canonical identification.
\end{enumerate}
%
%-----------------------------------------------------------------------------%
\subsubsection{Subdifferential property}
A key result is that the duality mapping coincides with a~\emph{subdifferential}. Recall that for a Banach space~$\mbbY$ and function $f:\mbbY\to \rr$, the subdifferential~$\partial f(y)$ of~$f$ at a point $y\in\mbbY$ is defined as the set:
\begin{alignat*}{2}
\partial f(y):=\Big\{y^*\in \mbbY^*: f(z)-f(y)\geq \<y^*,z-y\>_{\mbbY^{\ast},\mbbY} \,,\,\forall z \in \mbbY\Big\}\,.
\end{alignat*}
%
% The following result is known as Asplund's Theorem.
%
\begin{proposition}[Duality mapping is a subdifferential]\label{prop:subdifferential}
Let $f_\mbbY:\mbbY \to \rr$ be defined by $f_\mbbY(\cdot):=\frac{1}{2}\|\cdot\|_\mbbY^2$. Then, for any $y\in \mbbY$, % the following characterization of the duality mapping holds:
$
\mathcal J_\mbbY(y)=\partial f_\mbbY(y).
$
\end{proposition}
%-----------------------------------------------------------------------------%
%
%-----------------------------------------------------------------------------%
\begin{proof}
See, e.g., Asplund~\cite{AspBOAMS1967} or Cioranescu~\cite[p.~26]{CioBOOK1990}.
%
% PROOF
%If $x^*\in \mathcal J_\mbbX(x)$ and $y\in \mbbX$, we have:
%$$
%\<x^\ast,y\>_{\mbbX^\ast,\mbbX}\leq \|x\|_{\mbbX}\|y\|_{\mbbX}\leq {1\over2}\|x\|_{\mbbX}^2+{1\over2}\|y\|_{\mbbX}^2.
%$$
%Thus, subtracting $\<x^\ast,x\>_{\mbbX^\ast,\mbbX}=\|x\|^2_\mbbX$ to the above inequality, we get that $x^*\in \partial f(x)$. Conversely, if 
%$x^*\in \partial f(x)$, then for $y=\alpha x$ we have:
%$$
%{\alpha^2-1\over 2}\|x\|_\mbbX^2\geq (\alpha-1)\<x^\ast,x\>_{\mbbX^\ast,\mbbX}\,\,,\quad \forall \alpha\in\rr.
%$$
%Taking the lateral limits when $\alpha\to 1^+$ and $\alpha\to 1^-$ respectively, is straightforward to see that $\|x\|^2_\mbbX=\<x^\ast,x\>_{\mbbX^\ast,\mbbX}$. Combining this last fact with the subdifferential property we get:
%$$
%\<x^\ast,y\>_{\mbbX^\ast,\mbbX}\leq {1\over2}\|x\|_{\mbbX}^2+{1\over2}\|y\|_{\mbbX}^2\,,\quad \forall y \in \mbbX.
%$$
%The last inequality is also true for any element of $X$ of the type $(\|x\|_\mbbX\|y\|_\mbbX^{-1})y$, in which case we obtain that $\|x^*\|_{\mbbX^*}\leq \|x\|_\mbbX$ and therefore $x^*\in \mathcal J_\mbbX(x)$.
\end{proof}
%-----------------------------------------------------------------------------%
%
%-----------------------------------------------------------------------------%
\begin{remark}[G\^ateaux gradient]
The subdifferential of $f_\mbbY(\cdot)=\frac{1}{2}\|\cdot\|_\mbbY^2$ contains exactly one point if $\mbbY^*$ is strictly convex (recall from Section~\ref{sec:strictconvexdual}). In that case, $f_\mbbY$ is G\^ateaux differentiable with G\^ateaux gradient~$\nabla f_\mbbY(\cdot)$, 
%Thus, according to section~\ref{sec:reflexivecase}, for every $y^*\in \mbbY^*$, there is $y\in \mbbY$ such that $y^*=\nabla f(y)$.} 
and for any $y\in \mbbY$ we have (see e.g.~\cite[Corollary 2.7]{CioBOOK1990}):
\begin{alignat*}{1}
J_\mbbY(y)=\nabla f_\mbbY(y)\,.%
\footnotemark
\end{alignat*} 
%
% where~$\nabla$ denotes the .
\end{remark}
\footnotetext{If $\mbbY^*$ is \emph{uniformly convex}, then $f_\mbbY$ is Fr\'echet differentiable and the duality map $J_\mbbY:\mbbY\to\mbbY^*$ is uniformly continuous on the unit sphere of $\mbbY$. Moreover, by the Milman--Pettis Theorem~\cite[Section~3.7]{BreBOOK2011}, uniform convexity of $\mbbY^*$ implies reflexivity of $\mbbY^*$ (hence, reflexivity of $\mbbY$). Note however that a strictly convex space is not necessarily uniformly convex, see~\cite[Section~12.1]{DeiBOOK1985} for an example.}
%
%-----------------------------------------------------------------------------%
%
%\begin{remark}
%marginnote{@Ignacio: Should we remove this remark, or shorten it and turn it into a footnote?}
%Additionally, if $\mbbX^*$ is uniformly convex, then $f(\cdot)={1\over2}\|\cdot\|_\mbbX^2$ is Fr\'echet differentiable and the duality map $J_\mbbX:\mbbX\to\mbbX^*$ is uniformly continuous on the unit sphere of $\mbbX$. Moreover, by the Milman-Pettis Theorem \cite{BreBOOK2011}, uniform convexity of $\mbbX^*$ implies reflexivity of $\mbbX^*$ (hence, reflexivity of $\mbbX$). 
%Thus, according to section~\ref{sec:reflexivecase}, for every $x^*\in \mbbX^*$, there is $x\in \mbbX$ such that $x^*=\nabla f(x)$.  
%\end{remark}
%
\begin{example}[The $L^p$ case]\label{ex:Lp_Dmap}
We recall here an explicit formula for the duality map in the Banach space 
$L^p(\Omega)$ where $\Omega \subset \mbbR^d$, $d\ge 1$. For $p\in(1,+\infty)$ the space $L^p$ is reflexive and strictly convex (as well as the dual space $L^{q}$, where $q=\frac{p}{p-1}$); see e.g.~\cite[Chapter~II]{CioBOOK1990} and~\cite[Section~4.3]{BreBOOK2011}. For $v\in L^p(\Omega)$ the duality map is defined by the action: 
%
%\begin{subequations}
%\label{eq:Dmap}
\begin{equation}
\label{eq:Dmapa}
\bigdual{ J_{L^p(\Omega)}(v),w }_{L^q(\Omega),L^p(\Omega)}
 :=\bignorm{v}_{L^p(\Omega)}^{2-p}
 \int_\Omega |v|^{p-1}\sign(v)\,w ,\quad \forall w\in L^p(\Omega)\,,
\end{equation}
which can shown by computing the G\^{a}teaux derivative of~$v\mapsto \tfrac{1}{2}(\int_\Omega |v|^p)^{2/p}$, or by verifying the identities in Definition~\ref{def:Dmap}. %Throughout the paper, we systematically abuse the notation by just writing
%%
%\begin{alignat}{2}
%   J_p(v)\equiv J_{L^p}(v) \equiv J_{L^p(\Omega)}(v) \,.
%\end{alignat}
%\end{subequations}
%
In the case $p=1$, the formula in the right-hand side of~\eqref{eq:Dmapa} also works and defines an element in the set~$\mcJ_{L^1(\Omega)}(v)$. 
Note however that $L^1$ is not a special Banach space as discussed above. 
\end{example}
\subsection{Best approximation in Banach spaces}
\label{sec:bestAppr}
%
%To analyze~\eqref{eq:residual_minimization}, 
We now consider theory of best approximation. First we recall classical results on existence, uniqueness and characterization. Then, based on geometrical constants of the underlying Banach space, we develop two novel a~priori bounds for best approximations (Propositions~\ref{prop:ba_apriori_bound} and~\ref{prop:improvedApriori}), which are of independent interest.
\subsubsection{Existence, uniqueness and characterization}
Best approximation in Banach spaces is founded on the following classical result.
% The following is the main classical result for best approximation in Banach spaces.
% Note that part~(iii) provides an equivalent characterization in term of the duality mapping studied in Section~\ref{sec:duality_mapping}.
%
\begin{theorem}[Best approximation]\label{thm:minimizer_characterization}
Let $\mbbY$ be a Banach space, and $y\in\mbbY$.
\begin{enumerate}[(i)]
\item Suppose $\mbbM \subset\mbbY$ is a finite-dimensional subspace, then there exists a best approximation $y_0\in \mbbM$ to~$y$ such that 
\begin{equation*}
% label{eq:minimization}
 \norm{y-y_0}_\mbbY =\min_{z_0\in\mbbM}\norm{y-z_0}_{\mbbY}\;. 
\end{equation*}
\item Suppose $\mbbM \subset\mbbY$ is any subspace and~$\mbbY$ is strictly convex, then a best approximation $y_0\in \mbbM$ to~$y$ 
%in~\eqref{eq:minimization} 
is unique.
\item Suppose $\mbbM \subset\mbbY$ is a closed subspace, then the following statements are equivalent:
\begin{itemize}
\item $\ds{ y_0=\arg\min_{z_0\in\mbbM}\|y-z_0\|_\mbbY }$.
\item There exists a functional $y^\ast\in\mathcal{J}_\mbbY(y-y_0)$ which annihilates~$\mbbM$, i.e., 
$\<y^\ast,z_0\>_{\mbbY^\ast,\mbbY}=0$, for all $z_0\in \mbbM$, where $\mathcal{J}_\mbbY : \mbbY \rightarrow \mbbY^*$ is the duality mapping defined in Definition~\ref{def:Dmap}.
\end{itemize}
\end{enumerate}
%Let $\mbbY$ be a Banach space and let $\mbbM \subset\mbbY$ be a finite-dimensional subspace. For any $y\in\mbbY$, 
%%
%Additionally, if the Banach space $\mbbY$ is strictly convex, then $y_0$ is unique.
\end{theorem}
\begin{proof}
For parts (i) and (ii) see, e.g., Stakgold~\& Holst~\cite[Section~10.2]{StaHolBOOK2011} or DeVore~\& Lorentz~\cite[Chapter~3]{DevLorBOOK1993}. For part~(iii) in case of~$y\in \mbbY\setminus \mbbM$ see, e.g., Singer~\cite{SinBOOK1970} or Braess~\cite{BraBOOK1986}. The case of~$y\in \mbbM$ is trivial, because in that case~$y_0 = y$ and one can choose $y^*= 0$.
\end{proof} 
%
%-----------------------------------------------------------------------------%
\subsubsection{Banach--Mazur constant and nonlinear projector estimate}
% The is an elementary \emph{a priori estimate} for a best approximation. 
%
To describe the first of two novel a~priori bounds for best approximations, we recall the \emph{Banach--Mazur constant}~\cite[Definition~2]{SteNM2015}. The Banach--Mazur constant is based on the classical \emph{Banach--Mazur distance}, whose motivation is best described by Banach himself~\cite[p.~189]{BanBOOK2009}:
\begin{quote} 
\emph{``A Banach space~$X$ is isometrically isomorphic to a Hilbert space if and only if every two-dimensional subspace of~$X$ is isometric to a Hilbert space.''}
\end{quote}
%
%-----------------------------------------------------------------------------%
\begin{definition}[Banach--Mazur constant]
\label{def:Banach-Mazur}
Let $\mathbb Y$ be a normed vector space with~$\dim \mbbY \ge 2$, and let $\ell_2(\mathbb{R}^2)$ be the 2-D Euclidean space endowed with the $2$-norm. The Banach--Mazur constant of $\mathbb Y$ is defined by
\begin{alignat*}{2}
  \CBM(\mbbY) := \sup 
  \Big\{ \big(\dBM(\mathbb W,\ell_2(\realset^2))\big)^2 \,:\, \mathbb W \subset \mbbY\,,\, \dim \mathbb W = 2 
  \Big\}\,,
\end{alignat*}
where $\dBM(\cdot,\cdot)$ is the (multiplicative) Banach--Mazur distance:
\begin{alignat*}{2}
  \dBM(\mathbb W,\ell_2(\realset^2) ) := 
  \inf \Big\{  
  \norm{T} \norm{T^{-1}} \,:\, T \text{ is a linear isomorphism\footnotemark{} } \mathbb W
  \rightarrow \ell_2(\realset^2)
  \Big\}.
\end{alignat*}
\end{definition}
\footnotetext{i.e., $T$ is a linear bounded bijective operator.}
Since the definition only makes sense when~$\dim \mbbY \ge 2$, henceforth, whenever~$\CBM(\cdot)$ is written, we assume this to be the case. (Note that~$\dim \mbbY = 1$ is often an uninteresting trivial situation.)
% we silently assume we set $\CBM(\mbbY) = 1$ for such situations.
%-----------------------------------------------------------------------------%
%
%-----------------------------------------------------------------------------%
\begin{remark}[Elementary properties of~$\CBM$]\label{rem:BM}
It is known that $1\le \CBM(\mbbY)\le 2$, $\CBM(\mbbY) = 1$ if and only if~$\mbbY$ is a Hilbert space, and $\CBM(\mbbY) = 2$ if~$\mbbY$ is non-reflexive; see~\cite[Section~3]{SteNM2015}. In particular, for $\mbbY = \ell_p(\mathbb{R}^2)$, $\CBM(\mbbY) = 2^{|\frac{2}{p}-1|}$; cf.~\cite[Section~II.E.8]{WojBOOK1991} and \cite[Section~8]{JohLinBOOK-CH2001}. This result is also true for $L^p$ and Sobolev spaces~$W^{k,p}$ ($k\in \mathbb{N}$), see~\cite[Section~5]{SteNM2015}.
\end{remark}
%-----------------------------------------------------------------------------%
%
\par
The Banach--Mazur constant is used in the Lemma below to state a fundamental estimate for an abstract nonlinear projector. This nonlinear projector estimate, which is an extension of Kato's identity $\norm{I-P} = \norm{P}$ for Hilbert-space projectors~\cite{KatNM1960}, and a generalization of the estimate obtained by Stern~\cite[Theorem~3]{SteNM2015} for linear Banach-space projectors, will be used to prove the a~priori bound in Proposition~\ref{prop:ba_apriori_bound} and also Corollary~\ref{col:NPGbound} in Section~\ref{sec:practical}.
\begin{lemma}[Nonlinear projector estimate]
\label{lem:I-P}
Let $\mbbY$ be a normed space, $I:\mbbY\to\mbbY$ the identity and $Q:\mbbY\to\mbbY$ a nonlinear operator such that: 
\begin{enumerate}[(i)] 
\item \label{NP:P}
$Q$ is a nontrivial projector: $0\neq Q=Q\circ Q \neq I$\,.
\item \label{NP:homog} 
$Q$ is homogeneous: $Q(\lambda y)=\lambda Q(y)$, $\quad\forall y\in\mbbY$ and $\forall \lambda\in\mathbb R$\,. 
\item \label{NP:bounded} 
$Q$ is bounded in the sense that $\|Q\|:=\displaystyle\sup_{y\in\mbbY\setminus\{0\}}
\frac{\|Q(y)\|_\mbbY}{\|y\|_\mbbY}<+\infty$\,.
\item \label{NP:GOP}
$Q$ is a generalized orthogonal projector in the sense that 
\begin{alignat*}{2}
  Q(y) = Q\Big( Q(y)  + \eta\, (I-Q)(y)\Big)  \,,
  \qquad \text{for any } \eta \in \mbbR \text{ and any } y\in \mbbY \,.
\end{alignat*}
%
% $\|Q(y)\|_\mbbY\leq \|Q\| \left\|Q(y)+{\eta}(y-Q(y)) \right\|_\mbbY$, for any $y\in\mbbY$ and any $\eta\in\mbbR$.
\end{enumerate}
Then the nonlinear operator $I-Q$ is also bounded and satisfies 
\begin{alignat*}{2}
 \norm{I-Q} \le C_{\mathrm{S}} \norm{Q},
\end{alignat*}
where $C_{\mathrm{S}}$ is the constant introduced by Stern~\cite{SteNM2015}:
\begin{alignat}{2}
\label{eq:CS}
  C_{\mathrm{S}} = \min\Big\{ 1+ \norm{Q}^{-1} , \CBM(\mbbY)\Big\}.
\end{alignat}
% (See Definition~\ref{def:Banach-Mazur} for the Banach--Mazur constant $\CBM(\mbbY)$).
%
\end{lemma}
\begin{proof}
The proof of this result follows closely Stern~\cite[Proof of Theorem~3]{SteNM2015}. Although Stern considers linear projectors, his result generalizes to projectors with the properties in~(i)--(iv). See Section~\ref{sec:proof_lemma_I-P} for the complete proof. 
%It is similar to the proof of Proposition~\ref{prop:ba_apriori_bound}.
\end{proof}
\begin{remark}[Generalized orthogonal projectors]
Requirement~(\ref{NP:GOP}) in Lemma~\ref{lem:I-P} is a key nonlinear property. We point out that it is satisfied by linear projectors, by best-approximation projectors, by $I$~minus best-approximation projectors (as in the proof of Proposition~\ref{prop:ba_apriori_bound}), and by (inexact) nonlinear Petrov-Galerkin projectors~$P_n$ of Definition~\ref{def:nlpgp} (see Corollary~\ref{col:NPGbound}).
\end{remark}
%-----------------------------------------------------------------------------%
\subsubsection{A~priori bound~I}%-----------------------------------------------------------------------------%
The first a priori bound for best approximations is obtained by applying Lemma~\ref{lem:I-P}. 
%
%The following Lemma establish an apriori estimate for best approximations in closed subspaces by means of the Banach-Mazur constant. This Lemma is of independent interest and its proof is given in Section~\ref{sec:proof_lemma_ba}.
%
%-----------------------------------------------------------------------------%
\begin{proposition}[Best approximation: A priori bound~I]\label{prop:ba_apriori_bound}
Let $\mbbY$ be a Banach space and $\mbbM \subset\mbbY$ a closed subspace. Suppose $y_0\in \mbbM$ is a best approximation in $\mathbb M$ of a given $y\in\mathbb Y$ (i.e., $\|y-y_0\|_\mbbY\leq \|y-z_0\|_\mbbY$, for all $z_0\in\mbbM$), then $y_0$ satisfies the a priori bound:
\begin{equation}\label{eq:new_sharpen_estimate}
\norm{y_0}_\mathbb Y \leq \CBM(\mathbb Y) \norm{y}_\mathbb Y\,\,,
\end{equation}
where $\CBM(\mbbY)$ is the Banach-Mazur constant of the space $\mbbY$ (see Definition~\ref{def:Banach-Mazur}).~
\end{proposition}
%-----------------------------------------------------------------------------%
%
%-----------------------------------------------------------------------------%
\begin{proof}
We assume that $\mbbM\neq \{0\}$ and $\mbbM\neq\mbbY$ (otherwise the result is trivial). Consider a (nonlinear) map $P^\perp:\mbbY\to\mbbY$ such that $P^\perp(y)=y-y_0$, where $y_0\in\mbbM$ is a best approximation to $y\in\mbbY$. 
The map $P^\perp$ can be chosen in a homogeneous way, i.e., satisfying $\lambda P^\perp(y)=P^\perp(\lambda y)$ for any $\lambda\in \mathbb R$.
Observe that 
$$\|P^\perp(y)\|_\mbbY=\|y-y_0\|_\mbbY\leq \|y-0\|_\mbbY=\|y\|_\mbbY\,.$$
Hence, $P^\perp$ is bounded with $\|P^\perp\|\leq 1$.
Additionally, it can be verified that $P^\perp(P^\perp(y))=y-y_0-0=P^\perp(y)$. 
Thus, $Q=P^\perp$ satisfies the requirements~(\ref{NP:P}), (\ref{NP:homog}) and (\ref{NP:bounded}) of Lemma~\ref{lem:I-P} with $\|P^\perp\|=1$. To verify requirement~(\ref{NP:GOP}), notice that % $y-P^\perp(y)=y_0$. Hence, 
for any $\eta\in\mbbR$,
\begin{alignat*}{2}
  P^\perp\Big( P^\perp(y) + \eta \, (I-P^\perp)(y) \Big) = 
  P^\perp \Big( y-y_0 + \eta y_0 \Big) = y-y_0\,,
\end{alignat*}
since $\eta y_0$ is a best approximation in~$\mbbM$ to~$y-y_0 + \eta y_0 $.
%, indeed
%%
%\begin{alignat*}{2}
%  \norm{y-y_0}_\mbbY 
%  =  \norm{y-y_0 + \eta y_0 - \eta y_0}_\mbbY 
%  \le 
%  \bignorm{y-y_0 + \eta y_0 - w_0}_\mbbY 
%  \qquad \forall w_0 \in \mbbM\,.
%\end{alignat*}
%%
%  we have:
%
%\begin{alignat}{2}
%\|P^\perp(y)\|_\mbbY= \|y-y_0\|_\mbbY\leq\|y-y_0+\eta y_0\|_\mbbY= & \|P^\perp(y)+\eta(y- P^\perp(y))\|_\mbbY\nonumber\\
%= & \|P^\perp\| \|P^\perp(y)+\eta(y- P^\perp(y))\|_\mbbY.\nonumber
%\end{alignat}
%
Therefore, by Lemma~\ref{lem:I-P} we get:
\begin{alignat*}{2}
\|y_0\|_\mbbY=\bignorm{ (I-P^\perp)y}_\mbbY
\leq\min\Big\{1+\|P^\perp\|^{-1}\,,\,\CBM(\mathbb Y)\Big\}\|P^\perp\| \|y\|_\mbbY \,,
\end{alignat*}
and~\eqref{eq:new_sharpen_estimate} follows since $\|P^\perp\| = 1$ and $\CBM(\mathbb Y) \le 2$\,.
\end{proof}
%-----------------------------------------------------------------------------%
%
%-----------------------------------------------------------------------------%
%
\begin{remark}[Sharpness of~\eqref{eq:new_sharpen_estimate}]
Bound~\eqref{eq:new_sharpen_estimate} improves the \emph{classical bound} $\norm{y_0}_\mbbY\le 2\norm{y}_\mbbY$ (see, e.g.,~\cite[Sec.~10.2]{StaHolBOOK2011}), in the sense that it shows an explicit dependence on the geometry of the underlying Banach space. In particular, \eqref{eq:new_sharpen_estimate}~contains the standard result $\norm{y_0}_\mbbY\le \norm{y}_\mbbY$ for a Hilbert space, as well as the classical bound $\norm{y_0}_\mbbY\le 2\norm{y}_\mbbY$ for non-reflexive spaces such as~$\ell_1(\mathbb{R}^2)$ and $\ell_\infty(\mathbb{R}^2)$ (for which the bound is indeed sharp; see Example~\ref{ex:ell1}). However,~\eqref{eq:new_sharpen_estimate} need not be sharp for intermediate spaces; see Example~\ref{ex:CAOlp}.
%  This is illustrated in Figure~\ref{fig:bounds} on page~\pageref{fig:bounds} for~$\ell_p(\mathbb{R}^2)$, $1<p<\infty$. 
\end{remark}
\begin{example}[$\ell_1(\mathbb{R}^2)$]
\label{ex:ell1}
% Examples of non-uniqueness of the best approxmation and sharpness of~\eqref{eq:new_sharpen_estimate} can be constructed when the involved Banach space is not strictly convex. For example, 
In $\rr^2$ with the norm 
$\norm{(x_1,x_2)}_1=|x_1|+|x_2|$, i.e. $\mbbY = \ell_1(\mathbb{R}^2)$, the best approximation of the point $(1,0)$ over the line $\{(t,t) : t\in\rr\}$ is the whole segment $\{(t,t) : t\in[0,1]\}$. Moreover, the point $(1,1)$ is a best approximation and $\norm{(1,1)}_1=2=2\norm{(0,1)}_1$. Since the Banach--Mazur constant for this case equals~2, Eq.~\eqref{eq:new_sharpen_estimate}~is sharp for this example. 
\end{example}
%%%%%%%%%%%%%%%%%%%%%%%%%%%%%%%%%%

\subsubsection{Asymmetric-orthogonality constant}
% \subsubsection{Geometrical constant and a~priori bound~II}
%
We now construct an alternative a priori bound for best approximations (compare with Proposition~\ref{prop:ba_apriori_bound}). This bound is also a novel result, which is of independent interest. The describe the bound, we introduce the following new geometric constant.
\begin{definition}[Asymmetric-orthogonality constant]\label{def:AOconstant}
Let $\mbbY$ be a normed vector space with~$\dim \mbbY \ge 2$. The asymmetric-orthogonality constant is defined by:
\begin{alignat}{2}\label{eq:Lambda}
\CAO(\mbbY):= \sup_{
\substack{
  (z_0,z)\in \mathcal O_{\mbbY}
  \\
  z_0^*\in \mathcal J_\mbbY(z_0)
}}
   \frac{\dual{z_0^*,z}_{\mbbY^*,\mbbY}}
  {\norm{z}_\mbbY\norm{z_0}_\mbbY}\,,
\end{alignat} 
where the above supremum is taken over the set~$\mathcal O_\mbbY$ consisting of all pairs~$(z_0,z)$ which are \emph{orthogonal} in the following sense : 
\begin{alignat}{2}
\label{eq:orthoset}
\mathcal O_{\mbbY}:=\Big\{(z_0,z)\in \mbbY\times\mbbY : 
\exists\, z^*\in\mathcal J_\mbbY(z) \hbox{ satisfying }  \dual{z^*,z_0}_{\mbbY^*,\mbbY}=0\Big\}. 
\end{alignat}
\end{definition}
As in the case of~$\CBM(\mbbY)$, $\CAO(\mbbY)$ only makes sense when~$\dim \mbbY \ge 2$. Therefore as before, whenever~$\CAO(\cdot)$ is written, we assume this to be the case.
\par
\begin{remark}[Elementary properties of $\CAO$]
\label{rem:LambdaV}
The constant~$\CAO(\mbbY)$ is a \emph{geometric} constant since, it measures the degree to which the orthogonality relation~\eqref{eq:orthoset} fails to be symmetric. %That is, it measures the degree to which  $\dual{z_0^*,z}_{\mbbY^*,\mbbY}$ fails to be zero, for every $z_0^*\in\mathcal J_\mbbY(z_0)$, whenever there is a $z^*\in \mathcal J_\mbbY(z)$ such that $\dual{z^*,z_0}_{\mbbY^*,\mbbY}=0$. 
Using the Cauchy--Schwartz inequality it is easy to see that $0\leq \CAO(\mbbY)\leq 1$. 
If~$\mbbY$ is a Hilbert space, then $\CAO(\mbbY)=0$, since the single-valued duality map $J_\mbbY(\cdot)$ coincides with the self-adjoint Riesz map, and $\dual{J_\mbbY(\cdot),\cdot}_{\mbbY^*,\mbbY}$ coincides with the (symmetric) inner product in~$\mbbY$. On the other hand, the maximal value $\CAO(\mbbY) = 1$ holds for example for $\mbbY = \ell_1(\realset^2)$. Indeed taking ${z_0=(1,-1)}$ and $z = (\alpha,1)$, with $\alpha>0$, then $(2,-2)\in \mathcal J_\mbbY(z_0)$ and ${(1+\alpha,1+\alpha)\in\mathcal J_\mbbY(z)}$, so that upon taking $\alpha \rightarrow +\infty$ one obtains $\dual{z_0^*,z}_{\mbbY^*,\mbbY} / (\norm{z_0}_{\mbbY} \norm{z}_{\mbbY} ) \rightarrow 1$.
\end{remark}
\begin{example}[$\CAO(\ell_p)$]
\label{ex:CAOlp}
Consider the Banach space $\ell_p \equiv \ell_p(\mathbb R^2)$ with $1<p<+\infty$ (i.e., $\mbbR^2$ endowed with the $p$-norm). In this case the duality map is given by:
$$
\bigdual{J_{\ell^p}(x_1,x_2),(y_1,y_2)}_{(\ell_p)^*,\ell_p}=
\bignorm{(x_1,x_2)}_{\ell_p}^{2-p}
\sum_{i=1}^2|x_i|^{p-1}\sign(x_i)\,y_i\,,
$$
for all $(x_1,x_2),(y_1,y_2)\in\mathbb R^2$. 
By homogeneity of the duality map, the supremum in definition \eqref{eq:Lambda} can be taken over a normalized set (with unitary elements), in which case the computation of  $\CAO({\ell_p})$ is derived from the following constrained maximization problem:
\begin{alignat*}{2}
\max \quad  |x_1|^{p-2}x_1y_1 + |x_2|^{p-2}x_2y_2\,,
\quad 
\text{subject to } & \begin{cases}
|x_1|^p+|x_2|^p=1\,,\\
 |y_1|^p+|y_2|^p=1\,,\\
 |y_1|^{p-2}x_1y_1 + |y_2|^{p-2}x_2y_2=0\,.
 \end{cases}
\end{alignat*}
%
%\marginnote{Needs more details linking step from $J_{\ell^p}$ to $\Lambda_p$}
Using polar coordinates, the above constraints, and some symmetries, it is possible to reduce the above problem to the following one dimensional maximization:
$$
\CAO({\ell_p})=\max_{\theta \in [0,{\pi\over2}]}
{\big|(\cos\theta)^{p\over q}(\sin\theta)^{q\over p}-(\sin\theta)^{p\over q}(\cos\theta)^{q\over p}\big|\over\|(\cos\theta,\sin\theta)\|_p^{p\over q}\|(\cos\theta,\sin\theta)\|_q^{q\over p}}\,,
$$
where $q={p\over p-1}$. 
% \marginnote{Isn't it valid for $p=1$ and $\infty$?}
Observe that $\CAO({\ell_p})=\CAO({\ell_q})$ since the formula remains the same by switching the roles of $p$ and $q$ (cf.~Lemma~\ref{lem:Lambda_properties}). One can also show that $\CAO({\ell_1})=\CAO({\ell_\infty})=1$. 
Figure~\ref{fig:bounds} shows the dependence of~$\CAO({\ell_p})$ versus~$p-1$. It also illustrates the Banach--Mazur constant $\CBM({\ell_p})$ and the best-approximation projection constant $C_{\mathrm{best}}({\ell_p}):= \max_{u\in \ell_p(\mathbb{R}^2)} \norm{u_n}/\norm{u}$, with $u_n$ the best approximation to~$u$ on the worst 1-dimensional subspace of~$\ell_p(\mathbb{R}^2)$. The figure shows that
\begin{alignat*}{2} 
 C_{\mathrm{best}}(\ell_p) < \CBM(\ell_p) <  1 + \CAO({\ell_p})
\end{alignat*}
except for~$p=1$, $2$ and~$+\infty$, for which they coincide. 
\end{example}
\begin{figure}
\begin{center}
{\small
\psfrag{1+Lamb}{\scriptsize $1{+}\CAO({\ell_p})$}
\psfrag{CBM}{\scriptsize $\CBM(\ell_p)$}
\psfrag{CLp}{\scriptsize $C_{\mathrm{best}}(\ell_p)$}
\psfrag{p-1}{$p{-}1$}
\includegraphics[scale=1]{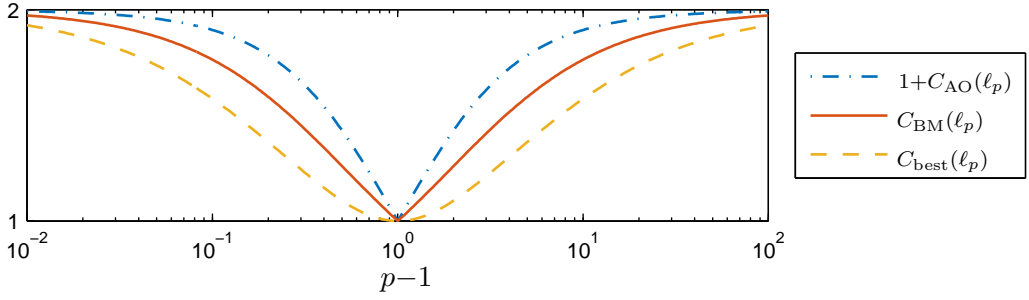}
}
\caption{Three different geometric constants and its dependence on $p-1$.}%This figure plots versus~$(p - 1)$ the value of the following constants for the space~$\mbbY = \ell_p(\mathbb{R}^2)$: The Banach--Mazur constant $\CBM(\ell_p)$, the constant $1{+}\Lambda_{\ell_p}$, and the best-approximation projection constant~$C_{\mathrm{best}}(\ell_p):= \max_{u\in \ell_p(\mathbb{R}^2)} \norm{u_n}/\norm{u}$, with $u_n$ the best approximation to~$u$ on the worst 1-dimensional subspace of~$\mbbY$.}
\label{fig:bounds}
\end{center}
\end{figure}
\par
We conclude our discussion of~$\CAO$ with a Lemma
%with the next Lemma~\ref{lem:Lambda_properties} is about two 
describing three important properties that are going to be used later in Section~\ref{sec:apriori}. 
\begin{lemma}[$\CAO$ in reflexive smooth setting]
\label{lem:Lambda_properties}
Assume the reflexive smooth setting where $\mbbY$ and $\mbbY^*$ are strictly convex and reflexive Banach spaces (see Section~\ref{sec:dualMapSmooth}). The following properties hold true:
\begin{enumerate}[(i)]
\item 
$\CAO(\mbbY)=\displaystyle\sup_{(z_0,z)\in \mathcal O_{\mbbY}}
   \frac{\dual{J_\mbbY(z_0),z}_{\mbbY^*,\mbbY}}
  {\norm{z}_\mbbY\norm{z_0}_\mbbY}\,$,
\par where $\mathcal O_{\mbbY}=\big\{(z_0,z)\in\mbbY\times\mbbY:\dual{J_\mbbY(z),z_0}_{\mbbY^*,\mbbY}=0 \big\}$.
\item $\CAO({\mbbY^*})=\CAO(\mbbY)\,$.
\item $\CAO(\mbbM)\leq \CAO(\mbbY)\,$, for any closed subspace $\mbbM\subset\mbbY$ with~$\dim \mbbM \ge 2$ endowed with the norm~$\norm{\cdot}_\mbbY$. 
\end{enumerate}
\end{lemma}
\begin{proof}
See Section~\ref{sec:CAO}.
\end{proof}
Note that result~(ii) and (iii) in Lemma~\ref{lem:Lambda_properties} actually imply that 
\begin{alignat}{2}
\label{eq:CAOM=CAOY}
  \CAO(\mbbM) = \CAO(\mbbY)\,,
\end{alignat}
because $\CAO(\mbbY) = \CAO(\mbbY^*) \le \CAO(\mbbM^*) = \CAO(\mbbM) \le \CAO(\mbbY)$.
\begin{example}[$\CAO(L^p)$]
Let $\Omega\subset\mathbb R^d$ be an open set and consider the Banach space $\mathbb Y:=L^p(\Omega)$, $1<p<+\infty$. Let $\Omega_1$ and $\Omega_2$ be two open bounded disjoint subsets. 
Define the functions $f_i\in L^p(\Omega)$ ($i=1,2$) by 
$f_i:=|\Omega_i|^{-{1\over p}}\mathbbm 1_{\Omega_i}$ and let 
$\mathbb M:=\hbox{span}\{f_1,f_2\}\subset\mathbb Y$. 
It is easy to see that $\mathbb M$ is isometrically isomorphic to
$\ell_p(\mathbb R^2)$ and thus, using~\eqref{eq:CAOM=CAOY}, we have
\begin{alignat*}{2}
\CAO({\ell_p})= \CAO(\mathbb M) = \CAO({L^p})\,.
\end{alignat*}
\end{example}
\subsubsection{A~priori bound~II}
The second a~priori bound for best approximations is based on the equivalent characterization given in Theorem~\ref{thm:minimizer_characterization} and the asymmetric-orthogonality constant.
\begin{proposition}[Best approximation: A priori bound~II]
\label{prop:improvedApriori}
Let $\mbbY$ be a Banach space, $y\in\mbbY$ and $\mbbM\subset\mbbY$ a closed subspace. Let $y_0\in \mbbM$ be such that $\|y-y_0\|_\mbbY\leq \|y-z_0\|_\mbbY$, for all $z_0\in\mbbM$. Then $y_0$ satisfies the a priori bound: 
\begin{alignat}{2}
\|y_0\|_\mbbY\leq\big(1+\CAO(\mbbY)\big)\|y\|_\mbbY\,,
\end{alignat} 
where $\CAO(\mbbY)\in[0,1]$ is the {asymmetric-orthogonality} constant of $\mbbY$ (see Definition~\ref{def:AOconstant}).  
\end{proposition}
\begin{proof} If $y_0=0$ or $y_0=y$, then the result is obvious. Hence, let us assume that $\|y_0\|_\mbbY>0$ and $\|y-y_0\|_\mbbY>0$. First of all, we estimate the error using the minimizing property of $y_0\in\mbbM$ :
\begin{alignat}{2}\label{eq:estRes}
\|y-y_0\|_\mbbY\leq \|y-0\|_\mbbY=\|y\|_\mbbY\,.
%\leq(1+\Lambda)\|y\|_\mbbY\,, 
\end{alignat}
Next, by Theorem~\ref{thm:minimizer_characterization}, there exists a~$z^*\in \mathcal J_\mbbY(y-y_0)$ which annihilates~$\mbbM$. Therefore, $(y_0,y-y_0)\in \mathcal O_{\mbbY}$, and we thus obtain for any $z_0^*\in\mathcal J_\mbbY(y_0)$: 
\begin{alignat*}{2}
\|y_0\|_\mbbY &= \displaystyle{\dual{z_0^*,y_0}_{\mbbY^*,\mbbY}\over \|y_0\|_\mbbY}
\\ & =
{\dual{z_0^*,y}_{\mbbY^*,\mbbY}\over \|y_0\|_\mbbY}
-{\dual{z_0^*,y-y_0}_{\mbbY^*,\mbbY}\over \|y_0\|_\mbbY}
\\ &\le
 \|y\|_\mbbY 
 -{\dual{z_0^*,y-y_0}_{\mbbY^*,\mbbY}\over \|y_0\|_\mbbY \|y-y_0\|_\mbbY}
 \|y-y_0\|_\mbbY 
\\ &\le  \|y\|_\mbbY + \CAO(\mbbY) \|y-y_0\|_{\mbbY}\,.
%\leq(1+\Lambda)\|y\|_\mbbY\,, 
\end{alignat*}
Conclude by using the estimate in~\eqref{eq:estRes}.
\end{proof}
%\begin{example}[$\ell_1$]
%
%\marginnote{$\alpha>0$}
%\end{example}
%\par
%
%-----------------------------------------------------------------------------%
\subsection{Proof of Lemma~\ref{lem:I-P}}\label{sec:proof_lemma_I-P}
%-----------------------------------------------------------------------------%
% In this section, we prove Lemma~\ref{lem:I-P} by closely following Stern~\cite[Proof of Theorem~3]{SteNM2015}. Contrary to Stern's requirement, the projector~$Q$ does not need to be linear, but only requires the properties stated in the Lemma~\ref{lem:I-P}.
In this section, we prove Lemma~\ref{lem:I-P}.
\par
The inequality $\|I-Q\|\leq 1+\|Q\|=(1+\|Q\|^{-1})\|Q\|$ is trivial, so we focus our attention in showing that 
$$
\|y-Q(y)\|_\mbbY\leq C_{\hbox{\tiny BM}}(\mbbY)\,\|Q\|\,\|y\|_\mbbY\,,
\qquad\forall y\in\mbbY.
$$
If $y-Q(y)=0$ the result holds true immediately.
On the other hand, by requirement~(\ref{NP:P}), since $0\neq Q=Q\circ Q$, we have that $\|Q\|\geq 1$. Moreover, $C_{\hbox{\tiny BM}}(\mbbY)\geq 1$ (see Remark~\ref{rem:BM}).  Thus, if $Q(y)=0$, then 
$$  
\|y-Q(y)\|_\mbbY =\|y\|_\mbbY \leq C_{\hbox{\tiny BM}}(\mbbY)\,\|Q\|\,\|y\|_\mbbY.
$$
Hence, we assume form now on that $y-Q(y)\neq 0$ and $Q(y)\neq 0$. 
\par
First of all observe that $y-Q(y)$ and $Q(y)$ are linearly independent. Indeed, suppose on the contrary that there exists $t\in\mathbb R \setminus \{0\}$ such that $y-Q(y)=t Q(y)$, then $y=(1+t)Q(y)$, hence applying~$Q$ and using  homogeneity (requirement~(\ref{NP:homog})), we get $t=0$ (a contradiction). 
\par
The proof follows next using a two-dimensional geometrical argument. Let us define $\mathbb W:=\hbox{span}\{Q(y),y-Q(y)\}$ and note that $\dim \mathbb W = 2$. Let $T:\mathbb W  \rightarrow \ell_2(\realset^2)$ be any linear isomorphism between $\mathbb{W}$ and $\ell_2(\realset^2)$ (the two-dimensional Euclidean vector space endowed with the norm~$\norm{\cdot}_2$). Define 
\begin{subequations}
\label{eq:alpha_beta}
\begin{alignat}{2}\label{eq:alpha_beta_a}
  0 \neq \alpha &:= \norm{T(y-Q(y))}_2 \,,
  \\ \label{eq:alpha_beta_b}
%   \text{and} \quad
  0 \neq \beta &:= \norm{T Q(y) }_2\,,
\end{alignat}
\end{subequations}
and, subsequently, let~$\tilde{y}\in \mbbW$ be defined by
\begin{alignat}{2}\label{eq:tilde x}
  \tilde{y} := \frac{\alpha}{\beta} Q(y)+\frac{\beta}{\alpha} (y-Q(y)) \,.
\end{alignat}
The proof will next be divided into four steps:
\begin{itemize}
\item[(S1)] To show that $\norm{y-Q(y)}_\mathbb Y\leq \big(\|T\|\|T^{-1}\|\big)\,\norm{\frac{\alpha}{\beta} Q(y)}_\mbbY$\,\,.
%$\norm{y-Q(y)}_\mathbb Y\leq \big(\|T\|\|T^{-1}\|\big)\,\|Q\|\, \norm{\tilde y}_\mathbb Y$\,\,.
\item[(S2)] To show that $\norm{\frac{\alpha}{\beta} Q(y)}_{\mbbY}\le \norm{Q}\norm{\tilde{y}}_\mbbY$\,\,.
\item[(S3)] To show that $\norm{\tilde y}_\mathbb Y\leq \big(\|T\|\|T^{-1}\|\big)\, \norm{y}_\mathbb Y$\,\,.
\item[(S4)] To conclude that $\norm{y-Q(y)}_\mathbb Y\leq \CBM(\mathbb Y)\, \|Q\|\,\norm{y}_\mathbb Y$\,\,.
\end{itemize}
\par 
Proof of~(S1):\; This follows from elementary arguments since~$\beta \neq 0$:
\begin{alignat*}{2}
 \norm{y-Q(y)}_\mbbY &\le \norm{T^{-1}} \norm{T(y-Q(y))}_2
\\\tag{by~\eqref{eq:alpha_beta_a}}
 &= \norm{T^{-1}} \, {\alpha}
 \\\tag{by~\eqref{eq:alpha_beta_b}}
 &= \norm{T^{-1}} \, \frac{\alpha}{\beta} \norm{TQ(y)}_2
  \\
 &\leq \norm{T^{-1}} \,\|T\| \Big|\frac{\alpha}{\beta} Q(y)\Big|_\mbbY
 \,.
%\\\tag{by~\eqref{eq:P(tilde x)}}
% & \leq \norm{T^{-1}} \norm{T} \, \|Q\| \,\|\tilde y\|_\mbbY\,\,.
\end{alignat*}
\par
Proof of~(S2):\;
Use requirement~(\ref{NP:GOP}) with~$\eta = \frac{\beta^2}{\alpha^2}$, and subsequently requirements~(\ref{NP:homog}) and~(\ref{NP:bounded}), to obtain:
\begin{alignat}{3}\label{eq:P(tilde x)}
\Bignorm{ \frac{\alpha}{\beta} Q(y) }_\mbbY
= \Bignorm{
\frac{\alpha}{\beta} Q\Big(
 Q(y) + \frac{\beta^2}{\alpha^2} (I-Q)(y) 
\Big)
}_\mbbY
= \bignorm{Q(\tilde{y})}_\mbbY
\leq \norm{Q} \norm{\tilde{y}}_\mbbY\,.
% \left\| \frac{\alpha}{\beta} Q(y)+\frac{\beta}{\alpha}(y- Q(y))\right\|_\mbbY
% \|Q\|\|\tilde{y}\|_\mbbY.
\end{alignat}
\par 
Proof of (S3):\; 
The key point here is to observe that $\norm{T \tilde{y}}_2 = \norm{T y}_2$, indeed,
\begin{alignat*}{2}
  \norm{T \tilde{y}}_2^2 
  &= \Bignorm{ \frac{\alpha}{\beta} T Q(y)+\frac{\beta}{\alpha} T (y-Q(y))}_2^2\tag{by~\eqref{eq:tilde x}} 
\\\tag{by~\eqref{eq:alpha_beta}} 
  &=  \alpha^2 + 2\, T Q(y)\cdot T(y-Q(y)) + \beta^2 
\\\tag{by~\eqref{eq:alpha_beta}}
  &= \Bignorm{ T (y-Q(y)) + T Q(y)}_2^2 
\\
  &= \norm{T y}_2^2\,\,.
\end{alignat*}
Therefore,
\begin{alignat*}{2}
\norm{\tilde{y}}_\mbbY
 \le \norm{T^{-1}}\, \norm{T \tilde{y}}_2
 = \norm{T^{-1}}\, \norm{T y}_2 
 \le \norm{T^{-1}}\, \norm{T}\, \norm{y}_\mbbY\,.
\end{alignat*}
\par 
Proof of (S4):\; 
Combining~(S1)--(S3) we get 
\begin{alignat*}{2}
\|y-Q(y)\|_\mbbY\leq\big(\|T\|\|T^{-1}\|\big)^2\,\|Q\|\,\|y\|_\mbbY\,.
\end{alignat*}
Finally, taking the infimum over all linear isomorphisms $T:\mathbb W\to \ell_2(\realset^2)$ we obtain
\begin{alignat*}{2}
\norm{y-Q(y)}_\mbbY\leq \left(\dBM(\mathbb W,\ell_2(\realset^2))\right)^2\,\|Q\|
\,\norm{y}_\mbbY\leq \CBM(\mathbb Y)\,\|Q\| \,\norm{y}_\mathbb Y\,.
\end{alignat*}
~\hfill%
$\ensuremath{_\blacksquare}$

\subsection{Proof of Lemma~\ref{lem:Lambda_properties}}
\label{sec:CAO}
%-----------------------------------------------------------------------------%
%
In this section we prove Lemma~\ref{lem:Lambda_properties}. 
\par
The reflexive smooth setting ensures that the duality mappings are single-valued bijections $J_\mbbY:\mbbY\to\mbbY^*$ and 
$J_{\mbbY^*}:\mbbY^*\to\mbbY^{**}$. Moreover, $J_{\mbbY^*}=J^{-1}_\mbbY$ by canonical identification.
\par Property (i) is a direct consequence of the definition of the constant $\CAO(\mbbY)$ (see~\eqref{eq:Lambda}) and the fact that the duality mapping is single-valued, i.e., $\mcJ_\mbbY(y) = \{J_\mbbY(y)\}$, for all $y\in\mbbY$.
\par To prove property (ii), we make use of property (i) replacing $\mbbY$ by $\mbbY^*$. We get 
$$
\CAO({\mbbY^*})=\sup_{(z^*,z_0^*)\in\mathcal O_{\mbbY^*}}\frac{\dual{J_{\mbbY^*}(z^*),z_0^*}_{\mbbY^{**},\mbbY^*}}
  {\norm{z_0^*}_{\mbbY^*}\norm{z^*}_{\mbbY^*}}
  =\sup_{(z^*,z_0^*)\in\mathcal O_{\mbbY^*}}\frac{\dual{z_0^*,J^{-1}_{\mbbY}(z^*)}_{\mbbY^{*},\mbbY}}
  {\norm{z_0^*}_{\mbbY^*}\norm{z^*}_{\mbbY^*}}.$$
Defining $z=J^{-1}_\mbbY(z^*)$ and $z_0=J^{-1}_\mbbY(z_0^*)$ we obtain
\begin{equation}\label{eq:LambdaDual}
\CAO({\mbbY^*})=\sup_{(z^*,z_0^*)\in\mathcal O_{\mbbY^*}}\frac{\dual{J_{\mbbY}(z_0),z}_{\mbbY^{*},\mbbY}}
  {\norm{z_0}_{\mbbY}\norm{z}_{\mbbY}}.
\end{equation}
Now observe that
\begin{alignat*}{2} 
\mathcal O_{\mbbY^*}&=  \big\{(z^*,z_0^*)\in\mbbY^*\times\mbbY^*:\dual{J_{\mbbY^*}(z_0^*),z^*}_{\mbbY^{**},\mbbY^*}=0 \big\}
\\
&= \big\{(J_\mbbY(z),J_\mbbY(z_0))\in\mbbY^*\times\mbbY^*:\dual{J_{\mbbY}(z),z_0}_{\mbbY^{*},\mbbY}=0 \big\}
\\
&=  \big\{(J_\mbbY(z),J_\mbbY(z_0))\in\mbbY^*\times\mbbY^*:(z_0,z)\in \mathcal O_\mbbY \big\}.
\end{alignat*}
Hence the supremum in~\eqref{eq:LambdaDual} can be taken over all $(z_0,z)\in \mathcal O_\mbbY$ and thus $\CAO({\mbbY^*})=\CAO(\mbbY)\,$.  
\par For the last property (iii) we make use of Lemma~\ref{lem:IJI} to show that
$$
\CAO(\mbbM)=\sup_{(z_0,z)\in\mathcal O_\mbbM}
\frac{\bigdual{J_\mbbM(z_0),z}_{\mbbM^*,\mbbM}}{\|z\|_\mbbY\|\|z_0\|_\mbbY}=\sup_{(z_0,z)\in\mathcal O_\mbbM}
\frac{\bigdual{J_\mbbY(I_\mbbM z_0),I_\mbbM z}_{\mbbY^*,\mbbY}}{\|z\|_\mbbY\|\|z_0\|_\mbbY}\, .
$$
To conclude, it remains to show that $\CAO(\mbbY)$ takes the supremum over a larger set (i.e., $I_\mbbM \mathcal O_\mbbM\subset
\mathcal O_\mbbY$). Indeed, if $(z_0,z)\in\mathcal O_\mbbM$, then $(I_\mbbM z_0,I_\mbbM z)\in \mbbY\times\mbbY$ and 
$$
\bigdual{J_\mbbY(I_\mbbM z),I_\mbbM z_0}_{\mbbY^*,\mbbY}=\bigdual{J_\mbbM( z),  z_0}_{\mbbM^*,\mbbM}=0\,,
$$
by Lemma~\ref{lem:IJI}. Hence $(I_\mbbM z_0,I_\mbbM z)\in\mathcal O_\mbbY$.
~\hfill%
$\ensuremath{_\blacksquare}$

\section{Residual minimization, nonlinear Petrov--Galerkin and monotone mixed formulation}
\label{sec:NPGMMM}
In this section, we analyze the residual minimization method~\eqref{eq:introResMin} and characterize its solution by means of the duality mapping.
%  (see Theorem~\ref{thm:minimizer_characterization}). 
The characterization will give rise to a nonlinear Petrov--Galerkin discretization and corresponding mixed formulation. The inexact version of this method is the subject of Section~\ref{sec:practical}.
%
%-----------------------------------------------------------------------------%
\subsection{Equivalent best-approximation problem}
To carry out the analysis, we reformulate~\eqref{eq:introResMin} as an equivalent best-approximation problem and apply the classical theory of Section~\ref{sec:bestAppr}. 
%
% \par
%
%Given a finite-dimensional subspace $\mbbU_n\subset \mbbU$, we define the approximation $u_n$ by
%%
%\begin{empheq}[left=\left\{\;,right=\right.,box=]{alignat=2} 
%\notag 
%& \text{Find } u_n\in \mbbU_n :
%\\ 
%\label{eq:residual_minimization}
%& \quad u_n = \arg\min_{w_n\in \mbbU_n} \norm{f-Bw_n}_{\mbbV^*}\,\,.
%\end{empheq}
%%
Let us introduce the norm
\begin{alignat*}{2}
  \|\cdot\|_{\mbbE}:=\|B(\cdot)\|_{\mbbV^\ast}\,,
\end{alignat*}
which, in some applications, is referred to as the \emph{energy norm} on $\mbbU$. Since $B$ is continuous and bounded below, it is clear that 
%(see Definition~\ref{def:M_B_gamma_B}),
$\|\cdot\|_{\mbbE}$ is an equivalent norm on~$\mbbU$; see~\eqref{eq:normEquiv}.
% equivalent to the norm $\|\cdot\|_\mbbU$ in $\mbbU$. Indeed,
%%
%\begin{equation}\label{eq:norm_equivalence}
%  \gamma_B \norm{w}_{\mbbU} 
%  \le  \norm{w}_{\mbbE}
%  \le M_B\norm{w}_\mbbU\,,
%  \qquad \forall w\in \mbbU\,,
%\end{equation}
%where $M_B>0$ and $\gamma_B>0$ stand for the \emph{continuity} and \emph{bounded-below} constants of $B$, respectively. 
% 
\par
Let us recall that existence of a unique solution to~\eqref{eq:Bu=f} is guaranteed for continuous and bounded-below~$B$, if $f\in \Image B$ or if~$\Kernel B^* = \{0\}$ (in which case $B$~is surjective); see, e.g.,~\cite[Appendix~A.2]{ErnGueBOOK2004} or~\cite[Section~5.17]{OdeDemBOOK2010}. Therefore, supposing $f\in \Image{B}$, then upon substituting $f=Bu$, \eqref{eq:introResMin} is equivalent to finding a best approximation~$u_n\in \mbbU_n$ to~$u$ measured by the energy norm: 
\begin{empheq}[left=\left\{\;,right=\right.,box=]{alignat=2} 
\notag 
& \text{Find } u_n\in \mbbU_n :
\\ 
\label{eq:energy_minimization}
& \quad u_n = \arg\min_{w_n\in \mbbU_n} \norm{u-w_n}_{\mbbE}\,\,.
\end{empheq}
%
%\par
%%
%In the case that $f\notin \Image B$, the following equivalent best-approximation problem will be useful:
%%
%\begin{empheq}[left=\left\{\;,right=\right.,box=]{alignat=2} 
%\notag 
%& \text{Find } f_n\in B \mbbU_n :
%\\ 
%\label{eq:data_minimization}
%& \quad f_n = \arg\min_{g_n\in B\mbbU_n} \norm{f-g_n}_{\mbbV^*}\,\,,
%\end{empheq}
%%
%and subsequently setting~$u_n= B^{-1}f_n$. This can be thought of as finding the best data-approximation in the subspace~$B\mbbU_n\subset \mbbV^*$ to~$f$.
%-----------------------------------------------------------------------------%
\subsection{Analysis of residual minimization}
The main result for the residual-minimization method~\eqref{eq:introResMin} now follows from the classical Theorem~\ref{thm:minimizer_characterization} for best approximations, while novel a~priori bounds follow from Propositions~\ref{prop:ba_apriori_bound} and~\ref{prop:improvedApriori}.
%
%-----------------------------------------------------------------------------%
\begin{theorem}[Residual minimization]
\label{thm:resMin}
Let $\mbbU$ and $\mbbV$ be two Banach spaces and let $B:\mbbU\to \mbbV^*$ be a linear, continuous and bounded-below operator with continuity constant $M_B>0$ and bounded-below constant $\gamma_B>0$. Given $f\in \mbbV^*$ and a finite-dimensional subspace
$\mbbU_n\subset \mbbU$, the following statements hold: 
\begin{enumerate}[(i)]
\item
There exists a residual minimizer $u_n\in \mbbU_n$ such that: 
\begin{alignat}{2}\label{eq:resmin}
u_n=\arg\min_{w_n\in\mbbU_n}\|f-Bw_n\|_{\mbbV^*}\,\,. %of problem~\eqref{eq:residual_minimization}.
\end{alignat}
\item
Any residual minimizer $u_n$ of \eqref{eq:resmin} satisfies the a~priori bounds 
\begin{subequations}
\begin{alignat}{2}
\label{eq:est1}
\norm{u_n}_\mbbU
& \leq \frac{\CBM(\mathbb V^*)}{\gamma_B} \norm{f}_{\mbbV^*}\,\,,
\\
\label{eq:alternative_apriori_bound}
\norm{u_n}_\mbbU 
&\leq {\big(1+\CAO(\mbbV)\big)\over \gamma_B}
\norm{f}_{\mbbV^*}\,\,.
\end{alignat}
\end{subequations}
where $\CBM(\mbbV^*)\in [1,2]$ is the Banach-Mazur constant of~$\mbbV^*$ and $\CAO(\mbbV)\in [0,1]$ is the asymmetric-orthogonality constant of~$\mbbV$ (see Definitions~\ref{def:Banach-Mazur} and~\ref{def:AOconstant}).
\item 
If $\mbbV^*$ is a strictly-convex Banach space, then the residual minimizer $u_n$ of \eqref{eq:resmin} is unique.
\item
If $f\in \Image(B)$ and $u\in \mbbU$ is the solution of the problem
$Bu=f$, then we have the a~posteriori and a~priori error estimates:
\begin{equation}\label{eq:cea}
\norm{u-u_n}_\mbbU\leq \frac{1}{\gamma_B}\norm{f-Bu_n}_{\mbbV^*}\leq
\frac{M_B}{\gamma_B}\inf_{w_n\in \mbbU_n}\|u-w_n\|_\mbbU.
\end{equation}
\end{enumerate}
\end{theorem}
%-----------------------------------------------------------------------------%
%
%-----------------------------------------------------------------------------%
\begin{proof}
%
%\marginnote{@Ignacio: Is yet another (more direct) proof possible based on equivalence with~\eqref{eq:energy_minimization}?}
%
%
%ADD BANACH SPACE WITH E-NORM.
We first consider the case that $f\in \Image(B)$, in which case $Bu=f$. 
\par
The proof of parts (i), (iii) and (iv) can be found in Guermond~\cite{GueSINUM2004}, but we present an alternative based on Theorem~\ref{thm:minimizer_characterization}. Since $\mathbb U$ endowed with the energy norm is a Banach space, the first statement is a direct application of Theorem~\ref{thm:minimizer_characterization}(i) by using 
the energy norm topology in $\mathbb U$ and the equivalence between \eqref{eq:introResMin} and~\eqref{eq:energy_minimization}. 
%
%We set $\mbbY=\mbbV^*$ and $\mbbY_n=B\mbbU_n$ to apply Lemma~\ref{lemma:finite_best_approximation}. Then there \removed{always }exists $f_n\in B(\mbbU_n)$ such that 
%$$
%\norm{f-f_n}_{\mbbV^*}=\min_{g_n\in B(\mbbU_n)}\norm{f-g_n}_{\mbbV^*}
%=\min_{w_n\in \mbbU_n}\norm{f-Bw_n}_{\mbbV^*}\,\,.$$
%Hence, $u_n=B^{-1}f_n\in \mbbU_n$ is a minimizer of problem~\eqref{eq:residual_minimization}, which proves the first statement.
%
% \par
%
If $\mbbV^*$ is strictly convex, then $\mathbb U$ endowed with the energy norm is also strictly convex. Hence, by \ref{thm:minimizer_characterization}(ii) the minimizer $u_n\in\mathbb U_n$ is unique, which proves the third statement. 
% \par
Finally, using the norm equivalence~\eqref{eq:normEquiv}, together with the minimizing property of $u_n$ in the energy norm, we get
$$
\norm{u-u_n}_\mbbU\leq \frac{1}{\gamma_B}\norm{u-u_n}_{\mbbE}\leq \frac{1}{\gamma_B}\norm{u-w_n}_{\mbbE}\leq \frac{M_B}{\gamma_B}\norm{u-w_n}_\mbbU\,,$$
for all $w_n\in \mbbU_n$, which proves the last statement.
\par
We now prove part~(ii). The bound provided by Proposition~\ref{prop:ba_apriori_bound} shows that
$$\norm{u_n}_\mbbU \leq \frac{1}{\gamma_B}\norm{u_n}_\mbbE 
\le 
\frac{\CBM(\mbb V^*)}{\gamma_B}\norm{u}_\mbbE
= 
\frac{\CBM(\mbb V^*)}{\gamma_B}\norm{f}_{\mbb V^*}\,\,,$$
which proves~\eqref{eq:est1}. A similar argument based on Proposition~\ref{prop:improvedApriori} proves~\eqref{eq:alternative_apriori_bound} (using also Lemma~\ref{lem:Lambda_properties}(ii)).
\par
The proof for general~$f\in \mbbV^*$, including the case that $f\notin \Image(B)$, follows similarly by considering the best approximation of $f$ in the space~$B\mbbU_n$, and using that any~$g_n\in B\mbbU_n$ has a unique~$w_n \in \mbbU_n$ such that $B w_n = g_n$.
\end{proof}
\begin{remark}[Finite element methods]\label{rem:FEM}
In the context of finite elements, there is a \emph{sequence} $\{\mbbU_h\}_{h>0}$ of finite-dimensional subspaces, $\mbbU_h \subset \mbbU$, having the approximation property
$$
\inf_{w_h\in \mbbU_h}\|w-w_h\|_\mbbU\leq \varepsilon(h)\norm{w}_\mbbZ
\,,\qquad
\forall w\in \mbbZ\,,
$$ 
where $\mbbZ\subset \mbbU$ is a more regular subspace and $\varepsilon(h)$ is a function that is continuous at zero and~$\varepsilon(0) = 0$. %
%
%\footnote{Explicit examples are given in Remark~\ref{rem:advFEM} and Remark~\ref{rem:LaplaceFEM}.}
%
This last statement, together with~\eqref{eq:cea}, gives a  guarantee that minimizers $u_n\in \mbbU_n \equiv \mbbU_h$ of~\eqref{eq:resmin} converge to $u=B^{-1}f$ upon $h\rightarrow 0^+$. 
% In other words, this is an example of the dictum: approximability and stability imply convergence.
\end{remark}

%-----------------------------------------------------------------------------%
\begin{remark}[Optimal test-space norm]
\label{rem:Vopt}
As proposed in~\cite{ZitMugDemGopParCalJCP2011} (cf.~\cite{DahHuaSchWelSINUM2012}), if $B$~is a linear bounded bijective operator and~$\mbbV$ reflexive (hence $B^*:\mbbV\rightarrow \mbbU^* $ is bijective), one can endow the space~$\mbbV$ with the equivalent \emph{optimal} norm
\begin{alignat*}{2}
  \norm{\cdot}_{\mbbV_{\mathrm{opt}}} = \norm{ B^*(\cdot) }_{\mbbU^*}
  \,.
\end{alignat*}
Then, residual minimization in $(\mbbV_{\mathrm{opt}})^*$ reduces precisely to best approximation of~$u$ measured in~$\norm{\cdot}_\mbbU$. In particular, instead of~\eqref{eq:cea}, one then obtains
\begin{alignat*}{2}
\norm{u-u_n}_\mbbU
 = \norm{f-Bu_n}_{(\mbbV_{\mathrm{opt}})^*}
 = \inf_{w_n\in \mbbU_n}\|u-w_n\|_\mbbU\, .
\end{alignat*}
\end{remark}
%-----------------------------------------------------------
\subsection{Characterization of residual minimization}
The first characterization for residual minimizers is given in general Banach spaces:%
\begin{proposition}[Characterization of residual minimization]
\label{thm:characterization1}
Let $\mbbU$ and $\mbbV$ be two Banach spaces and let $B:\mbbU \to\mbbV^*$ be a linear, continuous and bounded-below operator. Given $f\in \mbbV^*$ and a finite-dimensional subspace
$\mbbU_n\subset \mbbU$, an element $u_n\in \mbbU_n$ is a solution of the residual minimization problem~\eqref{eq:introResMin}, if and only if there is an~$r^{\ast\ast}\in\mathcal{J}_{\mbbV^*}(f-Bu_n)\subset \mbbV^{\ast\ast}$ satisfying:
\begin{equation}\label{eq:NPG1}
\<r^{\ast\ast}, B w_n\>_{\mbbV^{\ast\ast},\mbbV^\ast}=0,
\quad\forall w_n\in \mbbU_n.
\end{equation}
\end{proposition}
\begin{proof}
Apply Theorem~\ref{thm:minimizer_characterization}(iii) to the minimization problem~\eqref{eq:introResMin}, using $\mbbY=\mbbV^*$ and $\mbbM=B\mbbU_n$.
\end{proof}
\par
Defining the discrete space $\mbbV^*_n:=B\mbbU_n$, we see that \eqref{eq:NPG1} can be interpreted as the nonlinear Petrov--Galerkin discretization: 
\begin{empheq}[left=\left\{\;,right=\right.,box=]{alignat=2} 
\notag 
& \text{Find } u_n\in \mbbU_n \hbox{ such that for some }
r^{\ast\ast}\in \mathcal{J}_{\mbbV^*}(f-Bu_n)\,,
\\ 
\label{eq:nonlinearPG}
& \quad \<r^{\ast\ast},\nu_n\>_{\mbbV^{\ast\ast},\mbbV^{\ast}}=0,\quad\forall \nu_n\in\mbbV^*_n\,.
\end{empheq}
Observe that $r^{\ast\ast}\in \mathcal{J}_{\mbbV^*}(f-Bu_n)$ must fulfill the 
set of \emph{nonlinear} equations:
\begin{alignat*}{2}
\|r^{\ast\ast}\|_{\mbbV^{\ast\ast}}^2=\<r^{\ast\ast},f-Bu_n\>_{\mbbV^{\ast\ast},\mbbV^{\ast}}=\|f-Bu_n\|^2_{\mbbV^{\ast}}\,.
\end{alignat*}
\par
In special Banach spaces, because of specific properties of $\mcJ_{\mbbV^*}$ depending on the geometry of~$\mbbV$, the above characterizations can be reduced to other forms. For example, if~$\mbbV$ is reflexive, then $f-Bu_n\in \mathcal{J}_{\mbbV}(r)$ for some $r\in\mbbV$ such that $\<Bw_n,r\>_{\mbbV^*,\mbbV}=0$, for all $w_n\in \mbbU_n$. To have more useful characterizations, we shall restrict to the 
%reflexive smooth setting. 
%
%\subsection{The reflexive smooth setting}\label{sec:reflexivesmooth}
%We apply the previous result to the 
\emph{reflexive smooth setting}.
%, which refers to the particular situation for which $\mbbV$ and $\mbbV^*$ are strictly convex and reflexive. 
% Many equivalent formulations follow because of this additional structure. 
Recall from Section~\ref{sec:dualMapSmooth} that in this setting the duality mapping in~$\mbbV$ is a single-valued and bijective map, denoted by $J_\mbbV:\mbbV\rightarrow \mbbV^*$. 
\begin{theorem}[Equivalent characterizations]
\label{thm:ref_smooth}
Let $\mbbU$ and $\mbbV$ be two Banach spaces and let $B:\mbbU \to\mbbV^*$ be a linear, continuous and bounded-below operator. Assume additionally that $\mbbV$ and $\mbbV^*$ are strictly convex and reflexive. Given $f\in \mbbV^*$ and a finite-dimensional subspace
$\mbbU_n\subset \mbbU$. The following statements are equivalent:

\begin{enumerate}[(i)]
\item $u_n\in \mbbU_n$ is the unique residual minimizer such that $$u_n=\arg\!\!\!\min_{w_n\in\mbbU_n}\|f-Bw_n\|_{\mbbV^*}\,\,.$$ 
%satisfying~\eqref{eq:residual_minimization}.
%
\item $u_n\in \mbbU_n$ is the solution of the nonlinear Petrov--Galerkin formulation:
$$
\<\nu_n,J_\mbbV^{-1}(f-Bu_n)\>_{\mbbV^\ast,\mbbV}=0,\quad\forall \nu_n\in B\mbbU_n\,\,.
$$
\item 
There is a unique residual representation $r\in \mbbV$ such that $u_n\in \mbbU_n$ together with $r$ satisfy the semi-infinite monotone mixed formulation:
\begin{subequations}
\label{eq:mixed1}
\begin{empheq}[left=\left\{,right=\right.,box=]{alignat=3}
\label{eq:mixed1_a}
& \<{J_\mbbV}(r),v\>_{\mbbV^*,\mbbV}+\<Bu_n,v\>_{\mbbV^*,\mbbV} && =\<f,v\>_{\mbbV^*,\mbbV}, &\quad & \forall v\in \mbbV,
\\
\label{eq:mixed1_b}
& \<B^*r,w_n\>_{\mbbU^*,\mbbU} && =0, && \forall w_n\in \mbbU_n.
\end{empheq}
\end{subequations}
\item $u_n\in \mbbU_n$ is the Lagrange multiplier of the constrained minimization:
\begin{equation}\label{eq:constrainedminimization}
 \min_{v 
 \in (B\mbbU_n)^\perp} \frac{1}{2}\norm{v}^2_{\mbbV} - \dual{f,v}_{\mbbV^*,\mbbV}.
\end{equation}
\end{enumerate}
\end{theorem}
%
%\begin{remark}[Estimates]
%Note that for all the formulations in Theorem~\ref{thm:ref_smooth}, the a~priori bound in~\eqref{eq:est1}, its alternative in Remark~\ref{rem:sharpEst}, and error estimates in~\eqref{eq:cea} are valid on account of the stated equivalences. 
%\end{remark}
%
\begin{proof} 
We proceed by proving consecutively: (i) $\Leftrightarrow$ (ii) $\Leftrightarrow$ (iii) $\Rightarrow$ (iv) $\Rightarrow$ (iii).
\par (i) $\Leftrightarrow$ (ii) :\;
By Proposition~\ref{thm:characterization1}, $u_n\in \mbbU_n$ is the unique minimizer of~\eqref{eq:resmin} if and only if $J_{\mbbV^*}(f-Bu_n)\in \mbbV^{**}$ annihilates the discrete space $B\mbbU_n \subset \mbbV^*$. In other words, because of the identification $J_{\mbbV^*}=J^{-1}_\mbbV$ (see Section~\ref{sec:dualMapSmooth}), 
$$
\<Bw_n,J_\mbbV^{-1}(f-Bu_n)\>_{\mbbV^\ast,\mbbV}=\<J_{\mbbV^*}(f-Bu_n),Bw_n\>_{\mbbV^{\ast\ast},\mbbV^*}=0,
$$
for all $w_n\in\mbbU_n$.
\par (ii) $\Leftrightarrow$ (iii) :\;
Note that $r=J_\mbbV^{-1}(f-Bu_n)$.
\par (iii) $\Rightarrow$ (iv) :\;
The Lagrangian $\mcL:\mbbV\times \mbbU_n\rightarrow \realset$ associated with the constrained 
minimization~\eqref{eq:constrainedminimization} is:
$$
\mathcal L(v,w_n):={1\over2}\|v\|^2_\mbbV-\<f,v\>_{\mbbV^*,\mbbV}+\<B^*v,w_n\>_{\mbbU^*,\mbbU}.
$$
Let $(r,u_n)$ denote the solution to the mixed formulation~\eqref{eq:mixed1}. Firstly, since $r\in (B\mbbU_n)^\perp$, it is straightforward to see that $\mathcal L(r,w_n)=\mathcal L(r,u_n)$. Secondly, 
\begin{alignat*}{2}
\tag{by~\eqref{eq:mixed1_a}}
0&=  \<J_\mbbV(r),v-r\>_{\mbbV^*,\mbbV}+\<Bu_n,v-r\>_{\mbbV^*,\mbbV}-\<f,v-r\>_{\mbbV^*,\mbbV} 
\\
\tag{by Prop.~\eqref{prop:subdifferential}}
 &\leq {1\over2}\|v\|^2_\mbbV-{1\over2}\|r\|^2_\mbbV+\<Bu_n,v-r\>_{\mbbV^*,\mbbV}-\<f,v-r\>_{\mbbV^*,\mbbV}
\\
&=   \mathcal L(v,u_n)-\mathcal L(r,u_n).
\end{alignat*}
Therefore $(r,u_n)$ is a saddle-point of the Lagrangian, i.e., 
\begin{alignat}{2}
\label{eq:saddlePointForm}
  \mcL(r,w_n) \le \mcL(r,u_n) \le \mcL(v,u_n) 
  \qquad\forall (v,w_n) \in \mbbV \times \mbbU_n,
\end{alignat}
which is equivalent to~\eqref{eq:constrainedminimization}.
% the constrained minimization problem.
%
\par
\par (iv) $\Rightarrow$ (iii) :\;Let $(r,u_n)$ be a solution of~\eqref{eq:constrainedminimization}, i.e.,~\eqref{eq:saddlePointForm} holds. The first inequality in~\eqref{eq:saddlePointForm} implies
\begin{alignat*}{2}
  \dual{Bw_n,r}_{\mbbV^*,\mbbV} \le \dual{Bu_n,r}_{\mbbV^*,\mbbV}
  \quad \forall w_n\in \mbbU_n\,,
\end{alignat*}
which implies~\eqref{eq:mixed1_b} by a vector-space argument. Next, considering the second inequality in \eqref{eq:saddlePointForm} with $v$ equal to $r+ \lambda v$, and $\lambda >0$, it follows that
\begin{alignat*}{2}
 0 &\le \lambda^{-1} \Big( \mcL(r +\lambda v, u_n) - \mcL(r,u_n) \Big)
 \\
   &= \lambda^{-1} \Big( \tfrac{1}{2} \norm{r+\lambda v}_{\mbbV}^2 - \tfrac{1}{2} \norm{r}_{\mbbV}^2 \Big)
    - \dual{f, v}_{\mbbV^*,\mbbV}  + \dual{Bu_n,v}_{\mbbV^*,\mbbV}
 \\ \tag{by Prop.~\ref{prop:subdifferential}}
   &\le  \bigdual{J_\mbbV(r+\lambda v), v}_{\mbbV^*,\mbbV}
    + \dual{B u_n ,v}_{\mbbV^*,\mbbV} - \dual{f,v}_{\mbbV^*,\mbbV} 
\end{alignat*}
Therefore, upon~$\lambda \rightarrow 0^+$, invoking hemi-continuity of~$J_\mbbV$ (see~\eqref{eq:hemicont}) and repeating the above with~$-v$ instead of~$v$, one recovers~\eqref{eq:mixed1_a}.
%
%Differentiating with respect to $v$ and $w_n$ we get the optimality conditions:
%\begin{equation}\label{eq:optimal_conditions}
%\left\{
%\begin{array}{rl}
%J_\mbbV(v) - f + Bw_n =0 & \hbox{in } \mbbV^*,\\
%B^*v=0 & \hbox{in } \mbbU_n^*.
%\end{array}
%\right.
%\end{equation}
%These are exactly the equations satisfied by 
%$(r,u_n)\in \mbbV\times \mbbU_n$ solution of~\eqref{eq:mixed1}. Hence, $r\in \mbbV$ is the minimizer of~\eqref{eq:constrainedminimization} and $u_n\in\mbbU_n$ is the Lagrange multiplier of it.
%%
%\par
%%
%Conversely, let $u_n\in \mbbU_n$ be the Lagrange multiplier of the constrained minimization~\eqref{eq:constrainedminimization}, then defining the 
%associated Lagrangian $\mcL:\mbbV\times \mbbU_n\rightarrow \realset$ by
%$$
%\mathcal L(v,w_n)={1\over2}\|v\|^2_\mbbV-\<f,v\>_{\mbbV^*,\mbbV}+\<B^*v,w_n\>_{\mbbU^*,\mbbU}.
%$$
%$r=J_\mbbV^{-1}(f-Bu_n)\in \mbbV$ 
%and computing the corresponding optimality 
%conditions at $(r,u_n)$, we obtain (iii).
%
%we get that the pair 
%$(r,u_n)$ satisfies the optimality 
%conditions~\eqref{eq:optimal_conditions}. In particular,
%$$
%\<Bw_n,J_\mbbV^{-1}(f-Bu_n)\>_{\mbbV^\ast,\mbbV}=\<B^*r,w_n\>=0, \quad \forall w_n\in \mbbU_n,
%$$
%which implies that
%$$
%\<J_{\mbbV^*}(f-Bu_n),Bw_n\>_{\mbbV^{**},\mbbV^*}=0.
%$$ 
%Hence, by Theorem~\ref{thm:characterization1}, $u_n$ is a 
%residual minimizer of equation~\eqref{eq:residual_minimization}.
%Uniqueness follows from the fact that $\mbbV^*$ is strictly convex and \removed{the operator }$B:\mbbU\to \mbbV^*$ is injective.
\end{proof}
\begin{remark}[Mixed form for optimal test-space norm]
If one assumes $B$~is a linear bounded bijective operator and, 
instead of $\norm{\cdot}_{\mbbV}$, one uses the norm~$\norm{\cdot}_{\mbbV_{\mathrm{opt}}}$ on~$\mbbV$ (recall from Remark~\ref{rem:Vopt}), then one can show that \eqref{eq:mixed1_a} holds with~$\dual{J_\mbbV(r),v}_{\mbbV^*,\mbbV}$ replaced by $\bigdual{B^* v,J_{\mbbU}^{-1}(B^*r)}_{\mbbU^*,\mbbU}$.
\end{remark}

\section{Analysis of the inexact method}
\label{sec:practical}
We now consider the \emph{tractable} approximation.
%  in the reflexive smooth setting. 
%As usual, we consider two Banach spaces $\mbbU$ and $\mbbV$ in the reflexive smooth setting, together with a linear, continuous and bounded-below operator $B:\mbbU\to \mbbV^*$. 
The reflexive smooth setting guarantees that the semi-infinite mixed formulation~\eqref{eq:mixed1} introduced in Theorem~\ref{thm:ref_smooth} is well posed. For convenience, this formulation will be the starting point for the inexact method. 
\par
In addition to $\mbbU_n\subset \mbbU$, let $\mbbV_m\subset \mbbV$ be a finite-dimensional subspace. We shall then consider: 
\begin{subequations}
\label{eq:mixed_discrete}
\begin{empheq}[left=\left\{\;,right=\right.,box=]{alignat=3} 
\notag
& \text{Find } (r_m,u_n)\in \mbbV_m \times \mbbU_n :
\\ \label{eq:mixed_discrete_a}
& \quad \<J_\mbbV(r_m),v_m\>_{\mbbV^{\ast},\mbbV}+\<Bu_n,v_m\>_{\mbbV^{\ast},\mbbV} &&=\<f,v_m\>_{\mbbV^{\ast},\mbbV} 
 & \quad  & \forall v_m\in \mbbV_m\,,
\\ \label{eq:mixed_discrete_b}
& \quad \<B^*r_m,w_n\>_{\mbbU^{\ast},\mbbU} 
&& =0  && \forall w_n\in \mbbU_n\,.
\end{empheq}
\end{subequations}
Because the nonlinear operator~$J_\mbbV$ is monotone, we refer to the above as a \emph{monotone mixed method}. 
\subsection{Equivalent discrete settings}
Analogous to the semi-infinite mixed formulation, which is equivalent to residual minimization (see Theorem~\ref{thm:ref_smooth}), the monotone mixed method is related to residual minimization in the \emph{discrete} dual norm
\begin{alignat*}{2}
 \norm{\cdot}_{(\mbbV_m)^*} =
 \sup_{v_m\in \mbbV_m}{\< \,\cdot\, ,v_m\>_{(\mbbV_m)^*,\mbbV_m}\over \|v_m\|_{\mbbV}} 
%  =
% \sup_{v_m\in \mbbV_m}{\< I_m^*(\cdot) ,v_m\>_{\mbbV^*,\mbbV}\over \|v_m\|_{\mbbV}} 
\,.
\end{alignat*}
The next theorem summarizes this equivalence and, additionally, shows the equivalence with an \emph{inexact} version of the nonlinear Petrov--Galerkin discretization and a~\emph{discrete} constrained minimization. 

\begin{theorem}[Discrete equivalent characterizations]\label{teo:discrete_residual}
%
%\marginnote{@Ignacio: Can you add the inexact nonlinear PG?}
Let $\mbbU$ and $\mbbV$ be two Banach spaces and let $B:\mbbU \to\mbbV^*$ be a linear, continuous and bounded-below operator. Assume that $\mbbV$ and $\mbbV^*$ are reflexive and strictly convex. Given $f\in \mbbV^*$ and finite-dimensional subspaces $\mbbU_n\subset \mbbU$ and $\mbbV_m\subset \mbbV$, the following statements are equivalent:
\begin{itemize}
\item[(i)] $(r_m,u_n)\in \mbbV_m\times \mbbU_n$ is a solution of the discrete mixed problem:
$$
\left\{
\begin{array}{lll}
%\begin{subequations*}
%\label{eq:mixed_discrete}
%\begin{empheq}[left=\left\{\;,right=\right.,box=]{alignat=3} 
%\notag
%& \text{Find } (r_m,u_n)\in \mbbV_m \times \mbbU_n \hbox{ such that} 
%\\ %\label{eq:mixed_discrete_a}
 \<J_\mbbV(r_m),v_m\>_{\mbbV^{\ast},\mbbV}+\<Bu_n,v_m\>_{\mbbV^{\ast},\mbbV} & =\<f,v_m\>_{\mbbV^{\ast},\mbbV} \,,
 & \quad \forall v_m\in \mbbV_m\,,
\\ %\label{eq:mixed_discrete_b}
\<B^*r_m,w_n\>_{\mbbU^{\ast},\mbbU} 
& =0\,, & \quad\forall w_n\in \mbbU_n\,.
\end{array}
\right.$$
\item[(ii)]
$u_n\in\mbbU_n$ is a solution of the inexact non-linear Petrov-Galerkin discretization:
\begin{alignat}{2}\label{eq:dicreteNPG}
\Big< \nu_n
  \,,\, I_m J^{-1}_{\mbbV_m}\circ I_m^*(f-Bu_n) \Big>_{\mbbV^*,\mbbV} 
  = 0\,, \quad \forall
\nu_n\in B\mbbU_n\,.
\end{alignat}
and $r_m=J_{\mbbV_m}^{-1}\circ I_m^* (f-Bu_n)$, where $I_m:\mbbV_m\rightarrow \mbbV$ is the natural injection. 
\item[(iii)]
$u_n\in \mbbU_n$ is a minimizer of the discrete residual minimization problem: 
\begin{equation}\label{eq:discretemin}
  \min_{w_n\in \mbbU_n}\|I_m^*(f-Bw_n)\|_{(\mbbV_m)^*}\,\,,
\end{equation}
and $r_m=J_{\mbbV_m}^{-1}\circ I_m^* (f-Bu_n)$, where $I_m^*:\mbbV^*\rightarrow (\mbbV_m)^*$ is the natural injection.
%\item[(iii)] $(u_h,v_r)\in U_h\times V_r$ satisfies
%$\displaystyle\mathcal{L}(v_r,u_h)=\min_{y_r\in V_r}\max_{w_h\in U_h}\mathcal{L}(y_r,w_h)$, for the Lagrangian function
%$$
%\mathcal{L}(y_r,w_h):={1\over 2}\|y_r\|^2_V+\<Bw_h-f,y_r\>_{V^*,V}.
%$$

\item[(iii)]
$u_n\in \mbbU_n$ is the Lagrange multiplier of the discrete constrained minimization problem:
\begin{equation}\label{eq:discrete_constrained_min}
\min_{v_m\in \mbbV_m\cap (B\mbbU_n)^\bot}
\left\{
{1\over2}\|v_m\|_\mbbV^2-\<f,v_m\>_{\mbbV^*,\mbbV}
\right\},
\end{equation}
while $r_m\in\mbbV_m$ is the minimizer of it.
\end{itemize}
\end{theorem}

\begin{proof}
%We demonstrate each equivalence with respect to (i).
%
%
First notice the following direct equivalences:
\begin{alignat*}{2}
 &&  r_m &= J_{\mbbV_m}^{-1} \circ I_m^*(f-Bu_n)
\\
 \Leftrightarrow & \quad&  
  J_{\mbbV_m}( r_m )  &= I_m^*(f-Bu_n)
\\
\tag{by Lemma~\ref{lem:IJI}}
 \Leftrightarrow & \quad&  
  I_m^* J_{\mbbV}( I_m r_m )  &= I_m^*(f-Bu_n) \,,
\end{alignat*}
where the last statement is equivalent to~\eqref{eq:mixed_discrete_a}.
\par (i) $\Rightarrow$ (ii).
If $(u_n,r_m)\in \mbbU_n\times \mbbV_m$ is a solution of~\eqref{eq:mixed_discrete}, then $r_m=J_{\mbbV_m}^{-1}\circ I_m^* (f-Bu_n)$ and~\eqref{eq:mixed_discrete_b} is nothing but~\eqref{eq:dicreteNPG}. 
\par (ii) $\Rightarrow$ (iii).
%Next, if $(u_n,r_m)\in \mbbU_n\times \mbbV_m$ is a solution of~\eqref{eq:mixed_discrete}, then for any $w_n\in \mbbU_n$ we have:
Observe that for any $w_n\in\mbbU_n$ we have :
\begin{alignat*}{2}
\notag
\|I_m^*(f-Bu_n)\|_{(\mbbV_m)^*} &= 
\sup_{v_m\in\mbbV_m}{\<I_m^*(f-Bu_n),v_m\>_{(\mbbV_m)^*,\mbbV_m}\over \|v_m\|_\mbbV} 
\\
& = 
\sup_{v_m\in\mbbV_m}{\<J_{\mbbV_m}(r_m),v_m\>_{(\mbbV_m)^*,\mbbV_m}\over \|v_m\|_\mbbV} 
%\\
%\tag{by \eqref{eq:mixed_discrete_a}}
%&= 
%\sup_{v_m\in \mbbV_m}{\< J_\mbbV(r_m),v_m\>_{\mbbV^*,\mbbV}\over \|v_m\|_\mbbV}
\\ \tag{by \eqref{eq:normSup}}
% \label{eq:discrete_dual_norm}
&= {\<J_{\mbbV_m}(r_m),r_m\>_{(\mbbV_m)^*,\mbbV_m}\over \|r_m\|_\mbbV }
\\
%\tag{by \eqref{eq:mixed_discrete_a}}
 &=  {\<I_m^*(f-Bu_n),r_m\>_{(\mbbV_m)^*,\mbbV_m}\over \|r_m\|_\mbbV}
\\
\tag{by \eqref{eq:dicreteNPG}}
 &=   {\<I_m^*(f-Bw_n),r_m\>_{(\mbbV_m)^*,\mbbV_m}\over \|r_m\|_\mbbV}
%\\
% &=  {\<I_m^*(f-Bw_n),r_m\>_{(\mbbV_m)^*,\mbbV_m}\over \|r_m\|_\mbbV}
 \\
\notag
 &\le  \norm{I_m^*(f-Bw_n)}_{(\mbbV_m)^*}\,\,.
\end{alignat*}
Thus, $u_n$ is a minimizer of~\eqref{eq:discretemin}. 
% Observe that the last expression is bounded above by $\|f-Bw_n\|_{(\mbbV_m)^*}$. 
%
\par (iii) $\Rightarrow$ (i).
If $u_n\in \mbbU_n$ is a minimizer of~\eqref{eq:discretemin} and $r_m = J_{\mbbV_m}^{-1} \circ I_m^*(f-Bu_n) =  J_{(\mbbV_m)^*} \circ I_m^*(f-Bu_n) $, then by Theorem~\ref{thm:minimizer_characterization}, with $\mbbM=I_m^*B\mbbU_n \subset (\mbbV_m)^*=\mbbY$, $r_m$ satisfies: 
\begin{alignat*}{2}
0 = \dual{I_m^* Bw_n,r_m}_{(\mbbV_m)^*,\mbbV_m} = \dual{Bw_n,I_mr_m}_{\mbbV^*,\mbbV} 
= \dual{B^*r_m,w_n }_{\mbbU^*,\mbbU},
\quad \forall w_n \in \mbbU_n\,,
\end{alignat*}
which verifies~\eqref{eq:mixed_discrete_b}. 
\par (i) $\Leftrightarrow$ (iv).
The proof of this equivalence follows exactly the same reasoning as in the semi-infinite setting; see the proof of Theorem~\ref{thm:ref_smooth}, part~(iii) $\Leftrightarrow$ (iv).
%%
%The Lagrangian function associated with the constrained minimization~\eqref{eq:discrete_constrained_min} is:
%$$
%\mathcal{L}(v_m,w_n):={1\over 2}\|v_m\|^2_\mbbV+\<Bw_n-f,v_m\>_{\mbbV^*,\mbbV}\,\,.
%$$
%Differentiating with respect to $v_m$ and $w_n$ we get the optimality conditions:
%$$
%\left\{
%\begin{array}{rl}
%\<J_{\mbbV_m}(r_m),v_m\>_{\mbbV^*,\mbbV} + \<Bu_n-f,v_m\>_{\mbbV^*,\mbbV}=0, & \forall v_m\in \mbbV_m,\\
%\<B^*r_m,w_n\>_{\mbbU^*,\mbbU}=0, & \forall w_n\in \mbbU_n.
%\end{array}
%\right.
%$$
%Thus, statement (i) occurs if and only if statement (iii) occurs.
%
\end{proof}

\subsection{Well-posedness of the inexact method}
We now study the existence and uniqueness of solutions to the inexact method~\eqref{eq:mixed_discrete}. A critical ingredient for the uniqueness analysis is the following condition:
\begin{assumption}[Fortin condition] 
\label{assumpt:Fortin}
Let $\{(\mbbU_n,\mbbV_m)\}$ be a family of \emph{discrete} subspace pairs, where $\mbbU_n\subset\mbbU$ and $\mbbV_m\subset\mbbV$. For each pair $(\mbbU_n,\mbbV_m)$ in this family, there exists an operator $\Pi_{n,m}:\mbbV\to \mbbV_m$ and constants $C_\Pi>0$ and $D_\Pi>0$ (independent of $n$ and $m$) such that the following conditions are satisfied:
\begin{subequations}
\label{eq:Fortin}
\begin{empheq}[left=\left\{,right=\right.,box=]{alignat=2}
\label{eq:Fortin_a}
\,\, &
 \|\Pi_{n,m} v\|_\mathbb V\leq C_\Pi \|v\|_\mathbb V\,\,, & \quad & \forall v\in\mathbb V\,, % \hbox{ and some constant } C_\Pi >0\,;
\\ \label{eq:Fortin_b}
& \|(I-\Pi_{n,m}) v\|_\mathbb V\leq D_\Pi \|v\|_\mathbb V\,\,, && \forall v\in\mathbb V\,,% \hbox{ and some constant } D_\Pi >0\,;
\\ \label{eq:Fortin_c}
& \<Bw_n,v-\Pi_{n,m} v\>_{\mbbV^\ast,\mbbV}=0, &&  \forall w_n\in \mbbU_n,\,\forall v\in \mbbV,
\end{empheq}
\end{subequations}
where $I:\mathbb V\to\mbbV$ is the identity map in~$\mbbV$. For simplicity, we write $\Pi$ instead of~$\Pi_{n,m}$.
\end{assumption}
Such an operator is referred to as a \emph{Fortin} operator after Fortin's trick in mixed finite element methods~\cite[Section~5.4]{BofBreForBOOK2013}. For the existence of~$\Pi$, note that the last identity~\eqref{eq:Fortin_c} requires that $\dim \mbbV_m \ge \dim \Image(B|_{\mbbU_n}) = \dim \mbbU_n$ (for a bounded-below operator~$B$). Note that~\eqref{eq:Fortin_a} implies~\eqref{eq:Fortin_b} with~$D_\Pi = 1+ C_\Pi$, but to allow for sharper estimates, we prefer to retain the independent constant~$D_\Pi$.
\begin{theorem}[Inexact method: Discrete well-posedness]
\label{thm:discrete}
Let $\mbbU$ and $\mbbV$ be two Banach spaces and let $B:\mbbU \to\mbbV^*$ be a linear, continuous and bounded-below operator, with continuity constant $M_B>0$ and bounded-below constant $\gamma_B>0$. Assume that $\mbbV$ and $\mbbV^*$ are reflexive and strictly convex. Let $\mbbU_n\subset \mbbU$ and $\mbbV_m\subset \mbbV$ be finite-dimensional subspaces such that the (Fortin) Assumption~\ref{assumpt:Fortin} holds true.  Given $f\in \mbbV^*$, there exists a unique solution $(r_m,u_n)\in \mbbV_m\times \mbbU_n$ of the inexact method:
\begin{subequations}
\label{eq:thm:discrete:MMM}
\begin{empheq}[left=\left\{\;,right=\right.,box=]{alignat=3} 
\label{eq:thm:discrete:MMMa}
 \<J_\mbbV(r_m),v_m\>_{\mbbV^{\ast},\mbbV}
 +\<Bu_n,v_m\>_{\mbbV^{\ast},\mbbV} 
 &=\<f,v_m\>_{\mbbV^{\ast},\mbbV} \,,
 & \quad & \forall v_m\in \mbbV_m\,,
\\ \label{eq:thm:discrete:MMMb}
\<B^*r_m,w_n\>_{\mbbU^{\ast},\mbbU} 
 &=0\,, 
 & \quad & \forall w_n\in \mbbU_n\,.
\end{empheq}
\end{subequations}
Moreover, let $u\in\mbbU$ be such that $Bu=f$, then we have the 
a priori bounds:
\begin{subequations}
\label{eq:apriori_bounds}
\begin{empheq}[left=\left\{,right=\right.,box=]{alignat=2}
& \|r_m\|_\mbbV\leq \|f\|_{\mbbV^*}\leq M_B\|u\|_\mbbU
\quad\hbox{and}
\label{eq:r_apriori_bound}\\\label{eq:u_apriori_bound}
& \|u_n\|_\mbbU\leq {C_\Pi\over\gamma_B}\big(1+\CAO(\mbbV)\big)
\|f\|_{\mbbV^*}\leq {C_\Pi\over\gamma_B}\big(1+\CAO(\mbbV)\big)M_B\|u\|_\mbbU\,,
\end{empheq}
\end{subequations}
where $C_\Pi>0$ is the boundedness constant of the Fortin operator (see Assumption~\ref{assumpt:Fortin}) and $\CAO(\mbbV)\in[0,1]$ is the asymmetric-orthogonality geometrical constant related to the space $\mbbV$ (see Definition~\ref{def:AOconstant}).
\end{theorem}
%
%Note that by the equivalences in Theorem~\ref{teo:discrete_residual}, Theorem~\ref{thm:discrete} also implies that the inexact residual minimization, inexact nonlinear Petrov--Galerkin method and discrete constrained minimization are well-posed. 
%
\begin{proof}
To proof existence, we consider the equivalent discrete constrained minimization problem~\eqref{eq:discrete_constrained_min}. The existence of a minimizer $r_m\in\mathbb V_m \cap (B\mbbU_n)^\bot$ is guaranteed since the functional $v_m\mapsto \tfrac{1}{2} \norm{v_m}_\mbbV^2 - \dual{f,v_m}_{\mbbV^*,\mbbV}$ is convex and continuous, and $\mbbV_m \cap (B\mbbU_n)^\bot$ is a closed subspace.
\par
Next, we claim that there exist a $u_n\in \mbbU_n$ such that
\begin{alignat*}{2}
  \<Bu_n,v_m\>_{\mbbV^*,\mbbV}=\<f-J_\mbbV(r_m),v_m\>_{\mbbV^*,\mbbV}, \quad\forall v_m\in \mbbV_m.
\end{alignat*}
To see this, consider the restricted operator $B_n:\mbbU_n\to\mbbV^*$, such that $B_nw_n=Bw_n$ for all $w_n\in \mbbU_n$, and recall the natural injection $I_m: \mbbV_m \rightarrow \mbbV$. Then, the above translates into 
\begin{alignat*}{2}
 I^*_m B_n u_n
  = I^*_m \big( f -  J_\mbbV(r_m)  \big)
  \quad \text{in } (\mbbV_m)^*. 
\end{alignat*}
Thus, to proof existence, we show that $I_m^*(f-J_\mbbV(r_m))$ is in the (closed) range of the finite-dimensional operator 
$I_m^*B_n:\mbbU_n\to (\mbbV_m)^*$. 
Since $r_m$ is the minimizer of~\eqref{eq:discrete_constrained_min}, we have
\begin{alignat*}{2}
& 0  = \<J_\mbbV(r_m)-f,I_mv_m\>_{\mbbV^*,\mbbV}
 = \<I_m^* (J_\mbbV(r_m)-f),v_m\>_{(\mbbV_m)^*,\mbbV_m},
\\
  &   \forall v_m\in \mbbV_m\cap (B\mbbU_n)^\bot=\ker (B_{n}^\ast I_m).
\end{alignat*}
Hence, $I_m^*(f-J_\mbbV(r_m))\in (\ker (B_{n}^\ast I_m))^\bot=\Image (I_m^*B_n)$.
\par
To prove uniqueness assume that $(u_n,r_m)$ and $(\tilde u_n,\tilde r_m)$ are two solutions of problem
\eqref{eq:mixed_discrete}.  Then, by subtraction, it is immediate to see that:
$$
\<J_\mbbV(r_m)-J_\mbbV(\tilde r_m),r_m-\tilde r_m\>_{\mbbV^*,\mbbV}=0, 
$$
which implies that $\tilde r_m=r_m$ by strict monotonicity of $J_\mbbV$ (see~\eqref{eq:strictMonotone}). 
Going back to~\eqref{eq:thm:discrete:MMMa} we now obtain $\<B(u_n-\tilde u_n),v_m\>_{\mbbV^\ast,\mbbV}=0$, for all $v_m\in \mbbV_m$. Therefore, by  the Fortin-operator property~\eqref{eq:Fortin_c}, 
\begin{alignat*}{2}
\<B(u_n-\tilde u_n),v\>_{\mbbV^\ast,\mbbV}=\<B(u_n-\tilde u_n),\Pi v\>_{\mbbV^\ast,\mbbV}=0,
\quad \forall v\in \mbbV.
\end{alignat*}
Thus, $B(u_n-\tilde u_n)=0$ which implies $u_n-\tilde u_n=0$ since~$B$ is bounded below.
\par
The \emph{a priori} bound~\eqref{eq:r_apriori_bound} is straightforwardly  obtained by replacing $v_m=r_m$ in~\eqref{eq:thm:discrete:MMMa} and using~\eqref{eq:thm:discrete:MMMb} together with the Cauchy--Schwartz inequality.
\par
For the \emph{a priori} bound~\eqref{eq:u_apriori_bound}, we refer to Proposition~\ref{prop:Pbound} in Section~\ref{sec:apriori}.
\end{proof}
\par
Although~$\mbbV_m$ should be sufficiently large for stability, there is no need for it to be close to the entire~$\mbbV$. The following proposition essentially shows that the goal of~$\mbbV_m$ is to resolve the residual~$r$ in the semi-discrete formulation~\eqref{eq:mixed1}.
\begin{proposition}[Optimal~$\mbbV_m$]
\label{prop:r_in_V} Assuming the same conditions of Theorem~\ref{thm:discrete}, let $(r,u_n)\in \mbbV\times \mbbU_n$ be the solution of the semi-discrete formulation~\eqref{eq:mixed1}. 
%and let $\widetilde\mbbV_m:=\mbbV_m+<r>$. 
If $r\in\mbbV_m$, then the pair $(r,u_n)$ is also the unique solution of the fully-discrete formulation~\eqref{eq:thm:discrete:MMM}.
%\begin{subequations}
%\label{eq:tilde_mixed_discrete}
%\begin{empheq}[left=\left\{\;,right=\right.,box=]{alignat=3} 
%& \text{Find } (\widetilde r_m,\widetilde u_n)\in \widetilde\mbbV_m \times \mbbU_n \hbox{ such that}\notag\\
%\label{eq:tilde_mixed_discrete_a}
%& \quad \<J_\mbbV(\widetilde r_m),\widetilde v_m\>_{\mbbV^{\ast},\mbbV}+\<B\widetilde u_n,\widetilde v_m\>_{\mbbV^{\ast},\mbbV} &&=\<f, \widetilde v_m\>_{\mbbV^{\ast},\mbbV} 
% & \quad  & \forall\,\widetilde v_m\in \widetilde\mbbV_m\,,
%\\ \label{eq:tilde_mixed_discrete_b}
%& \quad \<B^*\widetilde r_m,w_n\>_{\mbbU^{\ast},\mbbU} 
%&& =0  && \forall\,w_n\in \mbbU_n\,.
%\end{empheq}
%\end{subequations}
\end{proposition}
\begin{proof}
Let $(\widetilde r_m,\widetilde u_n)\in\mbbV_m\times\mathbb U_n$ be the unique solution of~\eqref{eq:thm:discrete:MMM} which is guranteed by Theorem~\ref{thm:discrete}. 
%In fact, since $\mbbV_m\subset\widetilde\mbbV_m$, the Assumption~\ref{assumpt:Fortin} is inmediately satisfied for the pairing $(\mbbU_n,\widetilde\mbbV_m)$.  
The aim is to prove  that $(\widetilde r_m,\widetilde u_n)=(r,u_n)$. So testing~\eqref{eq:mixed1_a} with $v_m\in\mbbV_m$ and subtracting~\eqref{eq:thm:discrete:MMMa} (satisfied by $(\widetilde r_m,\widetilde u_n)$) we get:
\begin{equation}\label{eq:substraction}
\dual{B(u_n-\widetilde u_n), v_m}_{\mbbV^*,\mbbV}= -\dual{J_\mbbV(r)-J_\mbbV(\widetilde r_m), v_m}_{\mbbV^*,\mbbV}\,\,,\quad\forall\, v_m\in \mbbV_m\,.
\end{equation}
In particular for $v_m=r-\widetilde r_m\in B(\mbbU_n)^\bot$ we obtain:
$$
\dual{J_\mbbV(r)-J_\mbbV(\widetilde r_m),r-\widetilde r_m}_{\mbbV^*,\mbbV}=0\,,
$$
which implies $r=\widetilde r_m$ by strict monotonicity of $J_\mbbV$. Going back to~\eqref{eq:substraction} we get:
$$
\dual{B(u_n-\widetilde u_n),v_m}_{\mbbV^*,\mbbV}=0\,\,,\quad\forall\,v_m\in \mbbV_m\,,
$$
which implies $u_n=\widetilde u_n$ by the Fortin condition~\eqref{eq:Fortin_c} and the injectivity of~$B$.
\end{proof}

\subsection{Error analysis of the inexact method}

We next present an error analysis for the inexact method. Since the method is fundamentally related to (discrete) residual minimization, the most straightforward error estimate is of \emph{a~posteriori} type. Immediately after, an \emph{a~priori} error estimate follows naturally from the a posteriori estimate (compare with the error estimates for the \emph{exact} residual-minimization method in~\eqref{eq:cea}). The constant in the resulting a~priori estimate can however be improved by resorting to an alternative analysis technique, which we present in Section~\ref{sec:apriori}. 
\par
%
%In the following results, recall that $M_B>0$ and $\gamma_B>0$ denote the continuity and bounded-below constants of the operator~$B:\mbbU\to\mbbV^*$, \added{and that~$C_\Pi$ and~$D_\Pi$ are the constants in~\eqref{eq:Fortin_a}--\eqref{eq:Fortin_b} related to the Fortin operator~$\Pi:\mbbV\rightarrow \mbbV_m$. }
%
\begin{theorem}[Inexact method: A~posteriori error estimate]
\label{thm:aposteriori}
%Under the same conditions of Theorem~\ref{thm:discrete}, 
Let $\mbbU$ and $\mbbV$ be two Banach spaces and let $B:\mbbU \to\mbbV^*$ be a linear, continuous and bounded-below operator, with continuity constant $M_B>0$ and bounded-below constant $\gamma_B>0$. Assume that $\mbbV$ and $\mbbV^*$ are reflexive and strictly convex. Let $\mbbU_n\subset \mbbU$ and $\mbbV_m\subset \mbbV$ be finite-dimensional subspaces such that the (Fortin) Assumption~\ref{assumpt:Fortin} holds true.
Given $f=Bu\in\mbbV^*$, let $(r_m,u_n)\in \mathbb V_m\times \mathbb U_n$ be the unique solution of the discrete mixed problem:
%~\eqref{eq:mixed_discrete}. 
$$
\left\{
\begin{array}{lll}
%\begin{subequations*}
%\label{eq:mixed_discrete}
%\begin{empheq}[left=\left\{\;,right=\right.,box=]{alignat=3} 
%\notag
%& \text{Find } (r_m,u_n)\in \mbbV_m \times \mbbU_n \hbox{ such that} 
%\\ %\label{eq:mixed_discrete_a}
 \<J_\mbbV(r_m),v_m\>_{\mbbV^{\ast},\mbbV}+\<Bu_n,v_m\>_{\mbbV^{\ast},\mbbV} & =\<f,v_m\>_{\mbbV^{\ast},\mbbV} \,,
 & \quad \forall v_m\in \mbbV_m\,,
\\ %\label{eq:mixed_discrete_b}
\<B^*r_m,w_n\>_{\mbbU^{\ast},\mbbU} 
& =0\,, & \quad\forall w_n\in \mbbU_n\,.
\end{array}
\right.$$
%If $f\in \Image(B)$ and $u\in\mbbU$ is the solution of the continuous problem $Bu = f$, 
Then $u_n$ satisfies the following a posteriori error estimate:
\begin{equation}\label{eq:aposteriori}
\|u-u_n\|_\mathbb U\leq{1\over\gamma_B}\osc(f)+{C_\Pi\over\gamma_B}
\norm{r_m}_\mbbV
\,\,,
\end{equation}
where the data-oscillation term~$\osc(f)$ satisfies
\begin{alignat}{2}
\label{eq:osc}
\osc(f):=\displaystyle\sup_{v\in \mathbb V}{\<f,v-\Pi v\>\over \|v\|_\mathbb V}
\leq M_B D_\Pi \inf_{w_n\in\mbbU_n} \|u-w_n\|_\mathbb U\,,
\end{alignat}
and $\|r_m\|_\mbbV$ satisfies
\begin{alignat}{2}
\label{eq:aposteriori2}
 \norm{r_m}_{\mbbV} 
 %= \|I_m^*(f-Bu_n)\|_{(\mathbb V_m)^*}
 \le M_B\inf_{w_n\in\mbbU_n} \|u-w_n\|_\mbbU\,.
\end{alignat}
The constants $C_\Pi$ and~$D_\Pi$ correspond to the boundedness constants related to the Fortin operator~$\Pi:\mbbV\rightarrow \mbbV_m$ (see eq.~\eqref{eq:Fortin_a}--\eqref{eq:Fortin_b}). 
\end{theorem}
\begin{remark}[Lower bounds]\label{rem:aposteriori_Hilbert}
Observe that~\eqref{eq:aposteriori2} and~\eqref{eq:osc} say that:
$$\norm{r_m}_{\mbbV} \leq M_B \|u-u_n\|_\mbbU\qquad\hbox{and} 
\qquad\osc(f)\leq M_B D_\Pi \|u-u_n\|_\mbbU\,.$$ 
Hence, the a posteriori error estimate in~Theorem~\ref{thm:aposteriori} is \emph{reliable} and \emph{efficient}. This extends the result by Carstensen, Demkowicz~\& Gopalakrishnan~\cite{CarDemGopSINUM2014} and Cohen, Dahmen \& Welper~\cite[Proposition~3.2]{CohDahWelM2AN2012} for the Hilbert-space version of the method. 
\end{remark}
\begin{proof} Using that $B$ is bounded from below, and that $Bu=f$, we get:
\begin{alignat*}{2}
\|u-u_n\|_\mathbb U\leq {1\over\gamma_B}\|Bu-Bu_n\|_{\mathbb V^\ast}
 =  {1\over\gamma_B}\sup_{v\in \mathbb V}{\<f-Bu_n,v-\Pi v+\Pi v\>_{\mathbb V^*,\mathbb V}\over \|v\|_\mathbb V}.
\end{alignat*}
Next, by definition of the $\Pi$ operator (eq.~\eqref{eq:Fortin}), $Bu_n\in \mathbb V^*$ annihilates $v-\Pi v$, for all $v\in\mathbb V$. Hence, splitting the supremum we obtain:
\begin{alignat*}{2}
\begin{array}{rl} 
\|u-u_n\|_\mathbb U \leq & \displaystyle{1\over\gamma_B}\sup_{v\in \mathbb V}{\<f,v-\Pi v\>_{\mathbb V^*,\mathbb V}\over \|v\|_\mathbb V}+{1\over\gamma_B}\displaystyle\sup_{v\in \mathbb V}{\<f-Bu_n,\Pi v\>_{\mathbb V^*,\mathbb V}\over \|v\|_\mathbb V}\\\\
  \leq & \displaystyle{1\over\gamma_B}\hbox{osc}(f)+{C_\Pi\over\gamma_B}\sup_{v\in V}{\<f-Bu_n,\Pi v\>_{\mathbb V^*,\mathbb V}\over \|\Pi v\|_\mathbb V},
\end{array}
\end{alignat*}
where we used boundedness of~$\Pi$. 
To obtain~\eqref{eq:aposteriori}, we observe that:
\begin{alignat*}{2}
\sup_{v\in \mathbb V}{\<f-Bu_n,\Pi v\>_{\mathbb V^*,\mathbb V}\over \|\Pi v\|_\mathbb V} &= 
\sup_{v\in \mathbb V}{\<J_{\mbbV}(r_m),\Pi v\>_{\mathbb V^*,\mathbb V}\over \|\Pi v\|_\mathbb V}
\le \norm{J_\mbbV(r_m)}_{\mbbV^*}
= \norm{r_m}_{\mbbV}\,.
\end{alignat*}
%
%OLD:
%\begin{alignat*}{2}
%\sup_{v\in \mathbb V}{\<f-Bu_n,\Pi v\>_{\mathbb V^*,\mathbb V}\over \|\Pi v\|_\mathbb V} &\leq 
%\sup_{v_m\in \mathbb V_m}{\<I_m^*(f-Bu_n),v_m\>_{(\mathbb V_m)^*,\mathbb V_m}\over \|v_m\|_\mathbb V}
%\\
%&=\|I_m^*(f-Bu_n)\|_{(\mbbV_m)^*}
%\\
%&\added{= \norm{r_m}_{\mbbV}}
%\end{alignat*}
%
% \added{where the last identity can be found in the proof of Theorem~\ref{teo:discrete_residual}. }
%
\par
Next, observe that for all $w_n\in \mathbb U_n$ we have
\begin{alignat*}{2}
\osc(f)=\sup_{v\in\mathbb V}{\<f,v-\Pi v\>_{\mathbb V^*,\mbbV}\over \|v\|_\mathbb V}=\displaystyle\sup_{v\in \mathbb V}{\<f-Bw_n,v-\Pi v\>_{\mathbb V^*,\mbbV}\over \|v\|_\mathbb V}
\leq M_B D_\Pi \|u-w_n\|_\mathbb U\,.
\end{alignat*}
Finally, by the proof of Theorem~\ref{teo:discrete_residual}, part~(i)~$\Leftrightarrow$~(ii), 
\begin{alignat*}{2}
\|I_m^*(f-Bu_n)\|_{(\mbbV_m)^*}
 &= \frac{\dual{J_\mbbV(r_m),r_m}_{\mbbV^*,\mbbV}}{\norm{r_m}_\mbbV} 
 = \norm{r_m}_{\mbbV}
 \leq \|I_m^*(f-Bw_n)\|_{(\mbbV_m)^*} \,,
\end{alignat*}$\empty$
and 
\begin{alignat*}{2}
\|I_m^*(f-Bw_n)\|_{(\mbbV_m)^*}
 \leq \|Bu-Bw_n\|_{\mbbV^*}\leq M_B\|u-w_n\|_\mbbU\,.
\end{alignat*}
\end{proof}
\par
A straightforward a~priori error estimate follows naturally from the results in Theorem~\ref{thm:aposteriori}. 
\begin{corollary}[Inexact method: A priori error estimate I]
\label{cor:aPrioriErrorEst}
Under the same assumptions of Theorem~\ref{thm:aposteriori}, we have the following a priori error estimate:
%$$
%\hbox{osc}(f)\leq M_B D_\Pi \inf_{w_n\in\mbbU_n} \|u-w_n\|_\mathbb U,
%$$ 
%\end{proposition}
%and
% \|I_m^*(f-Bu_n)\|_{(\mbbV_m)^*}\leq M_B\inf_{w_n\in\mbbU_n} \|u-w_n\|_\mbbU.
%
\begin{subequations}\label{eq:aPrioriErrorEst}
\begin{alignat}{2}
\label{eq:aPrioriErrorEst_a}
 \|u-u_n\|_\mathbb U 
 &\leq \frac{1}{\gamma_B}\osc(f) 
 + \frac{C_\Pi M_B}{\gamma_B} \inf_{w_n\in\mbbU_n} \|u-w_n\|_\mbbU
\\
\label{eq:aPrioriErrorEst_b}
 &\leq 
 \frac{(D_\Pi + C_\Pi) M_B}{\gamma_B} \inf_{w_n\in\mbbU_n} \|u-w_n\|_\mbbU
 \,.
\end{alignat}
\end{subequations}
\end{corollary}
%
%\begin{proof}
%Observe that for all $w_n\in \mathbb U_n$ we have
%$$
%\hbox{osc}(f)=\displaystyle\sup_{v\in\mathbb V}{\<f,v-\Pi v\>_{\mathbb V^*,\mbbV}\over \|v\|_\mathbb V}=\displaystyle\sup_{v\in \mathbb V}{\<f-Bw_h,v-\Pi v\>_{\mathbb V^*,\mbbV}\over \|v\|_\mathbb V}
%\leq M_B D_\Pi \|u-w_h\|_\mathbb U,
%$$ 
%and by Theorem~\ref{teo:discrete_residual} 
%$$
%\|I_m^*(f-Bu_n)\|_{(\mbbV_m)^*}\leq \|I_m^*(f-Bw_n)\|_{(\mbbV_m)^*}\leq \|Bu-Bw_n\|_{\mbbV^*}\leq M_B\|u-w_n\|_\mbbU.
%$$
%\end{proof}
%
\begin{remark}[Oscillation]
In the context of finite-element approximations, the data-oscillation term in~\eqref{eq:aPrioriErrorEst_a} can generally be expected to be of higher order than indicated by the upper bound in~\eqref{eq:aPrioriErrorEst_b}; see discussion in~\cite{CarDemGopSINUM2014}. 
%In any case, as explained in~\cite{CarDemGopSINUM2014}, the combination of \eqref{eq:aPrioriErrorEst} and the upper bound in~\eqref{eq:osc} provides a quasi-optimal error estimate.
\end{remark}
\begin{remark}[$\mbbV_m=\mbbV$]
\label{rem:special_cases}
Note that if $\mbbV_m=\mbbV$, then $\osc(f)=0$, $D_\Pi=0$ and $C_\Pi = 1$ (choose $\Pi=I$), so that the estimates in \eqref{eq:aposteriori} and \eqref{eq:aPrioriErrorEst} reduce to those in the semi-infinite case~\eqref{eq:cea}.
\end{remark}
\subsection{Direct a~priori error analysis of the inexact method}
\label{sec:apriori}
A direct a priori error analysis is possible for the inexact method, without going through an a posteriori error estimate. The benefit of the direct analysis is that the resulting estimate is sharper than the worst-case upper bound given in~\eqref{eq:aPrioriErrorEst_b}. 
\par
The main idea of the direct analysis is based on the sequence of inequalities (formalized below):
\begin{alignat}{2}
\label{eq:directProof}
 \norm{u-u_n}_{\mbbU} \le \norm{I-P_n} \norm{u-w_n}_\mbbU 
 \le C \norm{P_n} \norm{u-w_n}_\mbbU \qquad \forall w_n \in \mbbU_n\,,
\end{alignat}
where $I$ is the identity, $P_n$ is the projector defined below in Definition~\ref{def:nlpgp} and the norm~$\norm{\cdot}$ corresponds to the standard operator norm.
\par
To define our projector $P_n$, consider any~$u\in \mbbU$. Next, let $(r_m,u_n)\in\mbbV_m\times\mbbU_n$ be the solution of the inexact monotone mixed method~\eqref{eq:mixed_discrete} with $f=Bu\in\mbbV^*$, i.e., 
\begin{subequations}
\label{eq:mixed_discrete2}
\begin{empheq}[left=\left\{\;,right=\right.,box=]{alignat=3} 
\label{eq:mixed_discrete_a2}
& \quad \<J_\mbbV(r_m),v_m\>_{\mbbV^{\ast},\mbbV}+\<Bu_n,v_m\>_{\mbbV^{\ast},\mbbV} &&=\<Bu,v_m\>_{\mbbV^{\ast},\mbbV} 
 & \quad  & \forall v_m\in \mbbV_m\,,
\\ \label{eq:mixed_discrete_b2}
& \quad \<B^*r_m,w_n\>_{\mbbU^{\ast},\mbbU} 
&& =0  && \forall w_n\in \mbbU_n\,.
\end{empheq}
\end{subequations}
\begin{definition}[Nonlinear PG projector]\label{def:nlpgp}
Under the same conditions of Theorem~\ref{thm:discrete}, we define the (inexact) \emph{nonlinear Petrov--Galerkin projector} to be the well-defined map 
$$
P_n:  \mbbU\to\mbbU_n\quad\hbox{such that}\quad P_n(u):=u_n,
$$
with $u_n$ the second argument of the solution~$(r_m,u_n)$ of~\eqref{eq:mixed_discrete2}. 
%that associates ${\mbbU\ni u\mapsto P_n(u):=u_n}$. %, is well-defined, and this map is called the \emph{nonlinear Petrov--Galerkin projector}.
%
% we define the nonlinear \emph{Petrov--Galerkin} projection $P_n:\mbbU\to\mbbU_n$ to be the map that associates $\mbbU\ni u\mapsto P_n(u):=u_n$.
%
% let $u\in\mbbU$ and let $(r_m,u_n)\in\mbbV_m\times\mbbU_n$ be the unique solution of the mixed problem \eqref{eq:mixed_discrete} with $Bu\in\mbbV^*$ as the right hand side, i.e.,
% We define then the nonlinear \emph{Petrov--Galerkin} projection $P_n:\mbbU\to\mbbU_n$ to be the map that associates $\mbbU\ni u\mapsto P_n(u):=u_n$.
\end{definition}
%
%\begin{remark}
%Observe that the nonlinear map that associates ${\mbbU\ni u\mapsto r_m\in \mbbV_m}$
%with~$(r_m,u_n)$ the solution of~\eqref{eq:mixed_discrete2} is also well-defined, 
%and by directly testing equation~\eqref{eq:mixed_discrete_a2} using $v_m=r_m$ we get 
%$$
%\|r_m\|^2_\mbbV=\dual{J_\mbbV(r_m),r_m}_{\mbbV^*,\mbbV}=
%\dual{Bu,r_m}_{\mbbV^*,\mbbV}\leq  \norm{Bu}_{\mbbV^*}\norm{r_m}_\mbbV\,\,.
%$$
%\end{remark}
%
\par
The next result establishes important properties of~$P_n$, including a fundamental bound that depends on the geometric constant~$\CAO(\mbbV)\in [0,1]$ (see Definition~\ref{def:AOconstant}).%
% The proof is shifted to Section~\ref{sec:proofPbound}.
%
\begin{proposition}[Nonlinear PG projector properties]
\label{prop:Pbound}
Under the conditions of Theorem~\ref{thm:discrete}, let $P_n:\mbbU\to\mbbU_n$ denote the nonlinear Petrov--Galerkin projector of Definition~\ref{def:nlpgp}. Then the following properties hold true:
\begin{enumerate}[(i)]
\item $P_n$ is a nontrivial projector: $0\neq P_n=P_n\circ P_n \neq I$\,.
\item $P_n$ is homogeneous: $P_n(\lambda u)=\lambda P_n(u)$, \quad $\forall u\in\mbbU$ and $\forall \lambda\in \mbbR$\,.
\item $P_n$ is bounded and
\begin{alignat}{2}
\label{eq:Pbound}
  \norm{P_n} = \sup_{u\in \mbbU} \frac{\norm{P_n(u)}_{\mbbU}}
  {\norm{u}_\mbbU}
  \le  %\frac{C_\Pi M_B}{\gamma_B} \CBM(\mbbV^*)
 \frac{C_\Pi}{\gamma_B}\big(1+\CAO(\mbbV)\big) M_B\,.
\end{alignat}
%
% where~$\CAO(\mbbV) \in [0,1]$ is the geometric constant of~$\mbbV$ given in Definition~\ref{def:AOconstant}.
\item $P_n$ is distributive in the following sense:
\begin{alignat}{2}\label{eq:quasi-linear}
P_n\big(u-P_n(w)\big)=P_n(u)-P_n(w), \qquad \forall u,w\in\mbbU\,.
\end{alignat}
\item $P_n$ is a generalized orthogonal projector in the sense that 
\begin{alignat*}{2}\label{eq:PnGOP}
P_n(u)=P_n\Big( P_n(u) + \eta \, (I-P_n)(u) \Big), \qquad \text{for any } \eta \in \mbbR \text{ and any } u\in\mbbU\,.
\end{alignat*}
\end{enumerate}
%%
%\begin{alignat*}{2}
%  \Lambda_\mbbV = \sup_{
%% {\substack
%{\tiny
%%\begin{array}{c}
%  (z,v)\in\, \mathcal S_\mbbV
%  %\\\dual{J_{\mbbV_m}(z),v}_{(\mbbV_m)^*,\mbbV_m} = 0
%%\end{array}                
%% } 
%}} 
%   \frac{\dual{J_\mbbV(z),v}_{\mbbV^*,\mbbV}}
%  {\norm{z}_\mbbV\norm{v}_\mbbV}\,,
%\end{alignat*}
%with $\mathcal S_\mbbV=\big\{(z,v)\in \mbbV\times\mbbV : 
%\dual{J_\mbbV(v),z}_{\mbbV^*,\mbbV}=0\big\}$.
%
%\removed{\item[(ii)] $R_m:\mbbU\to\mbbV_m$ is bounded and satisfies the upper bound
%\begin{alignat}{2}
%\label{eq:Rbound}
%\|R_m\|=  \sup_{u\in \mbbU} \frac{\norm{R_m(u)}_{\mbbV}}
%  {\norm{u}_\mbbU}
%  \le M_B.
%\end{alignat}
%%
%\end{itemize}}
\end{proposition}
\begin{proof}
See Section~\ref{sec:proofPbound}.
\end{proof}
Property~(iv) in Proposition~\ref{prop:Pbound} is key to establishing the first inequality in~\eqref{eq:directProof}, indeed, for any~$w_n\in \mbbU_n$, 
\begin{alignat}{2}
\label{eq:directProof2}
 \norm{u-P_n(u)}_\mbbU = \norm{u-w_n - P_n(u-w_n)}_\mbbU 
% = \norm{(I-P_n)(u-w_n}_\mbbU 
\le \norm{I-P_n} \norm{u-w_n}_\mbbU\,.
\end{alignat}
On the other hand, the second inequality in~\eqref{eq:directProof} can be established through properties (i)--(iii) and~(v) in Proposition~\ref{prop:Pbound}. Indeed, these properties correspond to the four requirements for the abstract nonlinear projector~$Q$ of Lemma~\ref{lem:I-P}. Hence we immediately obtain the following key estimate: 
%, which is an extension of 
%We now combine Lemma~\ref{lem:I-P} with Proposition~\ref{prop:Pbound} for the \emph{sharpened} a~priori error estimate.
%
\begin{corollary}[Nonlinear PG projector estimate]
\label{col:NPGbound}
Under the conditions of Proposition~\ref{prop:Pbound}, it holds that 
\begin{alignat*}{2}
  \norm{I-P_n} \le C_S \norm{P_n}\,,
\end{alignat*}
with~$C_S := \min \Big\{ 1 + \norm{P_n}^{-1}\,,\, \CBM(\mbbU) \Big\}$\,.
%
%\begin{alignat*}{2}
%  \norm{I-P_n} \le \min \Big\{ 
%  1 + \norm{P_n} \,,\,\CBM(\mbbV^*) \norm{P_n}
%  \Big\} \,.
%\end{alignat*}
%
\end{corollary}
\begin{proof}
Apply Lemma~\ref{lem:I-P}.
\end{proof}
\par
In conclusion, by combining~\eqref{eq:directProof2} with Corollary~\ref{col:NPGbound} and the bound in~\eqref{eq:Pbound}, we have established the following main result.
\begin{theorem}[Inexact method: A priori error estimate II]\label{thm:discrete_apriori}
Let $\mbbU$ and $\mbbV$ be two Banach spaces and let $B:\mbbU \to\mbbV^*$ be a linear, continuous and bounded-below operator, with continuity constant $M_B>0$ and bounded-below constant $\gamma_B>0$. Assume that $\mbbV$ and $\mbbV^*$ are reflexive and strictly convex. Let $\mbbU_n\subset \mbbU$ and $\mbbV_m\subset \mbbV$ be finite-dimensional subspaces such that the (Fortin) Assumption~\ref{assumpt:Fortin} holds true.
Given $f=Bu\in\mbbV^*$, let $(r_m,u_n)\in \mathbb V_m\times \mathbb U_n$ be the unique solution of the discrete mixed problem:
%~\eqref{eq:mixed_discrete}. 
$$
\left\{
\begin{array}{lll}
%\begin{subequations*}
%\label{eq:mixed_discrete}
%\begin{empheq}[left=\left\{\;,right=\right.,box=]{alignat=3} 
%\notag
%& \text{Find } (r_m,u_n)\in \mbbV_m \times \mbbU_n \hbox{ such that} 
%\\ %\label{eq:mixed_discrete_a}
 \<J_\mbbV(r_m),v_m\>_{\mbbV^{\ast},\mbbV}+\<Bu_n,v_m\>_{\mbbV^{\ast},\mbbV} & =\<f,v_m\>_{\mbbV^{\ast},\mbbV} \,,
 & \quad \forall v_m\in \mbbV_m\,,
\\ %\label{eq:mixed_discrete_b}
\<B^*r_m,w_n\>_{\mbbU^{\ast},\mbbU} 
& =0\,, & \quad\forall w_n\in \mbbU_n\,.
\end{array}
\right.$$
%Under the conditions of Theorem~\ref{thm:discrete}, let $(u_n,r_m)\in \mathbb U_n\times \mathbb V_m$ be the unique solution of the inexact method~\eqref{eq:mixed_discrete}. 
%If $f\in \Image(B)$ and $u\in\mbbU$ is the solution of the continuous problem $Bu = f$, 
Then $u_n$ satisfies the a priori error estimate:
\begin{alignat*}{2}
  \norm{u-u_n}_{\mbbU} 
  \le C
  \inf_{w_n\in \mbbU_n}   \norm{u-w_n}_\mbbU
\end{alignat*}
with 
\begin{alignat*}{2}
 C =  \min\left\{ 
  \frac{C_\Pi}{\gamma_B}\,\big(1+\CAO(\mbbV)\big)\,M_B \, \CBM(\mbbU) \,,\,1+ \frac{C_\Pi }{\gamma_B}\big(1+\CAO(\mbbV)\big) M_B  
  \right\}\,.
\end{alignat*}
(See Assumption~\ref{assumpt:Fortin} for the definition of $C_\Pi$; Definition~\ref{def:Banach-Mazur} for $\CBM(\mbbU)$; and Definition~\ref{def:AOconstant} for $\CAO(\mbbV)$.)
\end{theorem}
\begin{remark}[Hilbert-space case]
\label{rem:Hilbert_case}
If $\mbbU$ and $\mbbV$ are Hilbert spaces, then $\CBM=1$ and $\CAO(\mbbV) = 0$, hence $C = C_\Pi M_B / \gamma_B$ in the a~priori error estimate of Theorem~\ref{thm:discrete_apriori}. This coincides with the known result in the Hilbert-space setting~\cite{GopQiuMOC2014}; see Section~\ref{sec:connectHilbert} for further details on the connection to the method in Hilbert spaces.~
\end{remark}
\begin{corollary}[Vanishing discrete residual]
\label{cor:rm=0}
If $r_m=0$, the a~priori error estimate in Theorem~\ref{thm:discrete_apriori} reduces to:
\begin{alignat*}{2}
%\label{eq:Crm=0}
  \norm{u-u_n}_{\mbbU} \le 
\min \bigg\{ \frac{C_\Pi M_B}{\gamma_B} \CBM(\mbbU)
  \,,\, 1 + \frac{C_\Pi M_B}{\gamma_B}\bigg\}
  \inf_{w_n\in \mbbU_n}   \norm{u-w_n}_\mbbU\,.
\end{alignat*}
\end{corollary}
\begin{proof}
If $r_m=0$, Eq.~\eqref{eq:Pn(u)_bound2} in the proof of Proposition~\ref{prop:Pbound} implies the simpler bound:
$$
\|P_n(u)\|_\mbbU\leq {C_\Pi\over \gamma_B}M_B \|u\|_\mbbU\,.
$$
Combining~\eqref{eq:directProof2} with Corollary~\ref{col:NPGbound} and this bound, gives the desired result.
%can be verified that the result~\eqref{eq:upperboundI-P} in Lemma~\ref{lem:Pn} simplifies to
%
%\begin{alignat*}{2}
 %\norm{I-P_n} \le \min \bigg\{ \frac{C_\Pi M_B}{\gamma_B} \CBM(\mbbU)
%  \,,\, 1 + \frac{C_\Pi M_B}{\gamma_B}\bigg\}
%  \norm{P_n}\,.
%\end{alignat*}
%
%Substituting this in~\eqref{eq:thm:discrete_apriori:bound} gives the proof.
%
\end{proof}
One particular situation for which~$r_m = 0$ occurs, is when discrete dimensions are matched: $\dim\mbbV_m=\dim\mbbU_n$. In that case, one recovers actually the standard Petrov--Galerkin method; see Section~\ref{sec:connectStdPG} for an elaboration on this connection. 
% Sufficient conditions to have $r_m=0$ are to satisfy the \emph{Fortin condition} (Assumption~\ref{assumpt:Fortin}) together with the 
%
%On the other hand, the same simpler estimate is obtained in the Hilbert space setting
%since in this case $\CAO(\mbbV)=0$ (see Remark~\ref{rem:LambdaV}).
%

%---------------------------------
%
%\begin{remark}
%\label{rem:bypass_of(iv)}
%To prove the last step in (S1) without requierement (iv) of Lemma~\ref{lem:I-P} one only needs to show that 
%$$
%\left\|{\alpha\over\beta}P(x)\right\|_\mbbX\leq \|P\|\|\tilde x\|_\mbbX\,.
%$$ 
%This is exactly the situation when $P$ is the nonlinear map defined in the proof of Proposition~\ref{prop:ba_apriori_bound}, where $\|P\|=1$. Indeed, if $x_0$ is a best approximation of $x$, then $P(x)=x-x_0$ and $x-P(x)=x_0$. Thus,
%$$
%\left\|{\alpha\over\beta}P(x)\right\|_\mbbX={\alpha\over\beta}\|x-x_0\|_\mbbX\leq {\alpha\over\beta}\left\|x-x_0+{\beta^2\over\alpha^2}x_0\right\|_\mbbX
%%=\left\|{\alpha\over\beta}(x-x_0)+{\beta\over\alpha}x_0\right\|_\mbbX
%=\|\tilde x\|_\mbbX=\|P\| \|\tilde x\|_\mbbX\,.
%$$
%\end{remark}

\subsection{Proof of Proposition~\ref{prop:Pbound}}
\label{sec:proofPbound}
% \begin{proof}% [of Prop.~\ref{prop:Pbound}]
In this section we proof Proposition~\ref{prop:Pbound}. We proceed item by item.
\begin{enumerate}[(i)]
\item Take $u\in\mbbU$ and plug $u_n=P_n(u)$ in the right-hand side of~\eqref{eq:mixed_discrete_a2}. Then the unique solution of the mixed system~\eqref{eq:mixed_discrete2} will be $(0,u_n)$. Therefore $P_n(P_n(u))= P_n(u_n) = u_n$. The fact that $P_n\neq 0$ and $P_n\neq I$ is easy to verify whenever $\mbbU_n\neq \{0\}$ and $\mbbU_n\neq \mbbU$.
\item The result follows by multiplying both equations of the mixed system~\eqref{eq:mixed_discrete2} by $\lambda\in\mathbb R$ and using the homogeneity of the duality map
(see Proposition~\ref{prop:duality}).
\item Consider any~$u\in \mbbU$ and let $(r_m,u_n)\in \mbbV_m\times \mbbU_n$ denote the solution to~\eqref{eq:mixed_discrete2}. 
%Note that~$u_n=P_n(u)$ and $r_m=R_m(u)$.
%
Observe that 
\begin{equation}
\label{eq:Pn(u)_bounddd}
\|P_n(u)\|_\mbbU=\|u_n\|_\mbbU\le \frac{1}{\gamma_B}\sup_{v\in \mbbV}\frac{\dual{B u_n,v}_{\mbbV^*,\mbbV} }{\norm{v}_\mbbV}
  \le \frac{C_\Pi}{\gamma_B}\sup_{v\in \mbbV}\frac{\dual{B u_n,\Pi v}_{\mbbV^*,\mbbV} }{\norm{
  \Pi v}_\mbbV}\,.
\end{equation}
Let $y_m=I_m J^{-1}_{\mbbV_m}(I^*_mBu_n)$ and note that $y_m\in\mbbV_m \subset\mbbV$ is the supremizer of the last expression in~\eqref{eq:Pn(u)_bounddd}. Hence, using~\eqref{eq:mixed_discrete_a2} we get
\begin{alignat}{2}
%{\gamma_B\over C_\Pi}
\notag
\|P_n(u)\|_\mbbU 
&\leq
\frac{C_\Pi}{\gamma_B}
%\sup_{v\in \mbbV}\frac{\dual{B u_n,\Pi v}_{\mbbV^*,\mbbV} }{\norm{\Pi v}_\mbbV}=
  \frac{\dual{B u_n,y_m}_{\mbbV^*,\mbbV} }{\norm{
  y_m}_\mbbV}
\\
\label{eq:Pn(u)_bound2}
& =\frac{C_\Pi}{\gamma_B}
\left(
  {\dual{Bu,y_m}_{\mbbV^*,\mbbV}\over\|y_m\|_\mbbV}
  -{\dual{J_\mbbV(r_m),y_m}_{\mbbV^*,\mbbV}\over\|r_m\|_\mbbV\|y_m\|_\mbbV}\|r_m\|_\mbbV
  \right)
  \,.
\end{alignat}
The first term in the parentheses above is clearly bounded by 
$\norm{B u}_{\mbbV^*}$. To bound the second term first observe that
\begin{alignat*}{2}
\dual{J_\mbbV(y_m),r_m}_{\mbbV^*,\mbbV}
% =\dual{J_{\mbbV_m}(y_m),r_m}_{\mbbV_m^*,\mbbV_m}
= \dual{Bu_n,r_m}_{\mbbV^*,\mbbV}=0\,,
\end{alignat*}
where we used~\eqref{eq:mixed_discrete_b2}. Thus,
$(r_m,y_m)\in\mathcal O_\mbbV$ (see Lemma~\ref{lem:Lambda_properties} (i)) which implies that the second term is bounded by $\CAO(\mbbV)\|r_m\|_\mbbV\,$.  Using~\eqref{eq:mixed_discrete2}, note that
$$
\|r_m\|_\mbbV={\dual{J_\mbbV(r_m),r_m}_{\mbbV^*,\mbbV}
\over\|r_m\|_\mbbV}={\dual{Bu-Bu_n,r_m}_{\mbbV^*,\mbbV}
\over\|r_m\|_\mbbV}={\dual{Bu,r_m}_{\mbbV^*,\mbbV}
\over\|r_m\|_\mbbV}\leq \|Bu\|_{\mbbV^*}\,.
$$ 
In conclusion, 
\begin{alignat}{2}
\label{eq:Pbound2}
\|P_n(u)\|_\mbbU \leq 
\frac{C_\Pi}{\gamma_B}\big(1+\CAO(\mbbV)\big)\|Bu\|_{\mbbV^*}\,,
\qquad\forall u\in\mbbU
\end{alignat}
and we get the desired result upon using~$\norm{Bu}_{\mbbV^*} \le M_B \norm{u}_\mbbU$.
\par 
We note that an alternative proof can be given based on the second a priori bound for the best approximation (Proposition~\ref{prop:improvedApriori})
and Lemma~\ref{lem:Lambda_properties}(ii). 
\item Let $(r_m,u_n)$ be the solution of the mixed system~\eqref{eq:mixed_discrete2} and for some $\widetilde w\in\mbbU$, let $\widetilde w_n=P_n(\widetilde w) \in \mbbU_n$. By subtracting $\dual{B\widetilde w_n,v_m}_{\mbbV^*,\mbbV}$ on both sides of the identity in~\eqref{eq:mixed_discrete_a2}, we get that $(r_m,u_n-\widetilde w_n)$ is the unique solution of~\eqref{eq:mixed_discrete2} with right-hand side ${\dual{B(u-\widetilde w_n),v_m}_{\mbbV^*,\mbbV}}$. Therefore ${P(u-\widetilde w_n)=u_n-\widetilde w_n}$.
\item Statement~(v) follows from statements~(ii) and~(iv), indeed, for any $\eta\in\mbbR$, 
\begin{alignat}{2}
P_n\Big(P_n(u)+\eta\,\big(u-P_n(u)\big)\Big) 
&= P_n\Big(\eta u +P_n\big((1-\eta)u\big)\Big)
\tag{by~(ii)}\\
&=  P_n(\eta u) + P_n\big((1-\eta)u\big)\tag{by~(iv)}\\
&=  P_n(u)\,.\tag{by~(ii)}
\end{alignat}
% Hence, $\|P_n(u)\|_\mbbU\leq \|P_n\| \|P_n(u)+\eta(u-P_n(u))\|_\mbbU$, as we wanted to prove.
\end{enumerate}
~\hfill%
$\ensuremath{_\blacksquare}$

%-----------------------------------------------------------------------------%
%
%=============================================================================%
\section{Connection to other theories} 
% Connection to other quasi-optimality analyses}
\label{sec:connections}
In this last section, we elaborate on how the presented quasi-optimality analysis in Section~\ref{sec:practical} generalizes existing theories for other methods. In doing so, we collect some of our earlier observations and provide a coherent summary of of the connections.
\par
Figure~\ref{fig:connections} presents a schematic hierarchy with the connections among the methods and the constants~$C$ in their respective quasi-optimality bound:
\begin{alignat}{2}
\label{eq:conn:qopt}
  \norm{u-u_n}_\mbbU \le C \inf_{w_n \in \mbbU_n} \norm{u-w_n}_\mbbU\,.
\end{alignat}
At the top of the figure is the inexact residual minimization (\textsf{iRM}) method, or equivalently, the inexact nonlinear Petrov--Galerkin (\textsf{iNPG}) or monotone mixed method (\textsf{MMM}). By considering certain special cases, quasi-optimality constants are recovered for Petrov-Galerkin (\textsf{PG}) methods, \emph{exact} residual minimization (\textsf{RM}), and inexact residual minimization in \emph{Hilbert} spaces (\textsf{iRM-H}), which includes the DPG method. Naturally, in these connections, the conditions of Theorem~\ref{thm:discrete} (discrete well-posedness of the inexact method) are assumed to hold.
%
% It is indicated how a result can be obtained through a suitable restriction of a preceding result. 
% follow .  in the . while the results for the other methods follow 
%
% We focus on three branches of methods important cases: the standard Petrov-Galerkin discretization, exact residual minimization, and the inexact method in Hilbert spaces. 
%   
%-----------------------------------------------------------------------------%
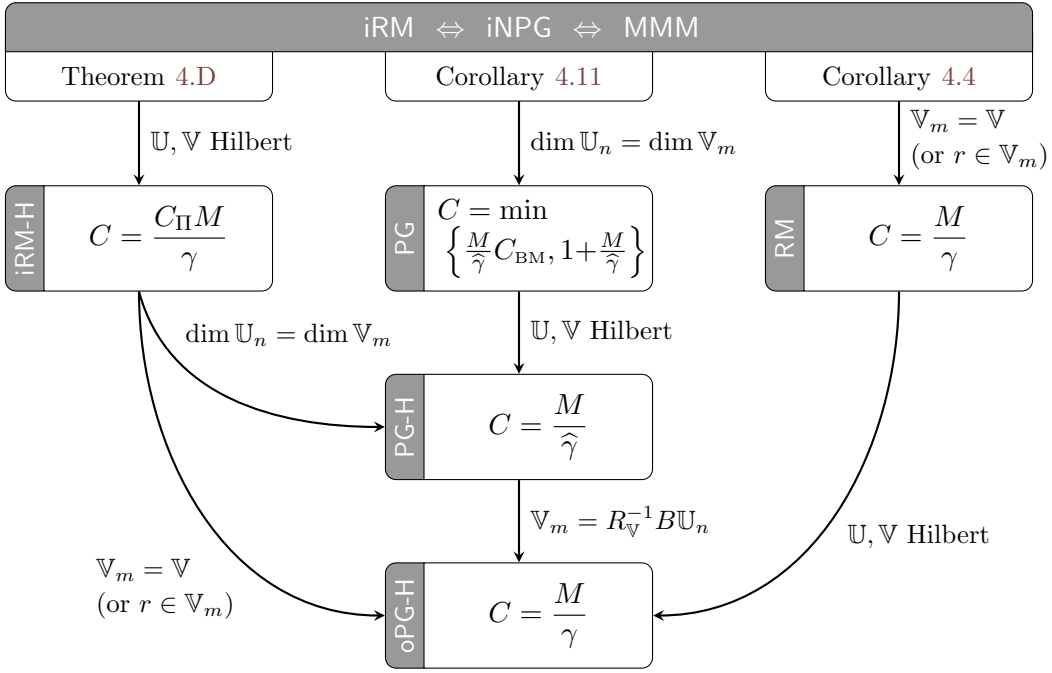
\begin{figure}[!t]
\newcommand{\labelFill}{black!40!white}
\newcommand{\mthd}[1]{\textsf{\textcolor{white}{#1}}}
%\newcommand{\mthd}[1]{\textsf{\textcolor{white}{\contour[4]{black}{#1}}}}
%\contourlength{0.025em}
%
\newcommand{\CboxLabel}[1]{\rotatebox{90}{{\parbox[t]{\widthof{\footnotesize\mthd{oPG-H}}}{\centering\mbox{\footnotesize{\mthd{#1}}}}}}}
\newcommand{\CboxText}[1]{{{\begin{minipage}[c][\widthof{\footnotesize\textsf{oPG-H}}]{3cm}\centering\small #1 \end{minipage}}}}
\newcommand{\Cbox}[2]{\nodepart[text width=0.5cm]{one}\CboxLabel{#1}
                      \nodepart[text width=3cm]{two}\CboxText{#2}}
\newcommand{\ArrowLabel}[1]{\footnotesize #1}
\newcommand{\BottomLabel}[1]{\footnotesize\textrm{#1}}
\tikzset{BottomLabel/.style = {anchor=mid, text height = 3ex, text width = 3.5cm, inner xsep=0, align=center}}
\tikzset{Cbox/.style = {rectangle split, rectangle split horizontal, rounded corners, rectangle split parts=2, draw, rectangle split part fill={\labelFill,white}, inner sep=0, align=center}}
\tikzset{FillRoundedNorth fill/.style={append after command={
   \pgfextra
        \draw[sharp corners, fill=#1]% 
    (\tikzlastnode.west)% 
    [rounded corners] |- (\tikzlastnode.north)% 
    [rounded corners] -| (\tikzlastnode.east)% 
    [rounded corners=0pt] |- (\tikzlastnode.south)% 
    [rounded corners=0pt] -| (\tikzlastnode.west);
   \endpgfextra}}}
\tikzset{FillRoundedSouth fill/.style={append after command={
   \pgfextra
        \draw[sharp corners, fill=#1]% 
    (\tikzlastnode.west)% 
    [rounded corners=0pt] |- (\tikzlastnode.north)% 
    [rounded corners=0pt] -| (\tikzlastnode.east)% 
    [rounded corners] |- (\tikzlastnode.south)% 
    [rounded corners] -| (\tikzlastnode.west);
   \endpgfextra}}}
\centering
\begin{tikzpicture}[auto, scale=1, >=stealth]
  \draw (10,0.125) node(COR) [FillRoundedSouth fill=white, BottomLabel] 
        {\BottomLabel{Corollary~\ref{cor:aPrioriErrorEst}}};
  \draw ( 5,0.125) node(REM) [FillRoundedSouth fill=white, BottomLabel] 
        {\BottomLabel{Corollary~\ref{cor:rm=0}}};
  \draw ( 0,0.125) node(THM) [FillRoundedSouth fill=white, BottomLabel] 
        {\BottomLabel{\phantom{y}Theorem~\ref{thm:discrete_apriori}\phantom{y}}};
  \draw (5,.8) node(iRM) 
        [FillRoundedNorth fill=\labelFill, 
         text width=13.5cm, minimum height = 3.5ex, text centered,
         inner sep=0]
        {\small \makebox[4em]{\hfill \mthd{iRM}} ~\mthd{$\Leftrightarrow$}~ 
         \mthd{iNPG} ~\mthd{$\Leftrightarrow$}~ \makebox[4em]{\mthd{MMM} \hfill}};
  \draw (0,-2) node(iRMH) [Cbox] 
        {\Cbox{iRM-H}{$\ds{C = \frac{C_\Pi M}{\gamma}}$}};
  \draw (5,-2) node(PG) [Cbox] 
        {\Cbox{PG}{$C = \min$\hspace*{3.125em}
          \\ \hspace*{.125em}
          $\Big\{  \tfrac{M}{\widehat{\gamma} }\CBM
                 ,1{+}\tfrac{M}{\widehat{\gamma}} 
           \Big\}$}%
        };
  \draw (5,-4.5) node(PGH) [Cbox] 
        {\Cbox{PG-H}{$\ds{C = \frac{M}{\widehat{\gamma}}}$}};
  \draw (5,-7) node(oPGH) [Cbox] 
        {\Cbox{oPG-H}{$\ds{C = \frac{M}{\gamma}}$}};
  \draw (10,-2) node(RM) [Cbox] 
        {\Cbox{RM}{$\ds{C = \frac{M}{\gamma}}$}};
  \draw [->,thick] (THM) to node 
        {\ArrowLabel{$\mbbU,\mbbV$ Hilbert}} (iRMH);
  \draw [->,thick] (COR) to node [near start]
        {\ArrowLabel{\parbox[t]{\widthof{$\mbbV_m = \mbbV$}}{\flushleft 
                     \mbox{$\mbbV_m = \mbbV$}
                     \\ \mbox{(or $r\in \mbbV_m$)$\negqquad$}}%
                    }} (RM);
  \draw [->,thick] (REM) to node 
        {\ArrowLabel{$\dim\mbbU_n = \dim\mbbV_m$}} (PG);
  \draw [->,thick] (iRMH.south) to [out=285, in=180] node [near start] 
        {\ArrowLabel{$\dim\mbbU_n = \dim\mbbV_m$}} (PGH.west);
  \draw [->,thick] (PG)  to node 
        {\ArrowLabel{$\mbbU,\mbbV$ Hilbert}} (PGH);
  \draw [->,thick] (PGH)  to node 
        {\ArrowLabel{$\mbbV_m = R_{\mbbV}^{-1}B\mbbU_n$ }} (oPGH);
  \draw [->,thick] (iRMH.south) to [out=270, in=180] node [swap] 
        {\ArrowLabel{\parbox[t]{\widthof{$\mbbV_m = \mbbV$}}{\flushleft 
                     \mbox{$\mbbV_m = \mbbV$}
                     \\ \mbox{(or $r\in \mbbV_m$)$\negqquad$}}%
                    }} (oPGH.west);
  \draw [->,thick] (RM)  to [out=270, in=0] node 
        {\ArrowLabel{$\mbbU,\mbbV$ Hilbert}} (oPGH);
\end{tikzpicture}
\caption{Hierarchy of discretization methods and their quasi-optimality result; see Sections~\ref{sec:connectStdPG}--\ref{sec:connectHilbert} for a detailed explanation. To lighten the notation, $\gamma \equiv \gamma_B$, $M \equiv M_B$ and $\CBM \equiv \CBM(\mbbU)$.}
\label{fig:connections}
\end{figure}
%-----------------------------------------------------------------------------%
%
%-----------------------------------------------------------------------------%
\subsection{Petrov--Galerkin methods}
\label{sec:connectStdPG}
The standard Petrov--Galerkin method (\textsf{PG} in Figure~\ref{fig:connections}) is obtained when $\dim \mbbV_m=\dim \mbbU_n$. Indeed, under this stipulation, the (Fortin) Assumption~\ref{assumpt:Fortin} implies the well-known \emph{discrete inf-sup condition}
%~\eqref{eq:discrete_infsup} 
(see, e.g., \cite{ErnGueBOOK2004,ErnGueCR2016})
with \emph{discrete inf-sup constant} $\widehat{\gamma}=\gamma_B/C_\Pi\,$, where $\gamma_B$ is the bounded-below constant of~$B$ and $C_\Pi$ the boundedness constant of the Fortin operator~$\Pi$. Therefore, $(r_m,w_n)\mapsto \dual{B^* r_m,w_n}$ corresponds to an invertible square system, so that~\eqref{eq:mixed_discrete_b} implies $r_m=0$, while~\eqref{eq:mixed_discrete_a} reduces to the standard (linear) Petrov--Galerkin form:
\begin{alignat}{2}
\label{eq:connect:PG}
\dual{ Bu_n,v_m }_{\mbbV^*,\mbbV}=\dual{f,v_m }_{\mbbV^*,\mbbV},\qquad\forall v_m\in \mbbV_m\,.
\end{alignat}
\par
The quasi-optimality result of Corollary~\ref{cor:rm=0} applies in this situation (since~$r_m = 0$), resulting in the constant
\begin{alignat*}{2}
  C = \min \bigg\{ {M_B\over\widehat{\gamma}} \CBM(\mbbU)\,,\,1+{M_B\over\widehat{\gamma}}\bigg\}\,.
\end{alignat*}
This coincides with the recent result obtained by Stern~\cite{SteNM2015}. Historically, the first quasi-optimality analysis for the PG~method was carried out in the pioneering work of Babu\v{s}ka~\cite{BabNM1971} who obtained the classical result~$C = 1 + M_B / \widehat{\gamma}$.
\par
Furthermore, if $\mbbU$ is a Hilbert space, then $\CBM(\mbbU)=1$. Hence, when $\mbbU$ and $\mbbV$ are Hilbert spaces, the constant reduces to $C = M_B/\widehat{\gamma}$, which is the established result for PG methods in Hilbert spaces (\textsf{PG-H}); see Xu~\& Zikatanov~\cite{XuZikNM2003}. Furthermore, when the discrete test space~$\mbbV_m$ equals~$R_V^{-1} B \mbbU_n$, one obtains the \emph{optimal} PG method in Hilbert spaces (\textsf{oPG-H}). In that case~$\widehat{\gamma} =\gamma$, and one recovers the result~$C = M_B/\gamma_B$ as obtained by Demkowicz~\& Gopalakrishnan~\cite{DemGopNMPDE2011}. We note that $M_B/\gamma_B$ can be made equal to~$1$, by suitably re-norming~$\mbbU$ or~$\mbbV$~\cite{ZitMugDemGopParCalJCP2011, DahHuaSchWelSINUM2012}.
%
%
%Historically, the first quasi-optimality analysis for the Petrov--Galerkin discretization was obtained in the pioneering work of Babu\v{s}ka~\cite{BabNM1971}. To our knowledge, its last improvement can be found in the work of Stern~\cite{SteNM2015}, who obtained the quasi-optimality constant~$C = \min\big\{ \frac{M_B}{\gamma_B} \CBM(\mbbU)\,,\, 1+\frac{M_B}{\gamma_B} \big\}$\,.
%
%\par
%
%Assuming that the \emph{Fortin condition} (Assumption~\ref{assumpt:Fortin}) holds true, the discrete stability of the problem is implied (cf.~\cite{BofBreForBOOK2013}), in which case, the discrete $\inf$-$\sup$ constant  is given by $\widehat{\gamma}=\gamma_B/C_\Pi\,$, where $\gamma_B$ is the bounded-below constant of the continuous operator $B$ and $C_\Pi$ is the continuity constant of the Fortin operator $\Pi$ (cf.~\eqref{eq:Fortin_a}). Moreover,
%since the discrete primal and dual problems become well-posed, from equation~\eqref{eq:mixed_discrete_b} we recover $r_m=0$, and thus from equation~\eqref{eq:mixed_discrete_a} we get the standard linear Petrov-Galerkin form
%$$
%\left<Bu_n,v_m\right>_{\mbbV^*,\mbbV}=\left<f,v_m\right>_{\mbbV^*,\mbbV},\qquad\forall v_m\in \mbbV_m.
%$$
%
%%%%%%%
%
%-----------------------------------------------------------------------------%
\subsection{Residual minimization}
\label{sec:conn:resmin}
The \emph{exact} residual minimization method~\eqref{eq:introResMin} (\textsf{RM} in Figure~\ref{fig:connections}) is obviously recovered when $\mbbV_m=\mbbV$ in~\eqref{eq:introInexactResMin}. 
%  i.e., the ``discrete'' space~$\mbbV_m$ is actually the infinite dimensional space~$\mbbV$. 
%
% \par
%
To demonstrate that the corresponding quasi-optimality result is also recovered, % for the exact method from the inexact one, 
note that when $\mbbV_m=\mbbV$, the (Fortin) Assumption~\ref{assumpt:Fortin} is straightforwardly satisfied by taking $\Pi=I$ (the identity). Then, $C_\Pi=1$ and $\osc(f)$ as defined in~\eqref{eq:osc} vanishes, which reduces the a~priori error estimate in Corollary~\ref{cor:aPrioriErrorEst} to~\eqref{eq:conn:qopt} with~$C = M_B / \gamma_B$.
%
%$$
%\|u-u_n\|_\mbbU\leq {M_B\over \gamma_B} \inf_{w_n\in\mbbU_n}\|u-w_n\|_\mbbU\,\,.
%$$
This result indeed coincides with the one for exact residual minimization (see~\eqref{eq:cea}) as originally obtained by Guermond~\cite{GueSINUM2004}.
%where we have used~\eqref{} and the fact that $\hbox{osc}(f)= 0$ (see~\eqref{}).
%Observe that the above estimate is exactly the error estimate of residual minimization (cf.~\eqref{eq:cea} or Guermond~\cite[Theorem~2.1]{GueSINUM2004})
%
Furthermore, the a~priori bound~\eqref{eq:alternative_apriori_bound} can also be recovered, upon substituting $C_\Pi = 1$ in~\eqref{eq:Pbound2} and taking into account that $\CAO(\mbbV^*)=\CAO(\mbbV)$ (see Lemma~\ref{lem:Lambda_properties}).
%also reduces to 
%of Proposition~\ref{prop:Pbound} is related now with the alternative apriori bound~\eqref{eq:alternative_apriori_bound} in residual minimization (cf. Remark~\ref{rem:sharpEst}), taken into account that $\Lambda_{\mbbV^*}=\Lambda_{\mbbV}$ (by Lemma~\ref{lem:Lambda_properties}) and $\|f\|_{\mbbV^*}\leq M_B\|u\|_\mbbU$.
%
\par
As an alternative to the case~$\mbbV_m = \mbbV$, the exact residual minimizer $u_n=\arg \min_{w_n\in\mbbU_n}\|f-Bw_n\|_{\mbbV^*}$ is also obtained when the (continuous) residual representer happens to be in~$\mbbV_m$, i.e., 
\begin{alignat}{2} 
\label{eq:conn:r}
  r := J_\mbbV^{-1}(f-B u_n) \in\mbbV_m\,.
\end{alignat}
See Proposition~\ref{prop:r_in_V} for the equivalence in this special situation. 
%Needless to say, when~\eqref{eq:conn:r} holds
%from the fully-discrete formulation~\eqref{eq:mixed_discrete} 
%
\par
Interestingly, when $\mbbU$ and $\mbbV$ are Hilbert spaces, the quasi-optimality constant~$C = M_B / \gamma_B$ for exact residual minimization remains the same as in the Banach-space case. This is consistent with the fact that the resulting method is equivalent to the optimal Petrov--Galerkin method in Hilbert spaces (\textsf{oPG-H}), as discussed in Section~\ref{sec:MainResult}; see~\eqref{eq:introOptPG}. 
% This is consistent with the result for the , which is equivalent.
%
%-----------------------------------------------------------------------------%
\subsection{Inexact method in Hilbert spaces}
\label{sec:connectHilbert}
The most important fact of the inexact residual minimization method in \emph{Hilbert} spaces (\mbox{\textsf{iRM-H} in Figure~\ref{fig:connections}}) is that it is a \emph{linear} method. Indeed, in this case the duality map $J_\mbbV$ is the Riesz map $R_\mbbV$, hence 
\begin{alignat*}{2}
\dual{J_\mbbV(r_m),v_m}_{\mbbV^*,\mbbV} 
 = \dual{R_\mbbV r_m ,v_m}_{\mbbV^*,\mbbV} 
 = \innerprod{r_m,v_m}_{\mbbV}\,,
\end{alignat*}
where $(\cdot,\cdot)_\mbbV$ denotes the inner product in~$\mbbV$. Therefore, the monotone mixed method~\eqref{eq:mixed_discrete} reduces to:
\begin{subequations}
\begin{empheq}[left=\left\{\;,right=\right.,box=]{alignat*=3} 
 \innerprod{ r_m ,v_m }_{\mbbV}
 +\<Bu_n,v_m\>_{\mbbV^{\ast},\mbbV} 
 &=\<f,v_m\>_{\mbbV^{\ast},\mbbV} \,,
 & \quad & \forall v_m\in \mbbV_m\,,
\\ 
\<B^*r_m,w_n\>_{\mbbU^{\ast},\mbbU} 
 &=0\,, 
 & \quad & \forall w_n\in \mbbU_n\,,
\end{empheq}
\end{subequations}
which is equal to the mixed form of the DPG method~\cite{DemGopBOOK-CH2014} as well as the inexact optimal Petrov--Galerkin method in Hilbert spaces~\cite{CohDahWelM2AN2012}. 
%
%To demonstrate the the quasi-optimality result 
% Moreover, as it was point out in Remark~\ref{rem:Hilbert_case}, 
In the Hilbert case, $\CBM(\mbbU)=1$ and $\CAO(\mbbV)=0$, so that the quasi-optimality constant in Theorem~\ref{thm:discrete_apriori} reduces to~$C = C_\Pi M_B / \gamma_B$. This coincides with the Hilbert-space result due to Gopalakrishnan and Qiu~\cite{GopQiuMOC2014}. 
\par
Furthermore, if $\dim \mbbV_m = \dim \mbbU_n$, by the same reasoning as in Section~\ref{sec:connectStdPG}, one obtains a Petrov--Galerkin method in Hilbert spaces (\textsf{PG-H}). The quasi-optimality constant reduces to $C = M_B / \widehat{\gamma}$, since the discrete inf-sup constant~$\widehat{\gamma}$ can be taken as~$C_\Pi / \gamma$. On the other hand, if $\mbbV_m = \mbbV$ (or \eqref{eq:conn:r} is valid), by the same reasoning as in Section~\ref{sec:conn:resmin}, one obtains \emph{exact} residual minimization in Hilbert spaces, which in turn is equivalent to the optimal Petrov--Galerkin method in Hilbert spaces (\textsf{oPG-H}). In that case~$C_\Pi = 1$, so that the quasi-optimality constant reduces to the expected result~$C = M_B/\gamma_B$.
%
% (cf.~\cite{DemGopBOOK-CH2014}).
%Furthermore, as expected, the a~priori bound~\eqref{eq:Pbound} of Proposition~\ref{prop:Pbound} reduces to the apriori bound in Remark~\ref{rem:r_m=0}. 
%
% On another hand, is quite remarkable that the aposteriori error analysis of the \emph{Inexact Monotone Mixed Method} is not affected by the geometry of the Banach spaces involved. As we can see, the aposteriori error estimates~\eqref{eq:aposteriori},~\eqref{eq:osc}~and~\eqref{eq:aposteriori2} in Theorem~\ref{thm:aposteriori} are exactly the same of the ones obtained by Carstensen et al.~\cite{CarDemGopSINUM2014} (cf. Remark~\ref{rem:aposteriori_Hilbert}).
%
%

%\input{Sec6.tex}
%\input{Sec7.tex}
%-----------------------------------------------------------------------------%
%\clearpage
\section*{Acknowledgements}
\addcontentsline{toc}{section}{Acknowledgements}
IM and KvdZ are grateful to Leszek Demkowicz for his early encouragement to investigate a Banach-space theory of DPG, and to Jay Gopalakrishnan for insightful conversations. KvdZ is also thankful to Michael Holst and Sarah Pollock for initial discussions on the topic, and to Weifeng Qiu, Paul Houston and Sarah Roggendorf for additional discussions. 
\par
The work by IM was done in the framework of Chilean FONDECYT research project~\#1160774. IM was also partially supported by European Union’s Horizon 2020 research and innovation program under the Marie Sklodowska-Curie grant agreements No~644202 and No~777778. KvdZ is grateful to the support provided by the Royal Society International Exchanges Scheme / Kan Tong Po Visiting Fellowship Programme, and by the above-mentioned FONDECYT project.
%-----------------------------------------------------------------------------%
\small
\bibliography{BibFile_OAP}

\begin{thebibliography}{10}

\bibitem{AdaFouBOOK2003}
{\sc R.~A. Adams and J.~F. Fournier}, {\em Sobolev Spaces}, vol.~140 of Pure
  and Applied Mathematics, Academic Press, Oxford, 2nd~ed., 2003.

\bibitem{AspBOAMS1967}
{\sc E.~Asplund}, {\em Positivity of duality mappings}, Bull. Amer. Math. Soc.,
  73 (1967), pp.~200--203.

\bibitem{BabNM1971}
{\sc I.~Babu{\v{s}}ka}, {\em Error-bounds for finite element method}, Numer.
  Math., 16 (1971), pp.~322--333.

\bibitem{BanBOOK2009}
{\sc S.~Banach}, {\em Theory of Linear Operations}, Dover Books on Mathematics,
  Dover Publications, 2009.
\newblock Reprint of the Elsevier Science Publishers, 1987 edition. Translation
  of original French version~1931.

\bibitem{BocGunBOOK2009}
{\sc P.~B. Bochev and M.~D. Gunzburger}, {\em Least-{S}quares {F}inite
  {E}lement {M}ethods}, vol.~166 of Applied Mathematical Sciences, Springer
  Science \& Business Media, 2009.

\bibitem{BofBreForBOOK2013}
{\sc D.~Boffi, F.~Brezzi, and M.~Fortin}, {\em Mixed Finite Element Methods and
  Applications}, vol.~44 of Springer Series in Computational Mathematics,
  Springer, Berlin, 2013.

\bibitem{BraBOOK1986}
{\sc D.~Braess}, {\em Nonlinear Approximation Theory}, vol.~7 of Springer
  Series in Computational Mathematics, Springer, Berlin, 1986.

\bibitem{BreBOOK2011}
{\sc H.~Brezis}, {\em Functional Analysis, Sobolev Spaces and Partial
  Differential Equations}, Universitext, Springer, New York, 2011.

\bibitem{BroSteCAMWA2014}
{\sc D.~Broersen and R.~Stevenson}, {\em A robust {Petrov--Galerkin}
  discretisation of convection--diffusion equations}, Comput. Math. Appl., 68
  (2014), pp.~1605--1618.

\bibitem{CarDemGopSINUM2014}
{\sc C.~Carstensen, L.~Demkowicz, and J.~Gopalakrishnan}, {\em A posteriori
  error control for {DPG} methods}, SIAM J. Numer. Anal., 52 (2014),
  pp.~1335--1353.

\bibitem{ChaEvaQiuCAMWA2014}
{\sc J.~Chan, J.~A. Evans, and W.~Qiu}, {\em A dual {P}etrov--{G}alerkin finite
  element method for the convection--diffusion equation}, Comput. Math. Appl.,
  68 (2014), pp.~1513--1529.

\bibitem{ChiBOOK2009}
{\sc C.~Chidume}, {\em Geometric Properties of Banach Spaces and Nonlinear
  Iterations}, vol.~1965 of Lecture Notes in Mathematics, Springer, London,
  2009.

\bibitem{CiaBOOK2013}
{\sc P.~G. Ciarlet}, {\em Linear and Nonlinear Functional Analysis with
  Applications}, SIAM, Philadelphia, 2013.

\bibitem{CioBOOK1990}
{\sc I.~Cioranescu}, {\em Geometry of {Banach} Spaces, Duality Mappings and
  Nonlinear Problems}, vol.~62 of Mathematics and Its Applications, Kluwer
  Academic Publishers, Dordrecht, The Netherlands, 1990.

\bibitem{CohDahWelM2AN2012}
{\sc A.~Cohen, W.~Dahmen, and G.~Welper}, {\em Adaptivity and variational
  stabilization for convection-diffusion equations}, M2AN Math. Model. Numer.
  Anal., 46 (2012), pp.~1247--1273.

\bibitem{DahHuaSchWelSINUM2012}
{\sc W.~Dahmen, C.~Huang, C.~Schwab, and G.~Welper}, {\em Adaptive
  {P}etrov--{G}alerkin methods for first order transport equations}, SIAM J.
  Numer. Anal., 50 (2012), pp.~2420--2445.

\bibitem{DeiBOOK1985}
{\sc K.~Deimling}, {\em Nonlinear Functional Analysis}, Springer, Berlin, 1985.

\bibitem{DemGopCMAME2010}
{\sc L.~Demkowicz and J.~Gopalakrishnan}, {\em A class of discontinuous
  {P}etrov--{G}alerkin methods. {P}art {I}. {T}he transport equation}, Comput.
  Methods Appl. Mech. Engrg., 199 (2010), pp.~1558--1572.

\bibitem{DemGopNMPDE2011}
\leavevmode\vrule height 2pt depth -1.6pt width 23pt, {\em A class of
  discontinuous {P}etrov--{G}alerkin methods. {II}. {O}ptimal test functions},
  Numer. Methods Partial Differential Equations, 27 (2011), pp.~70--105.

\bibitem{DemGopBOOK-CH2014}
\leavevmode\vrule height 2pt depth -1.6pt width 23pt, {\em An overview of the
  discontinuous {Petrov Galerkin} method}, in Recent Developments in
  Discontinuous Galerkin Finite Element Methods for Partial Differential
  Equations: 2012 John H Barrett Memorial Lectures, X.~Feng, O.~Karakashian,
  and Y.~Xing, eds., vol.~157 of The IMA Volumes in Mathematics and its
  Applications, Springer, Cham, 2014, pp.~149--180.

\bibitem{DevLorBOOK1993}
{\sc R.~A. DeVore and G.~G. Lorentz}, {\em Constructive Approximation},
  vol.~303 of Grundlehren der Mathematischen Wissenschaften, Springer, 1993.

\bibitem{ErnGueBOOK2004}
{\sc A.~Ern and J.-L. Guermond}, {\em Theory and Practice of Finite Element
  Methods}, vol.~159 of Applied Mathematical Sciences, Springer-Verlag, New
  York, 2004.

\bibitem{ErnGueCR2016}
\leavevmode\vrule height 2pt depth -1.6pt width 23pt, {\em A converse to
  {F}ortin's {L}emma in {B}anach spaces}, C.~R. Math. Acad. Sci. Paris, 354
  (2016), pp.~1092--1095.

\bibitem{GopARXIV2014}
{\sc J.~Gopalakrishnan}, {\em Five lectures on {DPG} methods}.
\newblock arXiv:1306.0557v2 [math.NA], Aug 2014.

\bibitem{GopQiuMOC2014}
{\sc J.~Gopalakrishnan and W.~Qiu}, {\em An analysis of the practical {DPG}
  method}, Math. Comp., 83 (2014), pp.~537--552.

\bibitem{GueSINUM2004}
{\sc J.~L. Guermond}, {\em A finite element technique for solving first-order
  {PDEs} in {$L^p$}}, SIAM J. Numer. Anal., 42 (2004), pp.~714--737.

\bibitem{JohLinBOOK-CH2001}
{\sc W.~B. Johnson and J.~Lindenstrauss}, {\em Basic concepts in the geometry
  of {Banach} spaces}, in Handbook of the Geometry of Banach Spaces, W.~B.
  Johnson and J.~Lindenstrauss, eds., vol.~1, Elsevier Science B.~V., 2001,
  ch.~1, pp.~1--84.

\bibitem{KatNM1960}
{\sc T.~Kato}, {\em Estimation of iterated matrices with application to von
  {Neumann} condition}, Numer. Math., 2 (1960), pp.~22--29.

\bibitem{LioBOOK1969}
{\sc J.~Lions}, {\em Quelques M{\'{e}}thodes de R{\'{e}}solution des
  Probl{\`{e}}mes aux Limites Non Lin{\'{e}}aires}, {\'{E}}tudes
  Math{\'{e}}matiques, Dunod, 1969.

\bibitem{OdeDemBOOK2010}
{\sc J.~T. Oden and L.~F. Demkowicz}, {\em Applied Functional Analysis}, CRC
  Press, 2nd~ed., 2010.

\bibitem{PetJFA1970}
{\sc W.~V. Petryshyn}, {\em A characterization of strict convexity of {Banach}
  spaces and other uses of duality mappings}, J.~Funct. Anal., 6 (1970),
  pp.~282--291.

\bibitem{SinBOOK1970}
{\sc I.~Singer}, {\em Best Approximation in Normed Linear Spaces by Elements of
  Linear Subspaces}, vol.~171 of Die Grundlehren der mathematischen
  Wissenshaften, Springer, Berlin, 1970.

\bibitem{StaHolBOOK2011}
{\sc I.~Stakgold and M.~Holst}, {\em Green's Functions and Boundary Value
  Problems}, vol.~99 of Pure and Applied Mathematics, John Wiley \& Sons,
  Hoboken, New Jersey, 3$^{\text{rd}}$~ed., 2011.

\bibitem{SteNM2015}
{\sc A.~Stern}, {\em Banach space projections and {Petrov--Galerkin}
  estimates}, Numer. Math., 130 (2015), pp.~125--133.

\bibitem{SzyNALG2006}
{\sc D.~Szyld}, {\em The many proofs of an identity on the norm of oblique
  projections}, Numer. Algorithms, 42 (2006), pp.~309--323.

\bibitem{WojBOOK1991}
{\sc P.~Wojtaszczyk}, {\em Banach Spaces for Analysts}, no.~25 in Cambridge
  studies for advanced mathematics, Cambridge University Press, Cambridge,
  1991.

\bibitem{XuZikNM2003}
{\sc J.~Xu and L.~Zikatanov}, {\em Some observations on {Babu{\v{s}}ka} and
  {Brezzi} theories}, Numer. Math., 94 (2003), pp.~195--202.

\bibitem{ZeiBOOK1990b}
{\sc E.~Zeidler}, {\em Nonlinear Functional Analysis and its Applications,
  {II}/{B}: Nonlinear Monotone Operators}, Springer-Verlag, New York, 1990.

\bibitem{ZitMugDemGopParCalJCP2011}
{\sc J.~Zitelli, I.~Muga, L.~Demkowicz, J.~Gopalakrishnan, D.~Pardo, and V.~M.
  Calo}, {\em A class of discontinuous {P}etrov--{G}alerkin methods. {P}art
  {IV}: {T}he optimal test norm and time-harmonic wave propagation in {1D}},
  J.~Comput. Phys., 230 (2011), pp.~2406--2432.

\end{thebibliography}
\bibliographystyle{siam}
\end{document}
%-----------------------------------------------------------------------------%
%
%=============================================================================%